\documentclass[11pt]{amsart}
\usepackage{amsmath}
\usepackage{amssymb}
\usepackage{bm}
\usepackage{ esint }
\usepackage{graphicx}
\usepackage{verbatim}
\usepackage{enumitem}
\setlist[enumerate,1]{leftmargin=*}
\setlist[itemize,1]{leftmargin=*}
\usepackage{mathrsfs}
\usepackage[dvipsnames]{color}
\usepackage{todonotes}
\usepackage[pdftex,pdfstartview=FitH,pdfborderstyle={/S/B/W 1},%
colorlinks=true, linkcolor=blue, urlcolor=blue, citecolor=blue,%
pagebackref=true]{hyperref}
\usepackage[nobysame,alphabetic,initials,msc-links]{amsrefs}
\usepackage{comment}
\usepackage{mathtools}

\calclayout 

\newtheorem{Lemma}{Lemma}[section]
\newtheorem{Theorem}[Lemma]{Theorem}
\newtheorem{Proposition}[Lemma]{Proposition}
\newtheorem{Corollary}[Lemma]{Corollary}

\theoremstyle{definition}

\newtheorem{Definition}[Lemma]{Definition}

\newtheorem*{Standing}{Standing assumptions}

\numberwithin{equation}{section}

\def\C{\mathbb{C}}

\def\geq{\geqslant}
\def\leq{\leqslant}
\def\R{\mathbb{R}}

\def\Z{\mathbb{Z}}
\def\N{\mathbb{N}}

\def\R{\mathbb{R}}
\def\id{\mathrm{id}}

\def\Q{\mathbb{Q}}

\def\v{\mathbf{v}}

\def\0{\mathbf{0}}

\def\cF{\mathcal F}

\def\h{\mathfrak{h}}
\def\SO{\operatorname{SO}}
\def\GL{\operatorname{GL}}
\def\adj{\operatorname{adj}}

\def\epsilon{\varepsilon}

\def\ad{\operatorname{ad}}
\def\Ad{\operatorname{Ad}}

\def\diag{\operatorname{diag}}

\def\ker{\operatorname{ker}}
\def\SL{\operatorname{SL}}
\def\supp{\operatorname{supp}}

\def\sl{\mathfrak{sl}}
\def\PSL{\operatorname{PSL}}

\newcommand{\bea}{\begin{eqnarray}}
  \newcommand{\eea}{\end{eqnarray}}
  \newcommand{\beab}{\begin{eqnarray*}}
  \newcommand{\eeab}{\end{eqnarray*}}

  \newcommand{\be}{\begin{equation}}
  \newcommand{\ee}{\end{equation}}

\newcommand{\abs}[1]{\left| #1 \right|}

\newcommand{\norm}[1]{\abs{\abs{#1}}}

\begin{document}

\title[Directional measure rigidity]{Directional measure rigidity for time-changed first row unipotent actions}

\author[Davide Ravotti]{Davide Ravotti}
\address{D.R.: Laboratoire Paul Painlev\'{e}, Universit\'{e} de Lille, B\^{a}timent M3 Cit\'{e} Scientifique, 59655 Villeneuve-d'Ascq, France}
\email{davide.ravotti@gmail.com}

\author[Daren Wei]{Daren Wei}
\thanks{D.W. acknowledges support by NUS startup grants A-0009806-00-00 and A-0009806-01-00}
\address{D.W.: Department of Mathematics, National University of Singapore, 119076, Singapore}
\email{darenwei@nus.edu.sg}

\setlength{\marginparwidth }{2 cm}

\begin{abstract}
We study flows on finite volume homogeneous manifolds arising as directional restrictions of time changes of higher rank abelian unipotent actions. We establish Ratner-type rigidity results in this setting. For first row actions on finite volume quotients of $\SL_n(\R)$, we show that natural quantitative assumptions on the cocycle force the flow to be algebraic and its ergodic invariant measure to be homogeneous. These results extend to higher dimensional subactions and to a general abstract setting. We also show that, on compact quotients of $\SL_3(\R)$, the smooth time change is smoothly conjugate to a linear algebraic action by a leaf preserving diffeomorphism if and only if the cocycle of the directional flow has bounded deviations along one direction. We also obtain precise asymptotics for compact root increments. The principal thresholds in our results are sharp.

\end{abstract}

\maketitle

\section{Introduction}
Ratner's measure classification and orbit closure theorems place unipotent dynamics among the most rigid systems in smooth ergodic theory. On a finite-volume homogeneous space, every ergodic probability measure invariant under a unipotent flow is homogeneous, and every orbit closure is homogeneous as well \cites{ratner1990measure,ratner1991measure,ratner1991topological}. 

Time changes provide a natural test of how much of this rigidity survives beyond the homogeneous category. For a flow, a time change reparametrizes each orbit by a positive function. For an $\R^d$-action, it is described by a cocycle whose fiber maps reparametrize the entire orbit. The orbit foliation is preserved, but the homogeneous parametrization is lost; invariant measures and lower dimensional restrictions may change substantially.

The orbit foliation alone cannot retain much rigidity in higher dimension. Rudolph \cite{RudolphSmoothOrbitEquivalence79} proved that any two free ergodic probability preserving $\R^d$-actions, $d\geq 2$, are orbit equivalent by a measure preserving map whose restriction to almost every orbit is $C^{\infty}$. Consequently, a rigidity theorem for higher rank time changes cannot follow merely from the smooth geometry of their leaves. It must recover the homogeneous structure from additional geometric information carried by the action.

The model for such a result is Ratner's rigidity theorem for time changes of horocycle flows \cite{ratner1986rigidity}. Ratner imposed regularity in the compact rotation direction of $\SL_2(\R)$ and proved that an isomorphism between sufficiently regular time changes is algebraic up to a cohomological displacement along the horocycle orbits. Her theorem shows that, although time changes leave the homogeneous category, the transverse geometry of the ambient group can still force algebraic rigidity.

The cohomological and statistical theory of horocycle time changes was developed further by Flaminio and Forni \cite{FlaminioForni03InvariantDistributions}, who classified the invariant distributions and the obstructions to the cohomological equation, and by Forni and Ulcigrai \cite{forni2012time}, who established quantitative mixing for smooth time changes on compact quotients. These results exhibit time changed horocycle flows as genuinely nonhomogeneous parabolic systems which nevertheless retain strong rigidity and statistical properties.

Higher rank parabolic actions introduce a further obstruction. Wang \cite{Wang2015Cohomological} showed that common cohomological equations for commuting unipotent directions depend sharply on the surrounding geometry. For compatible cocycle data, smooth common solutions exist when the relevant directions embed in a subalgebra of type $\sl_2(\R)\oplus\R$, while counterexamples arise when this configuration is absent. In particular, her results show that smoothness alone does not yield a general cocycle rigidity statement in the $\SL_3(\R)$ geometry.

This phenomenon has a direct dynamical counterpart in the first author's work on parabolic perturbations \cite{RavottiParabolic19}. On compact $\SL_3(\R)$-quotients, he constructed nonhomogeneous perturbations in a commuting unipotent direction, proved mixing with respect to an equivalent invariant measure, and related smooth triviality of the perturbation to cocycle rigidity of the ambient action. Although these perturbations are not time changes, this work is one of the motivations for the present paper: it shows that a higher rank parabolic deformation may preserve strong dynamical properties while still carrying a genuine cohomological obstruction to algebraicity.

Ratner's time change rigidity theorem has also been extended beyond the classical horocycle setting. Tang \cites{tang2020new,Tang23} studied time changes of unipotent flows on compact Lorentz quotients and established rigidity results for their joinings. Artigiani, Flaminio and the first author \cite{afr22} subsequently extended Ratner's graph-joining method to general unipotent flows on finite volume semisimple quotients, under suitable regularity and graph-joining hypotheses.

A complementary perspective comes from Kakutani equivalence, cf.~\cites{kakutani1943131,katok1977monotone,ornstein1982equivalence}. Ratner proved that horocycle flows are loosely Kronecker \cite{ratner1978horocycle}, that their Cartesian squares are not \cite{ratner1979cartesian}, and introduced invariants distinguishing products with different multiplicities \cite{ratner1981some}. Kanigowski, Vinhage and the second author \cite{kanigowski2021kakutani} computed Ratner's invariant for general unipotent flows and characterized the loosely Kronecker class. More recently, Lindenstrauss and the second author \cite{LindenstraussWei2025} proved a dichotomy for one-parameter unipotent flows under time change: outside the exceptional loosely Kronecker class, measurable equivalence after time change forces an algebraic isomorphism, and the time change is cohomologous, after renormalization, to a constant one. 

\medskip

These developments point to the rigidity problem studied here. Rudolph's theorem shows that the smooth geometry of the orbit foliation is too flexible by itself, while Wang's work shows that smoothness of a higher rank parabolic cocycle alone need not force triviality. We ask whether compact directions in the ambient group can nevertheless recover the homogeneous dynamics of an ergodic directional restriction. We seek a classification of both the restriction and its invariant measure, together with sharpness examples for which smooth changes of coordinates within the unipotent leaves do not restore algebraicity.

For the first row action on every finite volume $\SL_3(\R)$-quotient, we prove that a strict mean square estimate in a single quadratic compact direction forces an ergodic directional restriction with nonzero drift to agree almost everywhere with its algebraic drift (the notion of \emph{drift} will be defined below). Its invariant probability measure is therefore homogeneous. The strict inequality is essential even when smooth leaf preserving conjugacies are allowed: on $\SL_3(\R)/\SL_3(\Z)$ we construct arbitrarily small smooth time changes satisfying the boundary estimate for which neither the selected restriction nor the full action is smoothly conjugate to an algebraic action by a leaf preserving diffeomorphism.

More generally, the method applies to abelian unipotent actions admitting a compact Ratner frame, as defined in Definition~\ref{def:CptRatnerFra}. Such a frame consists of compact directions whose adjoint polynomials along the drift have highest degree coefficients spanning the acting algebra. Pointwise growth bounds, or suitable correlation bounds, in these directions force algebraicity of the restriction and homogeneity of its invariant measure. The pointwise hypotheses are sharp: on a product of two compact horocycle quotients, we construct arbitrarily small smooth time changes satisfying the boundary estimates whose diagonal restriction is mixing, but for which neither the diagonal flow nor the full action admits a smooth leaf preserving conjugacy to an algebraic action.

For first row actions in $\SL_n(\R)$, the frame has one quadratic compact root direction and $n-2$ degree one directions. These different adjoint degrees determine the corresponding regularity thresholds, giving rigidity criteria on every finite volume quotient. The directional results also determine higher dimensional subactions. Under the stated differentiability and integrability assumptions, it is enough to verify the rigidity hypothesis along finitely many ergodic directions spanning the parameter subspace; the cocycle identity then recovers the entire subaction.

The pointwise thresholds for these first row actions are sharp for exact algebraicity on every such quotient. Smooth changes of coordinates along the unipotent leaves give arbitrarily small time changes satisfying the boundary estimates in all frame directions, while their selected restrictions are not algebraic in the original coordinates. Unlike the preceding nonconjugate examples, these actions are smoothly conjugate to algebraic ones. Thus failure of exact algebraicity at the boundary need not reflect an obstruction to smooth conjugacy. 

\medskip

On compact $\SL_3(\R)$ quotients, every smooth time change of the full first row action satisfies the quadratic compact root bound. This makes it necessary to distinguish what follows from the boundary estimate from what requires an additional cohomological condition. We prove that smooth leaf preserving conjugacy of the full action to a linear algebraic action is equivalent to bounded inverse cocycle deviations along one direction, with the $L^2$ bound uniform in time. 
On a prescribed compact quotient, time changes outside this conjugacy class exist exactly when there is a smooth closed pair whose components have zero Haar mean and which has no common smooth primitive. 

The compact root asymptotics explain how this conjugacy criterion differs from exact algebraicity. For a restriction ergodic with respect to the natural smooth invariant probability, after normalizing the drift, we identify the leading quadratic coefficient and show that it vanishes precisely when the full time change is linear. For compact root perturbations of size $\sigma/t^2$, with $\sigma$ fixed, the difference between the inverse cocycle values at the perturbed and the original base points is asymptotic in $L^2$ to the full inverse cocycle deviation evaluated at the endpoint of the directional orbit. This deviation need not depend linearly on $\sigma$. Smooth conjugacy requires this deviation to remain bounded, not to vanish, and therefore allows a nonzero quadratic coefficient. The same asymptotics determine the optimal compact root exponents for nonconstant smoothly conjugate time changes: for each prescribed H\"older exponent $0<\delta\leq 1$, the optimal time growth exponent is $2\delta$, and the bound is attained. Such examples exist on every compact quotient.

We now state the precise results.
\subsection*{Setting}
Let $n\geq 3$, set $G=\SL_n(\R)$, and let $\Gamma<G$ be a lattice. Define
    \[
    U_j=E_{1j}\quad (2\leq j\leq n),\qquad \h_n=\operatorname{span}\{U_2,\ldots,U_n\},
    \]
    \[
    H_n=\exp(\h_n),\qquad X=G/\Gamma.
    \]
We identify $\R^{n-1}$ with $\h_n$ by
\[
\iota(p_2,\ldots,p_n)=\sum_{j=2}^np_jU_j,
\]
and define $\phi_px=\exp(\iota(p))x$ for the algebraic action. A diffeomorphism is called leaf preserving if it preserves each individual orbit of the algebraic unipotent action under consideration.

\begin{Definition}[Continuous time change]\label{def:continuousTimeChange}
    A continuous time change of $\phi$ is a jointly continuous cocycle $\rho:X\times\R^{n-1}\to\R^{n-1}$ satisfying
    \begin{equation*}
        \rho(x,p+q)=\rho(x,p)+\rho(\phi_px,q),
    \end{equation*}
    such that every fiber map $\rho_x:p\mapsto\rho(x,p)$ is a homeomorphism of $\R^{n-1}$, and the inverse family $\tau(x,w)=\rho_x^{-1}(w)$ is jointly continuous. The induced action is 
    \begin{equation*}
        \alpha_wx=\phi_{\tau(x,w)}x.
    \end{equation*}
\end{Definition}

A time change is called smooth if its reparametrization cocycle is jointly smooth and every parameter map is a smooth diffeomorphism. For an algebraic abelian unipotent $\R^d$-action $\phi$, a linear algebraic reparametrization is the action $w\mapsto\phi_{A^{-1}w}$, where $A\in\GL_d(\R)$. A smooth leaf preserving conjugacy to this action is a smooth leaf preserving diffeomorphism $\Psi$ satisfying
\[
\Psi\alpha_w=\phi_{A^{-1}w}\Psi,\qquad w\in\R^d.
\]

\begin{Standing}
    Fix $v\in \R^{n-1}\setminus\{0\}$. We define the directional flow $(T_t)_{t\in\R}$ along $v$ and the associated cocycle $\beta=\beta(x,t)$ as
    \begin{equation}\label{eq:defin_of_T_and_beta}
    T_t=\alpha_{tv},\qquad \beta(x,t)=\iota\big(\tau(x,tv)\big)\in\h_n.
    \end{equation}
    Let $\nu$ be a $T_t$-invariant ergodic probability. We assume that there is $V=V_{\nu}\in\h_n\setminus\{0\}$ such that
    \begin{equation}\label{eq:betaConv}
        \lim_{t\to+\infty}\frac{\beta(x,t)}{t}=V\quad \text{for $\nu$-a.e. $x$.}
    \end{equation}
    We call $V$ as above, if it exists, the \emph{asymptotic drift} of $(T_t)_{t\in\R}$ with respect to $\nu$.
    
    If $Lx\subset X$ is a closed finite-volume homogeneous orbit, $m_{Lx}$ denotes its normalized Haar probability. Throughout the paper, all measures on $X$ are Borel measures. We identify a Borel probability measure with its completion when discussing almost-everywhere statements.
\end{Standing}

Lemma~\ref{lem:contSameOrbPar} proves that $(\alpha_w)$ is a continuous abelian action with the same orbit partition as $H_n$, and proves the cocycle identity for $\beta$. Lemma~\ref{lem:oneStepIntCri} gives a standard one-step integrability criterion for \eqref{eq:betaConv}.

\medskip

\subsection*{The case of \texorpdfstring{$\SL_3(\R)$}{SL3R}}
Lemma~\ref{lem:normalization} conjugates every nonzero $V\in\h_n$ to $E_{12}$. Whenever normalized coordinates are used, the action, lattice, and measure are the transported objects defined in \S\ref{sec:cocycleNormalization}; in particular, $\nu$ denotes the corresponding pushforward measure. We work in the resulting coordinates and put
\begin{equation}
    K_2=E_{21}-E_{12},\qquad K_j=E_{j2}-E_{2j}\quad \text{for }3\leq j\leq n.
\end{equation}
Each $\exp(\R K_j)$ is a compact root circle. For a continuous $F:X\to\h_n$, set
\[
[F]_{K,\theta,s_0}=\sup_{0<\abs{s}\leq s_0}\sup_{x\in X}\frac{\norm{F(\exp(sK)x)-F(x)}}{\abs{s}^{\theta}}.
\]

For $n=3$, we define
\[
\Delta_s^{(2)}(x,t)=\beta(\exp(sK_2)x,t)-\beta(x,t).
\]
\begin{Definition}[One-root mean square class]\label{def:oneRootMeanSquare}
    For one normalization supplied by Lemma~\ref{lem:normalization}, the restricted flow belongs to $\mathcal{M}_2(v,\nu)$ if there are $C,s_0>0$, $\delta>0$, and $\gamma\geq 0$ such that
    \begin{equation}
        \norm{\Delta_s^{(2)}(\cdot,t)}_{L^2(\nu)}\leq C\abs{s}^{\delta}(1+\abs{t})^{\gamma}, \qquad 2\delta>\gamma,
    \end{equation}
    for $\abs{s}\leq s_0$ and $t\geq 0$.
\end{Definition}

\begin{Theorem}[One-root mean square rigidity]\label{thm:oneRootSL3rigidity}
Let $n=3$ and $\Gamma<\SL_3(\R)$ a lattice. If the normalized restriction belongs to $\mathcal{M}_2(v,\nu)$, then there are a closed connected subgroup $L<\SL_3(\R)$ and a closed finite-volume orbit $Lx$ such that
\[
\exp(\R V_{\nu})\leq L,\qquad T_ty=\exp(tV_{\nu})y\quad \text{$\nu$-a.e. $y$ and $\forall t\in\R$},\qquad \nu=m_{Lx}.
\]
\end{Theorem}

\medskip

\begin{Theorem}[Sharpness of Theorem~\ref{thm:oneRootSL3rigidity}]\label{thm:sharpnessSL3rigidityOneRoot}
   On $X=\SL_3(\R)/\SL_3(\Z)$ there is a smooth family $(\alpha^{\epsilon})_{0<\epsilon<\epsilon_0}$ of time changes of the standard first row $H_3$-action, converging to the algebraic action smoothly on compact sets as $\epsilon\to0$, with a selected restriction $T_t^{\epsilon}$ and a full support ergodic invariant probability $\nu_{\epsilon}$ such that:
   \begin{enumerate}[label=(\roman*)]
       \item\label{item:sharpSL31} the asymptotic drift of $T_t^{\epsilon}$ is $U_2=E_{12}$;
       \item\label{item:sharpSL32} with $K_2=E_{21}-E_{12}$ and $T_t^{\epsilon}x=\exp(q_{\epsilon}(x,t)U_2)x$,
       \[
       \norm{q_{\epsilon}(\exp(sK_2)\cdot, t)-q_{\epsilon}(\cdot,t)}_{L^2(\nu_{\epsilon})}\leq C_{\epsilon}\abs{s}(1+\abs{t})^2
       \]
       for every $s,t\in\R$;
       \item\label{item:sharpSL33} neither $T_t^{\epsilon}$ nor the full time-changed $H_3$-action is smoothly conjugate to an algebraic action by a leaf preserving diffeomorphism.
   \end{enumerate}
   Consequently, the equality case $(\delta,\gamma)=(1,2)$ does not imply rigidity, even after conjugacy by a smooth leaf preserving diffeomorphism.
\end{Theorem}

The pointwise form of Theorem~\ref{thm:oneRootSL3rigidity} has the same degree-two scale.
\begin{Corollary}[Root-H\"{o}lder rigidity in finite-volume $\SL_3$-quotients]\label{cor:RootRigiditySL3}
Let $n=3$ and $\Gamma<\SL_3(\R)$ a lattice. Assume that, in normalized coordinates, there are $C,s_0>0$, $0<\delta\leq 1$, and $\gamma\geq 0$ such that
\begin{equation}
    [\beta(\cdot,t)]_{K_2,\delta,s_0}\leq C(1+\abs{t})^{\gamma},\qquad 2\delta>\gamma,\quad t\geq 0.
\end{equation}
Then the conclusion of Theorem~\ref{thm:oneRootSL3rigidity} holds.
\end{Corollary}

\subsection*{Compact Ratner frames and pointwise optimality}
The preceding argument has an abstract form. Let $G_0$ be a connected semisimple real linear group, $\Gamma_0<G_0$ a lattice and  $X_0=G_0/\Gamma_0$. Let $\mathfrak{u}<\mathfrak{g}_0$ be an abelian Lie algebra of nilpotent elements and set $N=\exp(\mathfrak{u})$. Consider a directional restriction of a continuous time change of the algebraic $N$-action
\[
T_tx=\exp(\beta(x,t))x,\qquad \beta(x,t)\in\mathfrak{u},
\]
preserving an ergodic probability measure $\nu$. Assume that $\beta$ is continuous and that, for some $V\in\mathfrak{u}\setminus\{0\}$,
\[
\lim_{t\to+\infty}\frac{\beta(x,t)}{t}=V\qquad \text{for $\nu$-a.e. }x.
\]
\begin{Definition}[Compact Ratner frame]\label{def:CptRatnerFra}
    A compact Ratner frame for $(\mathfrak{u},V)$ is a family
    \[
    (K_i,d_i,A_i)_{i=1}^m,\qquad K_i\in\mathfrak{g}_0,\qquad d_i\geq1,\quad d_i\in\N
    \]
    such that:
    \begin{enumerate}[label=(\roman*)]
        \item $\exp(\R K_i)$ is relatively compact;
        \item $\ad_{Z_1}\ldots\ad_{Z_{d_i+1}}K_i=0$ for all $Z_1,\ldots,Z_{d_i+1}\in\mathfrak{u}$;
        \item $A_i=\frac{1}{d_i!}\ad_V^{d_i}K_i$ and $(A_1,\ldots,A_m)$ is a basis of $\mathfrak{u}$.
    \end{enumerate}
\end{Definition}

The frame is pointwise regular for $(T,\nu)$ if, for each $i$, there are constants $C_i,s_i>0$, exponents $\delta_i>0$ and $\gamma_i\geq 0$, such that
\[
\sup_x\norm{\beta(\exp(sK_i)x,t)-\beta(x,t)}\leq C_i\abs{s}^{\delta_i}(1+\abs{t})^{\gamma_i},\qquad d_i\delta_i>\gamma_i,
\]
for $\abs{s}\leq s_i$ and $t\geq 0$. The correlation regular version is given in Definition~\ref{def:RatnerCorReg}.

\begin{Theorem}[Abstract compact-frame rigidity]\label{thm:absCptRigid}
    Suppose that $(\mathfrak{u},V)$ admits a compact Ratner frame which is either pointwise regular or correlation regular for $(T,\nu)$. Then 
    \[
    \beta(x,t)=tV\qquad \text{for $\nu$-a.e. $x$ and $\forall t\in\R$,}
    \]
    and $\nu$ is a Ratner homogeneous probability measure for the algebraic flow $\exp(tV)$.
\end{Theorem}

The pointwise inequalities in Theorem~\ref{thm:absCptRigid} are also strictly optimal.
\begin{Theorem}[Sharpness of Theorem~\ref{thm:absCptRigid}]\label{thm:sharpnessAbsCptRigid}
    For $i=1,2$, let $G_i=\PSL_2(\R)$, $\Gamma_i<G_i$ a torsion-free cocompact lattice and $X_i=G_i/\Gamma_i$. In $\sl_2(\R)$ set
    \[
    U=E_{12},\qquad K=E_{21}-E_{12},
    \]
    and let $U^{(i)}$ and $K^{(i)}$ denote their copies in the $i$-th factor of $\sl_2(\R)\oplus\sl_2(\R)$.
    
    On $X=X_1\times X_2$, set 
    \[
    \mathfrak{u}=\R U^{(1)}\oplus\R U^{(2)},\qquad V=U^{(1)}+U^{(2)}.
    \]
    There exist $\epsilon_0>0$ and smooth time changes $(\alpha^{\epsilon})_{0<\abs{\epsilon}<\epsilon_0}$ of the product $\R^2$-action, converging to the algebraic action in $C^{\infty}$ as $\epsilon\to0$, such that:
    \begin{enumerate}[label=(\roman*)]
        \item\label{item:sharpnessAbstract1}  the diagonal restriction $T_t^{\epsilon}=\alpha_{(t,t)}^{\epsilon}$ is mixing for a full-support invariant probability and has drift $V$;
        \item\label{item:sharpnessAbstract2}  $(K^{(1)},2,-U^{(1)})$ and $(K^{(2)},2,-U^{(2)})$ form a compact Ratner frame for $(\mathfrak{u},V)$, and, for some $s_0>0$, 
        \[
        [\beta_{\epsilon}(\cdot,t)]_{K^{(i)},1,s_0}\leq C_{i,\epsilon}(1+\abs{t})^2,\qquad i=1,2,\quad t\in\R;
        \]
        thus $d_i\delta_i=\gamma_i=2$ for both frame directions;
        \item\label{item:sharpnessAbstract3}  neither the full action nor the diagonal restriction is smoothly conjugate to an algebraic action by a leaf preserving diffeomorphism.
    \end{enumerate}
\end{Theorem}

Theorems~\ref{thm:oneRootSL3rigidity} and~\ref{thm:sharpnessSL3rigidityOneRoot} give the exact mean square threshold in $\SL_3$. Theorem~\ref{thm:absCptRigid} gives abstract compact frame rigidity, while Theorem~\ref{thm:sharpnessAbsCptRigid} proves the sharpness of its pointwise hypothesis.

\subsection*{Further forms of the rigidity theorem}

For the general first row action, we use one packet containing the degree two direction $K_2$ and the degree-one directions $K_3,\ldots,K_n$. If $B=\sum_{j=2}^nb_jU_j$ and $0<\epsilon<1$, define
\begin{equation}
    Z_{\epsilon}(B)=-b_2\epsilon^2K_2-\epsilon\sum_{j=3}^nb_jK_j,\qquad g_{\epsilon,B}=\exp Z_{\epsilon}(B),
\end{equation}
and define
\[
\Delta_{\epsilon,B}(x,t)=\beta(g_{\epsilon,B}x,t)-\beta(x,t).
\]

Whenever an $L^2$-valued derivative occurs below, we use the same symbol for a fixed jointly measurable representative.
\begin{Definition}[First-row Ratner correlation class]\label{def:FirstRowRatnerCor}
    For one normalization supplied by Lemma~\ref{lem:normalization}, the restricted flow belongs to $\mathcal{K}_{\operatorname{fr}}(v,\nu)$ if there are $\chi>0$, $\delta_2,\ldots,\delta_n>0$, and $\epsilon_0>0$ such that 
    \begin{equation}\label{eq:kappaFr}
       \kappa_{\operatorname{fr}}\coloneq\min\{2\delta_2,\delta_3,\ldots,\delta_n\}>1-\frac{1}{2}\min\{\chi,1\}, 
    \end{equation}
    and, for every $R>0$, there is $C_R$ with the following property. For $0<\epsilon\leq\epsilon_0$ and $\norm{B}\leq R$, the map $t\mapsto \Delta_{\epsilon,B}(\cdot,t)$ is locally absolutely continuous as an $L^2(\nu,\h_n)$-valued map. If $\mathcal{D}_{\epsilon,B}$ denotes its derivative, then
    \begin{equation}
        \abs{\int_X\langle\mathcal{D}_{\epsilon,B}(x,u),\mathcal{D}_{\epsilon,B}(x,w)\rangle d\nu(x)}\leq C_R\left(\epsilon^{2\delta_2}+\sum_{j=3}^n\epsilon^{\delta_j}\right)^2(1+\abs{u-w})^{-\chi}
    \end{equation}
    for almost every $u,w\geq 0$.
\end{Definition}

\begin{Theorem}[First-row rigidity on finite-volume quotients]\label{thm:firstRowRigidity}
    Let $n\geq 3$ and $\Gamma<\SL_n(\R)$ a lattice. If the normalized restriction belongs to $\mathcal{K}_{\operatorname{fr}}(v,\nu)$, then the conclusion of Theorem~\ref{thm:oneRootSL3rigidity} holds with $\SL_n(\R)$ in place of $\SL_3(\R)$.
\end{Theorem}
\medskip

\begin{Definition}[Pointwise compact root class]\label{def:PointComRoot}
    For one normalization supplied by Lemma~\ref{lem:normalization}, the restricted flow belongs to $\mathcal{H}_{\operatorname{fr}}(v)$ if there are $C_j>0$, $s_0>0$, exponents $0<\delta_j\leq1$, and $\gamma_j\geq 0$ such that 
    \begin{equation}
        \begin{aligned}
            &[\beta(\cdot,t)]_{K_2,\delta_2,s_0}\leq C_2(1+\abs{t})^{\gamma_2},\qquad 2\delta_2>\gamma_2,\\
            &[\beta(\cdot,t)]_{K_j,\delta_j,s_0}\leq C_j(1+\abs{t})^{\gamma_j},\qquad \delta_j>\gamma_j\quad \text{for }3\leq j\leq n
        \end{aligned}
    \end{equation}
    for every $t\geq 0$.
\end{Definition}

\begin{Corollary}[Pointwise compact root criterion]\label{cor:pointCptCri}
    The conclusion of Theorem~\ref{thm:firstRowRigidity} holds if the normalized restriction belongs to $\mathcal{H}_{\operatorname{fr}}(v)$.
\end{Corollary}

\medskip

When $n=3$, the conclusion already follows from Corollary~\ref{cor:RootRigiditySL3} under the $K_2$-estimate alone; the $K_3$-estimate is needed only for the uniform compact-frame argument.

For $n>3$, the degree-one directions in Definition~\ref{def:PointComRoot} are necessary for this compact-frame construction: the calculation following Lemma~\ref{lem:firstRowCptRatner} shows that every degree-two top coefficient lies in the drift line, whereas $K_3,\ldots,K_n$ generate the remaining first row directions.

\medskip

The argument also treats higher-dimensional restrictions using finitely many ergodic directions. Let 
\[
\{0\}\neq P<\R^{n-1}
\]
be a linear subspace and $\nu$ a $P$-invariant probability measure, and, for $p\in P$, write
\[
T_t^p\coloneq \alpha_{tp}.
\]
For every $x\in X$, assume that the map
\[
P\ni p\longmapsto \tau(x,p)
\]
is Fr\'{e}chet differentiable at $0$. Denote its derivative by 
\[
D_P\tau(x,0)\in\operatorname{Hom}(P,\R^{n-1}).
\]
Assume that $x\mapsto D_P\tau(x,0)$ is continuous and that
\[
\iota\circ D_P\tau(\cdot,0)\in L^1(\nu,\operatorname{Hom}(P,\h_n)).
\]
Define the linear drift map
\begin{equation}
    \mathcal{V}_{\nu}(p)\coloneq\int_X\iota \big(D_P\tau(x,0)(p)\big)d\nu(x),\qquad p\in P.
\end{equation}

\begin{Theorem}[Finite-basis subaction rigidity]\label{thm:finiteSubactionRigidity}
Suppose that $p_1,\ldots,p_k$ form a basis of $P$. Assume that, for each $i$, the flow $T^{p_i}$ is $\nu$-ergodic, 
\[
\mathcal{V}_{\nu}(p_i)\neq0,
\]
and the corresponding normalized line flow satisfies the hypotheses of one of Theorem~\ref{thm:oneRootSL3rigidity}, Corollary~\ref{cor:RootRigiditySL3} when $n=3$, Theorem~\ref{thm:firstRowRigidity}, or Corollary~\ref{cor:pointCptCri}. Then
\[
\iota\big(\tau(x,p)\big)=\mathcal{V}_{\nu}(p)\quad \text{for $\nu$-a.e. $x$ and every $p\in P$.}
\]
Consequently, the $P$-subaction is algebraic modulo $\nu$, and $\nu$ is homogeneous.
\end{Theorem}
\medskip

The following proposition constructs time changes whose directional restrictions satisfy the boundary estimate but are not algebraic in the original coordinates. Nevertheless, the full actions are smoothly conjugate to algebraic actions by leaf preserving diffeomorphisms.
\begin{Proposition}[Sharpness for exact algebraicity]\label{prop:SharpProp}
    Let $\Gamma<\SL_n(\R)$ be any lattice and $v\neq0$. For every choice of exponents $\theta_2,\ldots,\theta_n\in(0,1]$, there exist a $C^{\infty}$ time change $\alpha^F$ and a full-support $T_t^F$-ergodic probability measure $\nu_F$ such that the asymptotic drift is $\iota(v)$ and $T_t^F$ is not algebraic modulo $\nu_F$. After normalizing the drift to $E_{12}$, there are $C,s_0>0$ such that, for every $t\in\R$,
    \[
    [\beta_F(\cdot,t)]_{K_2,\theta_2,s_0}\leq C(1+\abs{t})^{2\theta_2},
    \]
    \[
    [\beta_F(\cdot,t)]_{K_j,\theta_j,s_0}\leq C(1+\abs{t})^{\theta_j}\qquad \text{for }3\leq j\leq n.
    \]
    In dimension three the corresponding one-root $L^2$ estimate also holds with $2\delta=\gamma$. Thus the strict inequalities in  Definition~\ref{def:oneRootMeanSquare} and Definition~\ref{def:PointComRoot} cannot in general be replaced by non-strict inequalities when exact algebraicity in the original coordinates is required.
\end{Proposition}

\subsection*{Smooth time changes on compact quotients}
The directional rigidity theorems above concern continuous time changes and ergodic restriction measures that need not be equivalent to Haar measure. Their compact root hypotheses do not require joint smoothness of $\rho$. We now assume that the time change is smooth and the quotient is compact. In this setting the quadratic bound is automatic. We characterize smooth algebraic conjugacy and determine the leading compact root terms, including the term that survives at the critical scale.

In this subsection, let $X=\SL_3(\R)/\Gamma$, where $\Gamma$ is cocompact, and let $m$ be Haar probability. Write parameter vectors as columns $p=(p_1,p_2)^{\mathrm{T}}$, with $\iota(p)=p_1U_2+p_2U_3$ and $e_1=(1,0)^{\mathrm{T}}$. Let $\rho$ be a smooth time change, $\tau$ its parameter inverse, and $\alpha$ the induced action. Set
\[
M(x)=D_{p}\rho(x,0),\qquad A=\int_XMdm.
\]
By Proposition~\ref{prop:compactMeasure}, the matrix $A$ is invertible, and
\[
d\mu=\frac{\det M}{\det A}dm
\]
is a smooth full support probability invariant under the full action.

\begin{Theorem}[Smooth conjugacy criterion]\label{thm:smoothConjugacyCri}
The following are equivalent:
\begin{enumerate}[label=(\roman*)]
    \item\label{item:thm:smoothConjugacyCri1} $\sup_{t\in\R}\norm{\rho(\cdot,te_1)-tAe_1}_{L^2(m)}<\infty$;
    \item\label{item:thm:smoothConjugacyCri2}  $\sup_{t\in\R}\norm{\tau(\cdot,tAe_1)-te_1}_{L^2(\mu)}<\infty$;
    \item\label{item:thm:smoothConjugacyCri3}  there exists $b\in C^{\infty}(X,\R^2)$ such that
    \[
    \rho(x,p)=Ap+b(\phi_px)-b(x),\qquad x\in X,\quad p\in\R^2;
    \]
    \item\label{item:thm:smoothConjugacyCri4}  the full action is smoothly conjugate to a linear algebraic action by a diffeomorphism preserving each individual $H_3$-orbit;
    \item\label{item:thm:smoothConjugacyCri5}  for some $v\neq0$, the restriction $t\mapsto\alpha_{tv}$ is smoothly conjugate to an algebraic $H_3$-flow by a global diffeomorphism preserving each individual $H_3$-orbit.
\end{enumerate}
    When these conditions hold, the transfer $b$ in \ref{item:thm:smoothConjugacyCri3} is unique up to a constant vector. Given such a transfer, a full conjugacy and its inverse are
    \begin{equation}\label{eq:compactConjugacyMaps}
        \Psi(x)=\phi_{A^{-1}b(x)}x,\qquad \Psi^{-1}(x)=\alpha_{-b(x)}x.
    \end{equation}
    Each conjugacy $\Psi$ in \ref{item:thm:smoothConjugacyCri4} or \ref{item:thm:smoothConjugacyCri5} satisfies
    \[
    \Psi\alpha_{\eta}=\phi_{A^{-1}\eta}\Psi\quad \text{for every }\eta\in\R^2,\qquad \Psi_*\mu=m.
    \]
\end{Theorem}

The existence of time changes outside this conjugacy class is determined by the smooth cocycle equation on the given quotient.

\begin{Theorem}[Non conjugate time changes on a prescribed quotient]\label{thm:NonConjTCGivenQuo}
   For the fixed compact quotient $X$, the following are equivalent:
   \begin{enumerate}[label=(\roman*)]
       \item\label{item:thm:NonCOnjTCGivenQuo1} there are real $f,g\in C^{\infty}(X)$ with
       \[
       \int_Xfdm=\int_Xgdm=0,\qquad U_2g=U_3f,
       \]
       but no $b\in C^{\infty}(X)$ satisfying $(U_2b,U_3b)=(f,g)$;
       \item\label{item:thm:NonCOnjTCGivenQuo2} there is a smooth first row time change not smoothly conjugate to a linear algebraic action by a leaf preserving diffeomorphism;
       \item\label{item:thm:NonCOnjTCGivenQuo3} there are arbitrarily small such time changes whose first restrictions are positive scalar time changes of the $U_2$ flow, ergodic for smooth full support probabilities, and are not smoothly conjugate to algebraic $H_3$ flows by leaf preserving diffeomorphisms.
   \end{enumerate}
   Here smallness means convergence to the algebraic action in $C^{\infty}$ on compact parameter sets.
\end{Theorem}

For a fixed physical direction $v\neq0$, put
\[
T_t=\alpha_{tv},\qquad  D_s(x,t)=\tau(\exp(sK_2)x,tv)-\tau(x,tv).
\]
Thus $\beta(x,t)=\iota(\tau(x,tv))$ and $\Delta_s^{(2)}(x,t)=\iota(D_s(x,t))$. The operator $K_2$ differentiates the base point of $\tau$, and all parameter vector norms are Euclidean.

\begin{Theorem}[Compact root asymptotics]\label{thm:CptRtAsy}
\begin{enumerate}[label=(\roman*)]
    \item\label{item:CptRtAsyitem1} For every $v\neq 0$, there is $C>0$ such that
    \[
    \norm{K_2\tau(\cdot,tv)}_{\infty}\leq C(\abs{t}+t^2),\qquad \norm{D_s(\cdot,t)}_{\infty}\leq C\abs{s}(\abs{t}+t^2),\qquad \text{for all }s,t\in\R.
    \]
    \item\label{item:CptRtAsyitem2} Suppose that $T_t$ is $\mu$-ergodic, and use the normalized coordinates of Lemma~\ref{lem:normalization}, so that $A^{-1}v=e_1$. The leading compact root derivative is
    \begin{equation}\label{eq:CptRtAsyitem2eq1}
        K_2\tau(x,tv)=t^2(e_1-M^{-1}(T_tx)v)+o_{L^2(\mu)}(t^2),\qquad \text{$t\to+\infty$.}
    \end{equation}
    In particular,
    \begin{equation}\label{eq:CptRtAsyitem2eq2}
         \lim_{t\to+\infty}\frac{\norm{K_2\tau(\cdot,tv)}_{L^2(\mu)}}{t^2}=\norm{e_1-M^{-1}v}_{L^2(\mu)}.
    \end{equation}
    For every real function $s(t)\neq0$ with $\abs{s(t)}t^2\to0$, the increments satisfy the relative expansion
    \begin{equation}\label{eq:CptRtAsyitem2eq3}
        D_{s(t)}(x,t)=s(t)t^2(e_1-M^{-1}(T_tx)v)+o_{L^2(\mu)}(\abs{s(t)}t^2).
    \end{equation}
   
    \item\label{item:CptRtAsyitem3} Under the hypothesis of \ref{item:CptRtAsyitem2}, the term at the critical scale $s=\sigma/t^2$ is the full inverse cocycle deviation $\sigma e_1+\tau(T_tx,-\sigma v)$. More precisely, for every $S>0$
    \[
    \norm{\sup_{\abs{\sigma}\leq S}\abs{D_{\sigma/t^2}(\cdot,t)-\sigma e_1-\tau(T_t\cdot,-\sigma v)}}_{L^2(\mu)}\longrightarrow0,\qquad \text{for }t\to+\infty.
    \]
\end{enumerate}
Here $o_{L^2(\mu)}(a_t)$ denotes a vector valued remainder whose $L^2(\mu)$ norm is $o(a_t)$.
\end{Theorem}

In these normalized coordinates, the coefficient $e_1-M^{-1}v$ vanishes identically exactly when $\rho(x,p)=Ap$. Smooth conjugacy allows a nonzero coefficient. By Theorem~\ref{thm:smoothConjugacyCri}, it is instead characterized by
\[
\sup_{\sigma\in\R}\norm{\sigma e_1+\tau(\cdot,-\sigma v)}_{L^2(\mu)}<\infty.
\]
The critical term can therefore be nonzero and still remain bounded. For small $\sigma$ it begins with $\sigma(e_1-M^{-1}v)$; for fixed $\sigma$, the entire inverse cocycle occurs in \ref{item:CptRtAsyitem3}.

For the same ergodic normalized restriction, fix $s_0>0$ and define, for $\delta>0$,
\[
\begin{aligned}
  \mathcal{N}_{\delta}(t)=&\sup_{0<\abs{s}\leq s_0}\frac{\norm{D_s(\cdot,t)}_{L^2(\mu)}}{\abs{s}^{\delta}},\\
  \gamma_*(\delta)=&\inf\{\gamma\geq 0:\exists C<\infty\text{ such that $\mathcal{N}_{\delta}(t)\leq C(1+t)^{\gamma}$ for every $t\geq 0$}\}.
\end{aligned}
\]
The infimum of the empty set is $+\infty$; a finite infimum need not be attained.

\begin{Theorem}[Optimal compact root exponents]\label{thm:OptCptRtExp}
    Suppose that $T$ is $\mu$-ergodic, $A^{-1}v=e_1$, and $M$ is nonconstant. Then
    \[
    2\delta\leq\gamma_*(\delta)\leq1+\delta \quad \text{ for $0<\delta\leq 1$,}\qquad \gamma_*(\delta)=+\infty \quad \text{for $\delta>1$}.
    \]
    In particular $\gamma_*(1)=2$, with an estimate at equality. If the action is smoothly leaf preservingly conjugate to a linear algebraic action, then
    \[
    \gamma_*(\delta)=2\delta\quad \text{for $0<\delta\leq1$},
    \]
    and the bound at equality is attained for every $0<\delta\leq 1$. For these conjugate actions, the same equality and attained estimates hold with the uniform norm on $X$ in place of $L^2(\mu)$.
\end{Theorem}

Section~\ref{sec:cptExamplesOptCoe} constructs nonconstant smoothly conjugate examples on every compact quotient and verifies the attained optimal estimates directly. Each fixed example realizes every pair $(\delta,2\delta)$ with $0<\delta\leq 1$, including $(3/4,3/2)$ and $(1/\sqrt{2},\sqrt{2})$. For the natural measure and an ergodic normalized restriction, the strict inequality $2\delta>\gamma$ forces the full time change to be linear. The earlier rigidity theorems determine the restriction modulo its given ergodic measure, which may be singular. Smallness of a smooth perturbation can reduce the leading coefficient without improving its time growth exponent.

\subsection*{Overview of the argument}

We first normalize the asymptotic drift by an algebraic change of coordinates. In the first row action, conjugation of suitable compact root directions by the normalized unipotent flow produces polynomial displacements. One direction has degree two, while the remaining directions have degree one. The abstract notion of a compact Ratner frame isolates exactly the algebraic property needed from this collection of compact directions.

A perturbation of size $\varepsilon^d$ in a compact direction of degree $d$, followed for a time of order $\varepsilon^{-1}$, produces a nontrivial limiting displacement in the acting unipotent group. The regularity hypotheses ensure that the corresponding change in the physical cocycle vanishes at this scale. Passing to limits of paired orbit averages produces invariant measures supported on algebraic graphs, and hence commutation relations with the restricted flow.

The highest degree coefficients of a compact Ratner frame span the acting unipotent algebra. The resulting graph relations give commutation with the full algebraic unipotent action. Discreteness of stabilizers makes the physical cocycle invariant along the unipotent leaves. The cocycle identity and the drift then give $\beta(x,t)=tV$. Ratner's measure classification identifies the invariant probability measure.

The sharpness constructions reflect the two regularity settings of the paper. On the modular $\SL_3(\R)$-quotient, a primitive Siegel transform cocycle gives an exact mean square boundary example, and periodic unipotent orbits obstruct smooth algebraic conjugacy. In the pointwise compact frame setting, two independent nontrivial horocycle time changes give a boundary example on a product of compact quotients. 

\subsection*{Organization of the paper} 
Section~\ref{sec:cocycleNormalization} records the cocycle identities and the normalization of the drift. Sections~\ref{sec:cptRatnerFra} and \ref{sec:RigidityFullGRaFam} contain the two main steps of the proof of the rigidity result in the abstract setting of compact frames. Section~\ref{sec:firstRowCptFrame} constructs the first row frame, and sections~\ref{sec:sharperThreeDim} and \ref{sec:HigheerDimSubactions} prove the three-dimensional and higher-dimensional positive results. The following three sections then pass from exact algebraicity in section~\ref{sec:sharpnessScaleCond}, to mean square sharpness in section~\ref{sec:sl3Sharpness}, and finally to pointwise compact-frame sharpness in section~\ref{sec:PtSharpCptRatnerFra}. Section~\ref{sec:smoothTCcptSL3} establishes the compact conjugacy criterion, and Section~\ref{sec:cptRtAsyOpExp} proves the compact root asymptotics and optimal exponent results. Appendix~\ref{sec:smoothDisplacement} proves the smooth displacement lemma used in the nonconjugacy arguments.

\subsection*{Acknowledgement}
This paper originated in a question posed by Giovanni Forni to the first named author many years ago. We thank him for his warm support, helpful discussion and insights. We thank Zhenqi Wang for helpful discussions concerning her results on cohomological equations.

\section{Cocycle identities and normalization}\label{sec:cocycleNormalization}
In this section, we collect some basic facts about time changes of abelian actions and the flows obtained as restrictions to a single direction. In particular, for any directional flow as above, we define its asymptotic drift with respect to an ergodic measure. Finally, by an algebraic change of coordinates, we normalize the drift to be a fixed direction.

We retain the notation and standing assumptions of the introduction. The next lemma records the inverse cocycle identity used throughout the paper. 
\begin{Lemma}\label{lem:contSameOrbPar}
For a continuous time change,
\[
\tau(x,w+z)=\tau(x,w)+\tau(\alpha_wx,z)
\]
for all $x\in X$ and $w,z\in\R^{n-1}$. Consequently, $(\alpha_w)_{w\in\R^{n-1}}$ is a continuous abelian action, and its orbits coincide with the $H_n$-orbits. Finally, the 
function $\beta$ as defined in \eqref{eq:defin_of_T_and_beta} is a cocycle for the flow $(T_t)_{t\in\R}$; namely, it satisfies

\[
\beta(x,t+s)=\beta(x,t)+\beta(T_tx,s)
\]
for all $x\in X$ and $s,t\in\R$.
\end{Lemma}
\begin{proof}
    The cocycle identity at $p=q=0$ gives $\rho(x,0)=0$, hence $\tau(x,0)=0$. Define
    \[
    p=\tau(x,w),\qquad q=\tau(\alpha_wx,z).
    \]
    Then $\alpha_wx=\phi_px$, and Definition~\ref{def:continuousTimeChange} gives
    \[
    \rho(x,p+q)=\rho(x,p)+\rho(\phi_px,q)=w+z.
    \]
    Applying the inverse map $\tau(x,\cdot)$ proves the cocycle identity for $\tau$. Therefore
    \[
    \alpha_z(\alpha_wx)=\phi_{\tau(\alpha_wx,z)}\phi_{\tau(x,w)}x=\phi_{\tau(x,w+z)}x=\alpha_{w+z}x.
    \]
    Joint continuity follows from that of $\phi$ and $\tau$. Finally, $\alpha_wx=\phi_{\tau(x,w)}x$ and $\phi_px=\alpha_{\rho(x,p)}x$ show that the two actions have the same $H_n$-orbits.

    The cocycle identity for $\beta$ follows by restricting the cocycle identity for $\tau$ to the line $\R v$.
\end{proof}
\medskip

The following lemma shows the existence of the asymptotic drift with respect to an ergodic measure under an integrability assumption.

\begin{Lemma}\label{lem:oneStepIntCri}
    Let $\beta:X\times\mathbb{R}\longrightarrow\h_n$ be a continuous additive cocycle over the ergodic flow $T_t$. Assume 
    \[
    M(x)\coloneq\sup_{0\leq r\leq 1}\norm{\beta(x,r)}\in L^1(\nu).
    \]
    Then, for $\nu$-a.e. $x$,
    \[
    \lim_{t\to+\infty}\frac{\beta(x,t)}{t}=\int_X\beta(x,1) d\nu(x).
    \]
\end{Lemma}
\begin{proof}
Define $H(x)\coloneq\int_0^1\beta(x,r)dr$. The cocycle identity gives
\[
\int_0^t\beta(T_sx,1)ds=\beta(x,t)+H(T_tx)-H(x).
\]
It follows from the pointwise ergodic theorem for the flow $T_t$ that 
\[
\frac{1}{t}\int_0^t\beta(T_sx,1)ds\longrightarrow\int_X \beta(x,1)d\nu.
\]
If $M_2(x)\coloneq\sup_{0\leq r\leq 2}\norm{\beta(x,r)}$, then 
\[
M_2(x)\leq M(x)+(M\circ T_1)(x)\in L^1(\nu).
\]
Writing $t=N+r$ and $0\leq r<1$, we have
\[
\norm{H(T_tx)}\leq 2M_2(T_Nx).
\]
Since $T_N$ preserves $\nu$ and $M_2\in L^1(\nu)$, the Borel-Cantelli lemma gives $M_2(T_Nx)/N\to 0$ for $\nu$-a.e. $x\in X$ and therefore
\[
\norm{H(T_tx)}/t\to 0,
\]
which completes the proof of Lemma~\ref{lem:oneStepIntCri}.
\end{proof}
\medskip

In particular, if $\beta(x,t)=\int_0^t\xi(T_sx)ds$ with $\xi\in L^1(\nu,\h_n)$, the ergodic theorem gives $\frac{\beta(x,t)}{t}\to\int_X\xi d\nu$ for $\nu$-a.e. $x\in X$.

\begin{Lemma}\label{lem:normalization}
    Given
    \[
    0\neq V=\sum_{j=2}^nv_jU_j,\qquad \v=(v_2,\ldots,v_n),\qquad \varrho=\norm{\v},
    \]
    choose $O\in SO(n-1)$ with $\v O^{-1}=\varrho(1,0,\ldots,0)$, let $\lambda_V=\varrho^{1/n}$ and define
    \[
    g_V=\begin{pmatrix}
        \lambda_V^{-(n-1)} & 0 \\
        0 & \lambda_V O
    \end{pmatrix}\in\SL_n(\R).
    \]
    Then $g_V\in N_G(H_n)$ and $\Ad_{g_V}V=U_2$.
\end{Lemma}
\begin{proof}
    Relative to $\R e_1\oplus\operatorname{span}\{e_2,\ldots,e_n\}$, the element $V$ has upper-right row $\v$. Conjugation by $\diag(c,A)$ sends that row to~$c\v A^{-1}$. If $c=\lambda_V^{-(n-1)}$ and $A=\lambda_VO$, then we have
    \[
    c\v A^{-1}=\lambda_V^{-n}\v O^{-1}=\lambda_V^{-n}\varrho(1,0,\ldots,0)=(1,0,\ldots,0).
    \]
    The determinant of $g_V$ is one, and block diagonal matrices of this form normalize $H_n$. This proves the lemma.
\end{proof}

Define
\[
\Gamma_V\coloneq g_V\Gamma g_V^{-1},\qquad X_V\coloneq G/\Gamma_V,
\]
and let
\[
\mathcal{C}_V:X\longrightarrow X_V,\qquad \mathcal{C}_V(h\Gamma)\coloneq g_Vhg_V^{-1}\Gamma_V.
\]
Define the conjugated action by
\[
\alpha_w^V\coloneq\mathcal{C}_V\alpha_w\mathcal{C}_V^{-1},\qquad w\in\R^{n-1}.
\]
Since $g_V\in N_G(H_n)$, the conjugated full action $\alpha_w^V$ is again a continuous time change of the algebraic $H_n$-action on $X_V$.

Set
\[
\nu_V\coloneq(\mathcal{C}_V)_*\nu,\qquad T_t^V\coloneq\mathcal{C}_VT_t\mathcal{C}_V^{-1},
\]
and define
\[
\beta_V(\mathcal{C}_Vx,t)\coloneq\Ad_{g_V}\beta(x,t).
\]
Then, $\nu_V$ is $T_t^V$-invariant and ergodic, and its asymptotic drift is $U_2$. We henceforth work in this normalized system and suppress the subscript~$V$.

\section{Compact Ratner Frames}\label{sec:cptRatnerFra}
In this section and \S\ref{sec:RigidityFullGRaFam} we use independent notation. Let $G_0$ be a connected semisimple real linear group, $\Gamma_0<G_0$ a lattice and $\mathfrak{u}<\mathfrak{g}_0=\operatorname{Lie}(G_0)$ an abelian Lie algebra 
of dimension $m$ consisting of nilpotent elements.
Let $N=\exp(\mathfrak{u})$ and $X_0=G_0/\Gamma_0$. Equip $G_0$ with a complete right-invariant Riemannian metric and denote its distance by $d_{G_0}$. Let $d_0$ be the induced quotient distance on $X_0$, and set
\[
\bar{d}_0=\min\{1,d_0\}.
\]
Let
\[
T_tx=\exp(\beta(x,t))x,\qquad \beta(x,t)\in\mathfrak{u},
\]
be a directional restriction of a continuous time change of the $N$-action. Let $\nu$ be a $T_t$-ergodic probability measure and assume that the asymptotic drift exists and equals $V\in\mathfrak{u}\setminus\{0\}$.

For $K\in\mathfrak{g}_0$, we let
\[
\Delta_{K,s}(x,t)=\beta(\exp(sK)x,t)-\beta(x,t).
\]

For a compact Ratner frame as in Definition \ref{def:CptRatnerFra}, we introduce two families of perturbations and their cocycle increments.
\begin{Definition}\label{def:twoPackets}
    Let $\cF=\{(K_i,d_i,A_i)\}_{i=1}^m$ be a compact Ratner frame, and let 
    $B=\sum_ib_iA_i\in\mathfrak{u}$. For any $u\in \R$ and $\varepsilon\in (0,1)$, we define 
    \[
    P_B(u)\coloneq\sum_{i=1}^mb_iu^{d_i}A_i, \qquad \text{and} \qquad Z_{\epsilon}(B)=\sum_{i=1}^mb_i\epsilon^{d_i}K_i.
    \]
    Furthermore, we define 
    \[
    g_{\epsilon,B}^{\mathrm{exp}}\coloneq\exp Z_{\epsilon}(B), \qquad \text{and} \qquad g_{\epsilon,B}^{\mathrm{ord}}\coloneq\exp\bigl(b_m\epsilon^{d_m}K_m\bigr)\cdots\exp\bigl(b_1\epsilon^{d_1}K_1\bigr),
    \]
    and
    \[
    \Delta_{\epsilon,B}^{\diamond}(x,t)\coloneq\beta(g_{\epsilon,B}^{\diamond}x,t)-\beta(x,t),
    \]
    where $\diamond\in\{\mathrm{exp},\mathrm{ord}\}$. 
\end{Definition}
We now state our two main regularity assumptions.

\begin{Definition}[Pointwise regularity]\label{def:PtReg}
    A compact Ratner frame is pointwise regular for $(T,\nu)$ if there are $C_i>0$, $s_i>0$, $\delta_i>0$, and $\gamma_i\geq0$ such that
    \[
    \sup_{x\in X_0}\norm{\Delta_{K_i,s}(x,t)}\leq C_i\abs{s}^{\delta_i}(1+\abs{t})^{\gamma_i},\qquad d_i\delta_i>\gamma_i,
    \]
    for $\abs{s}\leq s_i$ and $t\geq 0$.
\end{Definition}

\begin{Definition}[Correlation regularity]\label{def:RatnerCorReg}
    A compact Ratner frame is correlation regular for $(T,\nu)$ if there are $\chi>0$ and $\delta_i>0$ such that
    \begin{equation}\label{eq:kappaDef}
        \kappa\coloneq\min_{1\leq i\leq m}d_i\delta_i>1-\frac{1}{2}\min\{\chi,1\},
    \end{equation}
    and, for every $R>0$, there are $C_R>0$ and $\epsilon_R>0$ such that, for $0<\epsilon<\epsilon_R$ and $\norm{B}\leq R$, the map $t\mapsto\Delta_{\epsilon,B}^{\mathrm{exp}}(\cdot,t)$ is locally absolutely continuous as an $L^2(\nu,\mathfrak{u})$-valued map. Denoting its derivative by $\mathcal{D}_{\epsilon,B}$, one has
    \begin{equation}\label{eq:derCorReg}
        \abs{\int_{X_0}\langle\mathcal{D}_{\epsilon,B}(x,u),\mathcal{D}_{\epsilon,B}(x,w)\rangle d\nu(x)}\leq C_R\left(\sum_{i=1}^m\epsilon^{d_i\delta_i}\right)^2(1+\abs{u-w})^{-\chi}
    \end{equation}
    for $0<\epsilon<\epsilon_R$, almost every $u,w\geq 0$ and $\norm{B}\leq R$.
\end{Definition}

The goal of this section is to prove the following results.
\begin{Theorem}[Large centraliser from correlation regularity]\label{thm:graphProRootCor}
    Assume that a compact Ratner frame is correlation regular for $(T,\nu)$. Then there is a conull set $E\subset X_0$ such that for every $x\in E$, every $B\in\mathfrak{u}$ and $t\in\R$
    \[
    \exp(B)T_tx=T_t\exp(B)x.
    \]
\end{Theorem}

\begin{Theorem}[Large centralizer from pointwise regularity]\label{thm:graphProPtRoot}
    Assume that a compact Ratner frame is pointwise regular for $(T,\nu)$. Then there is a conull set $E\subset X_0$ such that for every $x\in E$, every $B\in\mathfrak{u}$ and $t\in\R$
    \[
    \exp(B)T_tx=T_t\exp(B)x.
    \]
\end{Theorem}

We will prove Theorems \ref{thm:graphProRootCor} and \ref{thm:graphProPtRoot} at the end of the section. 

\subsection{Preliminary lemmas}
We prove a few results that will be useful below.

\begin{Lemma}\label{lem:exactRelId}
    Let $g\in G_0$, $x\in X_0$ and $y=gx$. Then, for every $t\in\R$, the following identity holds in~$X_0$
    \[
    T_ty=\exp\bigl(\beta(gx,t)-\beta(x,t)\bigr)\Bigl(\exp\bigl(\beta(x,t)\bigr)g\exp\bigl(-\beta(x,t)\bigr)\Bigr)T_tx.
    \]
\end{Lemma}
\begin{proof}
Define
\[
A=\beta(gx,t),\qquad C=\beta(x,t).
\]
Since
\[
T_ty=\exp(A)gx\quad\text{and}\quad x=\exp(-C)T_tx,
\]
we have
\[
T_ty=\exp(A)g\exp(-C)T_tx=\exp(A-C)\Bigl(\exp(C)g\exp(-C)\Bigr)T_tx,
\]
where the last equality uses $A,C\in\mathfrak{u}$ and the abelian property of $\mathfrak{u}$.
\end{proof}

For $\chi>0$ and $t\geq 0$, define
\begin{equation}\label{eq:Mchi}
    M_{\chi}(t)\coloneq \begin{cases}
        (1+t)^{1-\chi/2},&0<\chi<1,\\
        (1+t)^{1/2}\log^{1/2}(2+t),&\chi=1,\\
        (1+t)^{1/2},&\chi>1.
    \end{cases}    
\end{equation}

\begin{Lemma}\label{lem:covInt}
    Let $(Y,\nu)$ be a standard probability space, $\mathcal{H}$ a finite-dimensional real inner product space, and let
    \[
    F:[0,\infty)\longrightarrow L^2(\nu,\mathcal{H})
    \]
    be locally absolutely continuous with $F(0)=0$. Let $D(\cdot,t)$ denote its derivative as an $L^2(\nu,\mathcal{H})$-valued map, defined for almost every $t\geq0$, and let $(y,t)\mapsto D(y,t)$ be the chosen jointly measurable representative.
    Suppose that, for some $A\geq 0$ and $\chi>0$,
    \[
    \abs{\int_Y\langle D(y,u),D(y,w)\rangle_{\mathcal{H}} d\nu(y)}\leq A^2(1+\abs{u-w})^{-\chi}
    \]
    for almost every $u,w\geq 0$. Then there is a constant $C_{\chi}>0$ such that
    \[
    \norm{F(t)}_{L^2(\nu,\mathcal{H})}\leq C_{\chi}A M_{\chi}(t)\qquad t\geq 0.
    \]
\end{Lemma}
\begin{proof}
    By the fundamental theorem of calculus for Banach-valued absolutely continuous maps,
    \[
    F(t)=\int_0^tD(\cdot,u)du
    \]
    in $L^2(\nu,\mathcal{H})$, where the integral is understood in the Bochner sense. Since
    \[
    \begin{aligned}
        \int_0^t\int_0^t\abs{\langle D(\cdot,u),D(\cdot,w)\rangle_{L^2(\nu,\mathcal{H})}}dudw\leq\left(\int_0^t\norm{D(\cdot,u)}_{L^2(\nu,\mathcal{H})}du\right)^2<\infty,
    \end{aligned}
    \]
    continuity and bilinearity of the Hilbert space inner product give 
    \[
    \begin{aligned}
        \norm{F(t)}_{L^2(\nu)}^2=&\int_0^t\int_0^t\langle D(\cdot,u),D(\cdot,w)\rangle_{L^2(\nu,\mathcal{H})}dudw\\
        =&\int_0^t\int_0^t\int_Y\langle D(y,u),D(y,w)\rangle_{\mathcal{H}} d\nu(y)dudw\\
        \leq& A^2\int_0^t\int_0^t(1+\abs{u-w})^{-\chi}dudw.
    \end{aligned}
    \]
    Finally,
    \[
    \begin{aligned}
    \int_0^t\int_0^t(1+\abs{u-w})^{-\chi}dudw=2\int_0^t(t-r)(1+r)^{-\chi}dr&\ll_\chi M_{\chi}(t)^2.
    \end{aligned}
    \]
    Taking square roots proves the lemma.
\end{proof}

\begin{Lemma}\label{lem:tightGraphLem}
    Let $S:X_0\to X_0$ be continuous and $\{\eta_j\}_{j\geq 1}$ a sequence of probability measures on $X_0\times X_0$. Suppose that their first marginals converge weakly to a probability measure $\nu$ and
    \[
    \int_{X_0\times X_0}\bar{d}_0(y,Sx)d\eta_j(x,y)\longrightarrow0.
    \]
    Then
    \[
    \eta_j\longrightarrow(\id,S)_*\nu
    \]
    weakly. In particular, $\{\eta_j\}_{j\geq 1}$ is tight.
\end{Lemma}
\begin{proof}
    Let $\mu_j=(\pi_1)_*\eta_j$ be the first marginal. Fix $\epsilon\in(0,1)$. Since $X_0$ is a locally compact second countable metric space, the probability measure $\nu$ is tight. Choose a compact set $C\subset X_0$ such that $\nu(C)>1-\frac{\epsilon}{4}$, and then choose an open neighborhood $U\supset C$ with compact closure. By the Portmanteau theorem,
    \[
    \liminf_{j\to\infty}\mu_j(U)\geq\nu(U)\geq\nu(C).
    \]
    Hence, for all sufficiently large $j$
    \[
    \mu_j(U)>1-\frac{\epsilon}{2}.
    \]
    
    Choose an open set $W\subset X_0$ with compact closure such that 
    \[
    S(\bar{U})\subset W.
    \]
    Since $S(\bar{U})$ is compact and $X_0\setminus W$ is closed and disjoint from it, there exists $a>0$ such that
    \[
    \bar{d}_0(y,Sx)\geq a
    \]
    whenever $x\in\bar{U}$ and $y\in X_0\setminus W$. Therefore
    \[
    \eta_j\bigl(U\times(X_0\setminus W)\bigr)\leq a^{-1}\int_{X_0\times X_0}\bar{d}_0(y,Sx)d\eta_j(x,y).
    \]
    For all sufficiently large $j$, the RHS is smaller than $\epsilon/2$. Consequently,
    \[
    \eta_j(\bar{U}\times\bar{W})\geq\eta_j(U\times W)=\mu_j(U)-\eta_j\bigl(U\times(X_0\setminus W)\bigr)>1-\epsilon.
    \]
    Thus the tail of the sequence is tight. Since each of the finitely many remaining probability measures is individually tight, adjoining finitely many compact sets controlling those measures proves tightness of the whole sequence.

    Since $X_0\times X_0$ is Polish, Prokhorov's theorem implies that $\{\eta_j\}_{j\geq1}$ is relatively compact in the weak topology. Let $\eta_{j_k}\to\eta$ weakly. Continuity of the first projection gives $(\pi_1)_*\eta=\nu$. 
    Since $\bar{d}_0=\min(1,d_0)$, the function $(x,y)\mapsto \bar{d}_0(y,Sx)$ is bounded and continuous. Hence
    \[
    \int_{X_0\times X_0}\bar{d}_0(y,Sx)d\eta(x,y)=0.
    \]
    Therefore $\eta$ is concentrated on the graph of $S$. Since its first marginal is $\nu$,
    \[
    \eta=(\id,S)_*\nu.
    \]

    Thus every weakly convergent subsequence has the same limit. By relative compactness, every subsequence admits a weakly convergent subsubsequence, and therefore 
    \[
    \eta_j\longrightarrow(\id,S)_*\nu
    \]
    weakly.
\end{proof}

\subsection{Large centralizers}

Using Lemma \ref{lem:tightGraphLem}, we can formulate conditions that allow us to find elements of $G_0$ that commute with the flow $(T_t)_{t\in\R}$ almost everywhere. 

\begin{Proposition}\label{prop:scaleGraCri}
    Under the standing assumptions of this section, let $h\in G_0$, let 
    \[
    Q:[1/2,3/2]\to G_0
    \]
    be continuous with $Q(1)=h$, and, for $0<\eta<1$, let $g_{\eta}\in G_0$. Define
    \[
    \Delta_{\eta}(x,t)\coloneq\beta(g_{\eta}x,t)-\beta(x,t).
    \]
    Assume that 
    \begin{equation}\label{eq:scaleAssump1}
        \lim_{\eta\to0^+}\sup_{\frac{1}{2\eta}\leq t\leq \frac{3}{2\eta}}\norm{\Delta_{\eta}(\cdot,t)}_{L^2(\nu)}=0,
    \end{equation}
    and that, whenever $\eta_j\to0$ and $x_j\in X_0$ satisfy
    \begin{equation}\label{eq:scaleAssump2}
        \sup_{u\in[1/2,3/2]}\norm{\frac{\eta_j}{u}\beta\left(x_j,\frac{u}{\eta_j}\right)-V}\longrightarrow0,
    \end{equation}
    we have
    \begin{equation}\label{eq:scaleAssump3}
        \sup_{u\in[1/2,3/2]}d_{G_0}\left(\exp\left(\beta\left(x_j,\frac{u}{\eta_j}\right)\right)g_{\eta_j}\exp\left(-\beta\left(x_j,\frac{u}{\eta_j}\right)\right),Q(u)\right)\longrightarrow0.
    \end{equation}
    Then there is a conull set $E_h\subset X_0$ such that
    \[
    hT_{\sigma}x=T_{\sigma}hx,\qquad x\in E_h,\quad \sigma\in\R.
    \]
\end{Proposition}
\begin{proof}
    Define $Sx=hx$. Let $\{f_k\}_{k\geq1}\subset C_b(X_0)$ be a countable convergence determining family, which exists since $X_0$ is Polish. For $R>0$, define 
    \[
    \mathcal{A}_Rf(x)\coloneq\frac{1}{R}\int_0^Rf(T_tx)dt.
    \]

    \medskip
    \noindent
    \emph{Uniform genericity and the choice of base points.} Choose a sequence 
    \[
    0<\theta_j<\frac{1}{2},\qquad \theta_j\downarrow0.
    \]
    By the pointwise ergodic theorem, the assumed existence of the drift, and Egorov's theorem, there are measurable sets $\Omega_j\subset X_0$ with $\nu(\Omega_j)>3/4$, and numbers $R_j>0$ such that for $x\in \Omega_j$ and $R\geq R_j$,
    \begin{equation}\label{eq:UniGene}
        \begin{aligned}
            \max_{1\leq k\leq j}\abs{\mathcal{A}_Rf_k(x)-\int_{X_0}f_kd\nu}\leq \theta_j^2,\qquad \norm{\frac{\beta(x,R)}{R}-V}\leq\theta_j^2.
        \end{aligned}
    \end{equation}

    By the assumed mean square decay (cf.~\eqref{eq:scaleAssump1}), choose $\eta_j\in(0,1)$ so small that 
    \[
    \eta_j<\min\left(\frac{1}{2R_j},\frac{\theta_j}{j}\right),\qquad \sup_{\frac{1}{2\eta_j}\leq t\leq\frac{3}{2\eta_j}}\norm{\Delta_{\eta_j}(\cdot,t)}_{L^2(\nu)}\leq\frac{1}{j}.
    \]
    Define
    \[
    a_j\coloneq\frac{1-\theta_j}{\eta_j},\qquad b_j\coloneq\frac{1+\theta_j}{\eta_j},\qquad I_j=[a_j,b_j].
    \]
    Since $\theta_j\in(0,1/2)$, the intervals $I_j$ satisfy 
    \[
    I_j\subset\left[\frac{1}{2\eta_j},\frac{3}{2\eta_j}\right],\qquad a_j,b_j> R_j,\qquad \abs{I_j}=\frac{2\theta_j}{\eta_j}>2j.
    \]
    Moreover, the choice of $\eta_j$ gives
    \begin{equation}\label{eq:L2Win}
        \sup_{t\in I_j}\norm{\Delta_{\eta_j}(\cdot,t)}_{L^2(\nu)}\leq\frac{1}{j}.
    \end{equation}

    Define
    \[
    \mathcal{R}_j(x)\coloneq\frac{1}{\abs{I_j}}\int_{I_j}\norm{\Delta_{\eta_j}(x,t)}^2dt.
    \]
    By Tonelli's theorem and \eqref{eq:L2Win},
    \[
    \begin{aligned}
        m_j\coloneq&\int_{X_0}\mathcal{R}_jd\nu=\frac{1}{\abs{I_j}}\int_{I_j}\norm{\Delta_{\eta_j}(\cdot,t)}_{L^2(\nu)}^2dt
        \leq\sup_{t\in I_j}\norm{\Delta_{\eta_j}(\cdot,t)}_{L^2(\nu)}^2\leq\frac{1}{j^2}.
    \end{aligned}
    \]
    Markov's inequality, with the assertion trivial when $m_j=0$, gives 
    \[
    \nu\left(\{x:\mathcal{R}_j(x)>4m_j\}\right)\leq1/4.
    \]
    Hence we may choose
    \[
    x_j\in \Omega_j\cap\{x:\mathcal{R}_j(x)\leq 4m_j\}.
    \]
    Define
    \[
    q_j\coloneq \mathcal{R}_j(x_j),\qquad  y_j\coloneq g_{\eta_j}x_j.
    \]
    Then 
    \[
    q_j\leq 4m_j\longrightarrow0.
    \]

    Since $x_j\in \Omega_j$, equation~\eqref{eq:UniGene} gives \eqref{eq:scaleAssump2}. Hence \eqref{eq:scaleAssump3}, the continuity of $Q$ and $\eta_jt\in[1-\theta_j,1+\theta_j]$ for $t\in I_j$ imply
    \begin{equation}\label{eq:uniPolyGra}
        r_j\coloneq\sup_{t\in I_j}d_{G_0}\left(\exp\bigl(\beta(x_j,t)\bigr)g_{\eta_j}\exp\bigl(-\beta(x_j,t)\bigr),h\right)\longrightarrow0.
    \end{equation}

    \medskip
    \noindent
    \emph{Concentration near the algebraic graph.}
    Define
    \[
    \zeta_j\coloneq\frac{1}{\abs{I_j}}\int_{I_j}\delta_{(T_tx_j,T_ty_j)}dt.
    \]
    
    By the definition of the quotient metric,
    \begin{equation}\label{eq:quoMetIneq}
        d_0(az,bz)\leq d_{G_0}(a,b),\qquad a,b\in G_0,\quad z\in X_0.
    \end{equation}
    This estimate, together with the path $s\mapsto\exp(sZ)$, gives a constant $C_0>0$ such that
    \begin{equation}\label{eq:partExpEst}
        \bar{d}_0\bigl(\exp(Z)z,z\bigr)\leq C_0\min(1,\norm{Z}),\qquad  Z\in\mathfrak{u},\ z\in X_0.
    \end{equation}
    Lemma~\ref{lem:exactRelId}, the triangle inequality, \eqref{eq:quoMetIneq}, \eqref{eq:uniPolyGra}, and \eqref{eq:partExpEst} give, for~$t\in~I_j$,
    \[
    \bar{d}_0\bigl(T_ty_j,ST_tx_j\bigr)\leq r_j+C_0\min\bigl(1,\norm{\Delta_{\eta_j}(x_j,t)}\bigr).
    \]
    Consequently, by Cauchy-Schwarz,
    \[
    \begin{aligned}
        \int_{X_0\times X_0}\bar{d}_0(y,Sx)d\zeta_j(x,y)
        \leq&r_j+C_0\left(\frac{1}{\abs{I_j}}\int_{I_j}\norm{\Delta_{\eta_j}(x_j,t)}^2dt\right)^{1/2}\\
        =&r_j+C_0q_j^{1/2}\longrightarrow0.
    \end{aligned}
    \]

    \medskip
    \noindent
    \emph{The first marginal.}
    Let
    \[
    \lambda_j\coloneq (\pi_1)_*\zeta_j.
    \]
    For $1\leq k\leq j$,
    \[
    \begin{aligned}
        \int_{X_0}f_kd\lambda_j=\frac{1}{b_j-a_j}\int_{a_j}^{b_j}f_k(T_tx_j)dt=\frac{b_j\mathcal{A}_{b_j}f_k(x_j)-a_j\mathcal{A}_{a_j}f_k(x_j)}{b_j-a_j}.
    \end{aligned}
    \]
    Since $a_j,b_j\geq R_j$, equation~\eqref{eq:UniGene} gives 
    \[
    \begin{aligned}
        \abs{\int_{X_0}f_kd\lambda_j-\int_{X_0} f_kd\nu}\leq\frac{a_j+b_j}{b_j-a_j}\theta_j^2=\theta_j.
    \end{aligned}
    \]
    Thus, for every fixed $k$, the integrals converge once $j\geq k$. Since $\{f_k\}_{k\geq 1}$ is convergence determining, $\lambda_j\to\nu$ weakly. Lemma~\ref{lem:tightGraphLem} now gives
    \[
    \zeta_j\longrightarrow(\id,S)_*\nu.
    \]

    \medskip
    \noindent
    \emph{Diagonal invariance.}
    Fix $\sigma\in\R$. For $F\in C_b(X_0\times X_0)$, 
    \[
    \begin{aligned}
        \abs{\int_{X_0\times X_0}Fd(T_{\sigma}\times T_{\sigma})_*\zeta_j-\int_{X_0\times X_0}Fd\zeta_j}
        \leq&\frac{\abs{(I_j+\sigma)\triangle I_j}}{\abs{I_j}}\norm{F}_{\infty}\\
        \leq&\frac{2\abs{\sigma}}{\abs{I_j}}\norm{F}_{\infty}\longrightarrow0.
    \end{aligned}
    \]
    Hence $(T_{\sigma}\times T_{\sigma})_*(\id,S)_*\nu=(\id,S)_*\nu$. Because $T_{\sigma}$ preserves $\nu$, 
    \[
    (\id,T_{\sigma}ST_{-\sigma})_*\nu= (T_{\sigma}\times T_{\sigma})_*(\id,S)_*\nu=(\id,S)_*\nu.
    \]
    Testing this equality against the bounded continuous function
    \[
    \Phi(x,y)\coloneq\bar{d}_0(y,Sx)
    \]
    gives
    \[
    \begin{aligned}
         &\int_{X_0}\bar{d}_0\bigl(T_{\sigma}ST_{-\sigma}x,Sx\bigr)d\nu(x)=\int_{X_0\times X_0}\Phi d(\id,T_{\sigma}ST_{-\sigma})_*\nu\\
         =&\int_{X_0\times X_0}\Phi d(\id,S)_*\nu=\int_{X_0}\bar{d}_0(Sx,Sx)d\nu(x)=0.
    \end{aligned}
    \]
    Therefore $T_{\sigma}ST_{-\sigma}x=Sx$ for $\nu$-a.e. $x$. By $T_{\sigma}$-invariance of $\nu$, this implies $T_{\sigma}Sx=ST_{\sigma}x$ almost everywhere. Intersecting the corresponding conull sets over $\sigma\in\Q$ and using the continuity in $\sigma$ proves the assertion.
    \end{proof}

Under suitable assumptions, the previous proposition can be applied to all elements $h\in N$.

\begin{Lemma}\label{lem:polyPacLim}
    Fix a compact Ratner frame
    \[
    \mathcal{F}=\{(K_i,d_i,A_i)\}_{i=1}^m
    \]
    for $(\mathfrak{u},V)$,
    and let $B=\sum_ib_iA_i\in\mathfrak{u}$. 
    Let $\{\epsilon_j\}_{j\geq1}$ be a sequence of positive numbers with $\epsilon_j\to0$, and $J\subset(0,\infty)$ a compact set. Assume that $\{x_j\}_{j\geq1}\subset X_0$ satisfy
    \[
    \lim_{j\to\infty}\sup_{u\in J}\norm{\frac{\epsilon_j}{u}\beta\left(x_j,\frac{u}{\epsilon_j}\right)-V}=0.
    \]
    Then
    \[
    \Ad_{\exp\left(\beta(x_j,u/\epsilon_j)\right)}Z_{\epsilon_j}(B)\longrightarrow P_B(u)
    \]
    uniformly for $u\in J$. Moreover, for $\diamond\in\{\mathrm{exp},\mathrm{ord}\}$, as $j\to\infty$,
    \[
    \exp \left(\beta(x_j,u/\epsilon_j)\right)g_{\epsilon_j,B}^{\diamond}\exp\left(-\beta(x_j,u/\epsilon_j)\right)\longrightarrow \exp(P_B(u)) 
    \]
    uniformly for $u\in J$.
\end{Lemma}
\begin{proof}
    Define 
    \[
    V_j(u)\coloneq\frac{\epsilon_j}{u}\beta\left(x_j,\frac{u}{\epsilon_j}\right).
    \]
    By assumption, $V_j\to V$ uniformly on $J$. By definition of compact Ratner frame, we can write the finite expansion
    \begin{equation}\label{eq:polyDeco}
        \begin{aligned}
            \epsilon_j^{d_i}\Ad_{\exp\beta(x_j,u/\epsilon_j)}K_i
            =\sum_{\ell=0}^{d_i}\frac{u^{\ell}\epsilon_j^{d_i-\ell}}{\ell!}\ad_{V_j(u)}^{\ell}K_i.
        \end{aligned}
    \end{equation}
    The vectors $V_j(u)$ remain in a fixed bounded set, uniformly on $J$. The terms with $\ell<d_i$ therefore tend uniformly to zero, while the last term converges uniformly to $u^{d_i}\ad_V^{d_i}K_i/d_i!=u^{d_i}A_i$. Summing with coefficients $b_i$ proves the first assertion and, by continuity of the exponential, the assertion for $g_{\epsilon_j,B}^{\mathrm{exp}}$.

    For $g_{\epsilon_j,B}^{\mathrm{ord}}$, conjugate each factor separately. The limiting product is
    \[
    \begin{aligned}
    \exp\bigl(b_mu^{d_m}A_m\bigr)\cdots\exp\bigl(b_1u^{d_1}A_1\bigr)=\exp(P_B(u)),
    \end{aligned}
    \]
    since the $A_i$ belong to the abelian algebra $\mathfrak{u}$. All convergences are uniform on $J$.
\end{proof}

\begin{Proposition}\label{prop:regPcGraph}
Under the standing assumptions of this section, fix a compact Ratner frame
\[
\mathcal{F}=\{(K_i,d_i,A_i)\}_{i=1}^m
\]
for $(\mathfrak{u},V)$ and fix $\diamond\in\{\mathrm{exp},\mathrm{ord}\}$. 
    Assume that, for every $B\in\mathfrak{u}$,
    \[
    \lim_{\epsilon\to0^+}\sup_{\frac{1}{2\epsilon}\leq t\leq \frac{3}{2\epsilon}}\norm{\Delta_{\epsilon,B}^{\diamond}(\cdot,t)}_{L^2(\nu)}=0.
    \]
    Then there is a conull set $E\subset X_0$ such that
    \[
    T_{\sigma}\exp(B)x=\exp(B)T_{\sigma}x,\qquad \forall x\in E,\ \forall B\in\mathfrak{u},\ \forall\sigma\in\R.
    \]
\end{Proposition}
\begin{proof}
    Fix $B\in\mathfrak{u}$ and set
    \[
    h=\exp(B),\qquad Q_B(u)=\exp(P_B(u)),\qquad g_{\epsilon}=g_{\epsilon,B}^{\diamond}.
    \]
    The mean square hypothesis of  Proposition~\ref{prop:scaleGraCri} is the assumed decay, while its second hypothesis follows from Lemma~\ref{lem:polyPacLim}. Since $P_B(1)=B$, that proposition gives a conull set $E_B$ on which 
    \[
    T_{\sigma}\exp(B)x=\exp(B)T_{\sigma}x,\qquad \sigma\in\R.
    \]
    Apply this conclusion to a countable dense additive subgroup $\mathfrak{u}_{\Q}\subset\mathfrak{u}$ and intersect the corresponding conull sets. Continuity in $B$ extends the identity to every $B\in\mathfrak{u}$.
\end{proof}

\subsection{Proofs of  Theorems \ref{thm:graphProRootCor} and \ref{thm:graphProPtRoot}}

In order to prove Theorems \ref{thm:graphProRootCor} and \ref{thm:graphProPtRoot} we need to verify the assumptions of Proposition \ref{prop:regPcGraph}. This is the content of the next two lemmas.

\begin{Lemma}\label{lem:meanSquContr}
    Assume that the compact Ratner frame $\mathcal{F}=\{(K_i,d_i,A_i)\}_{i=1}^m$ is correlation regular. 
    For every $R>0$, there is $C'_R>0$ such that
    \[
    \norm{\Delta_{\epsilon,B}^{\mathrm{exp}}(\cdot,t)}_{L^2(\nu)}\leq C'_R\left(\sum_{i=1}^m\epsilon^{d_i\delta_i}\right)M_{\chi}(t)
    \]
    for every $B\in \mathfrak{u}$ with $\norm{B}\leq R$, $
    0<\epsilon<\epsilon_R$ and $t\geq 0$.
    In particular, 
    \[
    \lim_{\epsilon\to0^+}\sup_{\norm{B}\leq R}\sup_{\frac{1}{2\epsilon}\leq t\leq \frac{3}{2\epsilon}}\norm{\Delta_{\epsilon,B}^{\mathrm{exp}}(\cdot,t)}_{L^2(\nu)}=0.
    \]
\end{Lemma}
\begin{proof}
    Since $\Delta_{\epsilon,B}^{\mathrm{exp}}(\cdot,0)=0$ and $t\mapsto \Delta_{\epsilon,B}^{\mathrm{exp}}(\cdot,t)$ is locally absolutely continuous in $L^2(\nu,\mathfrak{u})$, Lemma~\ref{lem:covInt} applies with
    \[
    Y=X_0,\qquad \mathcal{H}=\mathfrak{u},\qquad F(t)=\Delta_{\epsilon,B}^{\mathrm{exp}}(\cdot,t),\qquad D(x,t)=\mathcal{D}_{\epsilon,B}(x,t),
    \]
    and
    \[
    A=C_R^{1/2}\sum_{i=1}^m\epsilon^{d_i\delta_i}.
    \]
    This proves the first assertion.

    \medskip
    
    Define $\kappa\coloneq\min_{1\leq i\leq m}d_i\delta_i$. For $0<\epsilon\leq 1$,
    \[
    \sum_{i=1}^m\epsilon^{d_i\delta_i}\leq m\epsilon^{\kappa}.
    \]
    Consequently, uniformly for 
    \[
    \norm{B}\leq R,\qquad \frac{1}{2\epsilon}\leq t\leq \frac{3}{2\epsilon},
    \]
    we have
    \[
     \norm{\Delta_{\epsilon,B}^{\mathrm{exp}}(\cdot,t)}_{L^2(\nu)}\ll_R\begin{cases}
        \epsilon^{\kappa-1+\chi/2},&0<\chi<1,\\
        \epsilon^{\kappa-1/2}\log^{1/2}(2+\epsilon^{-1}),&\chi=1,\\
        \epsilon^{\kappa-1/2},&\chi>1.
    \end{cases}
    \]
    The first and third expressions tend to zero because
    \[
    \kappa>1-\frac{1}{2}\min\{\chi,1\}.
    \]
    If $\chi=1$, then $\kappa-\frac{1}{2}>0$, and
    \[
    \epsilon^{\kappa-1/2}\log^{1/2}(2+\epsilon^{-1})\longrightarrow 0.
    \]
    This proves the second assertion.
\end{proof}

\begin{Lemma}\label{lem:pointPacContr}
Assume that the compact Ratner frame $\mathcal{F}=\{(K_i,d_i,A_i)\}_{i=1}^m$ is pointwise regular, and let $B=\sum_ib_iA_i \in \mathfrak{u}$.
    Let $C_i,s_i,\delta_i,\gamma_i$ be the constants in Definition~\ref{def:PtReg}. Then, for all sufficiently small $\epsilon>0$,
    \[
    \sup_{\substack{x\in X_0 \\1/(2\epsilon)\leq t\leq 3/(2\epsilon)}}\norm{\Delta_{\epsilon,B}^{\mathrm{ord}}(x,t)}\leq\sum_{i=1}^mC_i(5/2)^{\gamma_i}\abs{b_i}^{\delta_i}\epsilon^{d_i\delta_i-\gamma_i}.
    \]
    In particular,
    \[
    \sup_{\frac{1}{2\epsilon}\leq t\leq \frac{3}{2\epsilon}}\norm{\Delta_{\epsilon,B}^{\mathrm{ord}}(\cdot,t)}_{L^2(\nu)}\longrightarrow0.
    \]
\end{Lemma}
\begin{proof}
Define 
\[
h_{i,\epsilon}\coloneq \exp\bigl(b_i\epsilon^{d_i}K_i\bigr).
\]
For $x\in X_0$, define
\[
x_0=x,\qquad x_i=h_{i,\epsilon}x_{i-1}\quad (1\leq i\leq m).
\]
Since $g_{\epsilon,B}^{\mathrm{ord}}=\exp\bigl(b_m\epsilon^{d_m}K_m\bigr)\cdots\exp\bigl(b_1\epsilon^{d_1}K_1\bigr)$, we have
\[
x_m=g_{\epsilon,B}^{\mathrm{ord}}x.
\]
Therefore
\[
\begin{aligned}
    \Delta_{\epsilon,B}^{\mathrm{ord}}(x,t)=\beta(x_m,t)-\beta(x_0,t)=&\sum_{i=1}^m\bigl(\beta(x_i,t)-\beta(x_{i-1},t)\bigr)\\
    =&\sum_{i=1}^m\Delta_{K_i,b_i\epsilon^{d_i}}(x_{i-1},t).
\end{aligned}
\]

Choose $\epsilon\in(0,1)$ sufficiently small that
\[
\abs{b_i}\epsilon^{d_i}\leq s_i \qquad (1\leq i\leq m).
\]
Since the pointwise estimates are uniform in the base point,
\[
\begin{aligned}
    \norm{\Delta_{\epsilon,B}^{\mathrm{ord}}(x,t)}\leq\sum_{i=1}^m\norm{\Delta_{K_i,b_i\epsilon^{d_i}}(x_{i-1},t)}\leq\sum_{i=1}^mC_i\abs{b_i}^{\delta_i}\epsilon^{d_i\delta_i}(1+t)^{\gamma_i}.
\end{aligned}
\]
If $\frac{1}{2\epsilon}\leq t\leq\frac{3}{2\epsilon}$, then $(1+t)^{\gamma_i}\leq\left(5/2\right)^{\gamma_i}\epsilon^{-\gamma_i}$.
Therefore
\begin{equation}\label{eq:supXTlimit}
    \begin{aligned}
    \sup_{\substack{x\in X_0\\1/(2\epsilon)\leq t\leq 3/(2\epsilon)}}\norm{\Delta_{\epsilon,B}^{\mathrm{ord}}(x,t)}\leq\sum_{i=1}^{m}C_i\left(5/2\right)^{\gamma_i}\abs{b_i}^{\delta_i}\epsilon^{d_i\delta_i-\gamma_i}.
\end{aligned}
\end{equation}
Since $d_i\delta_i-\gamma_i>0$ for every $1\leq i\leq m$, the RHS of \eqref{eq:supXTlimit} tends to zero. Since $\nu$ is a probability measure, the same conclusion holds in~$L^2(\nu)$.
\end{proof}

We can now prove the two main results of this section.

\begin{proof}[Proof of Theorem \ref{thm:graphProRootCor}]
Take $\diamond=\mathrm{exp}$. Then, Lemma~\ref{lem:meanSquContr}, together with the strict exponent inequality in Definition~\ref{def:RatnerCorReg},  verifies the hypothesis of Proposition~\ref{prop:regPcGraph}, which proves the result.   
\end{proof}

\begin{proof}[Proof of Theorem \ref{thm:graphProPtRoot}]
    Take $\diamond=\mathrm{ord}$. Then, Lemma~\ref{lem:pointPacContr} verifies the hypothesis of Proposition~\ref{prop:regPcGraph}, which proves the result.
\end{proof}

\section{A large centralizer implies rigidity}\label{sec:RigidityFullGRaFam}

The goal of this section is to prove that if the directional flow commutes with the entire algebraic $N$-action, then it is itself algebraic. 

We continue with the notation of \S\ref{sec:cptRatnerFra}. For a closed finite-volume homogeneous orbit $Lx\subset X_0$, write $m_{Lx}$ for its normalized Haar probability. For $x\in X_0$, define
\[
\Lambda_x\coloneq\{B\in\mathfrak{u}:\exp(B)x=x\}.
\]
\begin{Lemma}\label{lem:algStab}
    For every $x\in X_0$, the subgroup $\Lambda_x<\mathfrak{u}$ is discrete. Moreover,
    \[
    \Lambda_{\exp(C)x}=\Lambda_x\qquad C\in\mathfrak{u}.
    \]
    If $\Gamma_0$ is cocompact, then $\Lambda_x=\{0\}$ for every $x\in X_0$.
\end{Lemma}
\begin{proof}
    Recall that $N=\exp(\mathfrak{u})$ and that $\exp:\mathfrak{u}\to N$ is a diffeomorphism. If $x=g\Gamma_0$, then
    \[
    \exp(\Lambda_x)=\{n\in N:nx=x\}=N\cap g\Gamma_0g^{-1}.
    \]
    The last group is discrete because $g\Gamma_0g^{-1}$ is discrete. Hence $\Lambda_x$ is discrete. The second assertion follows from the abelianity of $N$.
    
    \medskip

    Suppose now that $\Gamma_0$ is cocompact. Let $x=g\Gamma_0$ and suppose, toward a contradiction, that 
    \[
    0\neq B\in\Lambda_x.
    \]
    By the definition of $\Lambda_x$, $\exp(B)g\Gamma_0=g\Gamma_0$ and hence $g^{-1}\exp(B)g\in\Gamma_0$. Define 
    \[
    Y\coloneq\Ad_{g^{-1}}B,\qquad u_B\coloneq g^{-1}\exp(B)g=\exp(Y)\in\Gamma_0.
    \]
    Since $B$ is a nonzero nilpotent element, so is $Y$. Hence $u_B$ is a nontrivial unipotent element of $\Gamma_0$.

    By the real Jacobson-Morozov theorem \cite{Jacobson1951Sl2}, there is $H\in\mathfrak{g}_0$ such that
    \[
    [H,Y]=2Y.
    \]
    For $a_s=\exp(-sH)$, we have
    \[
    a_su_Ba_s^{-1}=\exp(e^{-2s}Y)\longrightarrow e\qquad \text{as }s\to+\infty.
    \]

    Since $G_0/\Gamma_0$ is compact, there is a compact set $\mathcal{C}\subset G_0$ such that $G_0=\mathcal{C}\Gamma_0$. For every $n\geq 1$, write
    \[
    a_n=c_n\gamma_n,\qquad c_n\in\mathcal{C},\qquad \gamma_n\in\Gamma_0.
    \]
    After passing to a subsequence, $c_n\to c\in \mathcal{C}$. Therefore
    \[
    \gamma_nu_B\gamma_n^{-1}=c_n^{-1}(a_nu_Ba_n^{-1})c_n\longrightarrow e.
    \]
    Every $\gamma_nu_B\gamma_n^{-1}$ is a nontrivial element of $\Gamma_0$. This contradicts the discreteness of $\Gamma_0$. Hence $\Lambda_x=\{0\}$.
\end{proof}

\begin{Theorem}[Large centralizer implies rigidity]\label{thm:fullGraRigid}
    Assume that there is a conull set $E\subset X_0$ such that for every $x\in E$, every $B\in\mathfrak{u}$ and every $t\in\R$,
    \[
    \exp(B)T_tx=T_t\exp(B)x.
    \]
    Then
    \[
    \beta(x,t)=tV,\quad \text{$\nu$-a.e. $x$, every $t\in\R$.}
    \]
    Consequently, $T_tx=\exp(tV)x$ modulo $\nu$, and $\nu$ is a homogeneous probability measure.
\end{Theorem}
\begin{proof}
Fix $x\in E$ and $t\in\R$. For every $B\in\mathfrak{u}$, the graph identity and the abelianity of $N$ give
\[
\begin{aligned}
    \exp\bigl(\beta(\exp(B)x,t)\bigr)\exp(B)x&=T_t\exp(B)x=\exp(B)T_tx
\\
&=\exp(\beta(x,t))\exp(B)x.
\end{aligned}
\]
Hence the continuous map
\[
\Psi_{x,t}:\mathfrak{u}\longrightarrow\mathfrak{u},\qquad \Psi_{x,t}(B)=\beta(\exp(B)x,t)-\beta(x,t),
\]
satisfies
\[
\Psi_{x,t}(B)\in\Lambda_{\exp(B)x}=\Lambda_x\qquad B\in\mathfrak{u}.
\]
By Lemma~\ref{lem:algStab}, $\Lambda_x$ is discrete. Since $\mathfrak{u}$ is connected, $\Psi_{x,t}$ is constant. As  $\Psi_{x,t}(0)=0$, it vanishes identically. Therefore
\begin{equation}\label{eq:betaNInv}
    \beta(\exp(B)x,t)=\beta(x,t)\qquad B\in\mathfrak{u},\ x\in E,\ t\in\R.
\end{equation}

Since
\[
T_sx=\exp(\beta(x,s))x,
\]
equation~\eqref{eq:betaNInv}, applied with $B=\beta(x,s)$, gives
\begin{equation}\label{eq:betaTInv}
    \beta(T_sx,t)=\beta(x,t)\qquad x\in E,\ s,t\in\R.
\end{equation}

\medskip

Let $E_0\subset E$ be the conull set on which the asymptotic drift is $V$. For $x\in E_0$, the cocycle identity and \eqref{eq:betaTInv} give
\[
\beta(x,s+t)=\beta(x,s)+\beta(T_sx,t)=\beta(x,s)+\beta(x,t).
\]
Thus $t\mapsto\beta(x,t)$ is continuous and additive, so $\beta(x,t)=tW(x)$ for some $W(x)\in\mathfrak{u}$. The drift assumption gives $W(x)=V$.

It follows that $T_t$ agrees modulo $\nu$ with the algebraic unipotent flow $\exp(tV)$. Hence $\nu$ is invariant and ergodic for the algebraic unipotent flow $x\mapsto\exp(tV)x$. Ratner's measure classification theorem \cites{ratner1990measure,ratner1991measure} gives a closed connected subgroup $L<G_0$ and a closed finite-volume orbit $Lz$ such that
\[
\exp(\R V)\leq L,\qquad \nu=m_{Lz}.
\]
\end{proof}

\begin{proof}[Proof of Theorem~\ref{thm:absCptRigid}]
    Apply Theorem~\ref{thm:graphProRootCor} or Theorem~\ref{thm:graphProPtRoot}, followed by Theorem~\ref{thm:fullGraRigid}.
\end{proof}

\section{The first row compact frame}\label{sec:firstRowCptFrame}
We now return to first row actions. The purpose of this section is purely algebraic: we realize the abstract frame by one quadratic compact root and $n-2$ degree one compact roots. The analytic hypotheses of the introduction then become exactly the pointwise and correlation regularity conditions of the abstract theorem.

We return to the notation of the introduction and work in normalized coordinates $V=U_2=E_{12}$. Recall that
\[
K_2=E_{21}-E_{12},\qquad K_j=E_{j2}-E_{2j}\quad \text{for }3\leq j\leq n.
\]
Define
\[
d_2=2,\qquad d_j=1\quad \text{for }3\leq j\leq n.
\]

\begin{Lemma}\label{lem:firstRowCptRatner}
    The family 
    \[
    (K_2,2,-U_2),\qquad (K_j,1,-U_j)\quad \text{for }3\leq j\leq n
    \]
    is a compact Ratner frame for $(\h_n,U_2)$.
\end{Lemma}
\begin{proof}
    Each $K_j$ is skew-symmetric. Hence
    \[
    \bigl(\exp(tK_j)\bigr)^{\mathsf T}\exp(tK_j)=\id,\qquad \det\exp(tK_j)=1,
    \]
    so $\exp(tK_j)\in \SO(n)$ for every $t\in\R$ and $\exp(\R K_j)$ is relatively compact. Simple computations give us
    \[
    \frac{1}{2}\ad_{U_2}^2K_2=-U_2,\qquad \ad_{U_2}K_j=-U_j\quad \text{for $3\leq j\leq n$.}
    \]
    For $j\geq 3$, we have $[\h_n,K_j]\subset\h_n$, and therefore all twofold $\h_n$-commutators vanish. For $K_2$, the first commutators belong to
    \[
    \operatorname{span}\{E_{11}-E_{22},E_{23},\ldots,E_{2n}\},
    \]
    the second commutators belong to $\h_n$, and the third commutators vanish. Finally, the vectors 
    \[
    \{-U_2,-U_3,\ldots,-U_n\}
    \]
    form a basis of $\h_n$, and hence the three conditions in Definition~\ref{def:CptRatnerFra} are satisfied.
    \end{proof}

    The quadratic direction alone cannot form a compact Ratner frame spanning $\h_n$. Indeed, since $V=E_{12}$, we have $V^2 = 0$ and, for any $Y=(Y_{i j})_{i,j}\in\sl_n(\R)$, we have $VYV=Y_{21}V$; therefore
    \[
    \ad_V^2Y=-2VYV\in\R V.
    \]
    Thus, every degree-two top coefficient lies in the drift line, whereas $K_3,\ldots,K_n$ generate the remaining first row directions at degree one.

    \begin{proof}[Proof of Theorem~\ref{thm:firstRowRigidity}]
        Work in normalized coordinates and use the frame in Lemma~\ref{lem:firstRowCptRatner}. Relative to $A_j=-U_j$, an element $B=\sum_{j=2}^nb_jU_j$ has coefficients $-b_j$. Hence
        \[
        g_{\epsilon,B}^{\mathrm{exp}}=\exp\left(-b_2\epsilon^2K_2-\epsilon\sum_{j=3}^nb_jK_j\right),\qquad 
        \min_{2\leq j\leq n}d_j\delta_j=\min\{2\delta_2,\delta_3,\ldots,\delta_n\}=\kappa_{\operatorname{fr}}.
        \]
        Definition~\ref{def:FirstRowRatnerCor} therefore gives the correlation regularity of this frame. Apply Theorem~\ref{thm:absCptRigid} and transport the conclusion back to the original coordinates.
    \end{proof}

    \begin{proof}[Proof of Corollary~\ref{cor:pointCptCri}]
        In normalized coordinates, the frame in Lemma~\ref{lem:firstRowCptRatner} has degrees $d_2=2$ and $d_j=1$ for $j\geq 3$. Definition~\ref{def:PointComRoot} is therefore exactly its pointwise regularity condition. Apply Theorem~\ref{thm:absCptRigid} and transport the conclusion back to the original coordinates.
    \end{proof}

\section{The sharper three-dimensional theorem}\label{sec:sharperThreeDim}

The full first row frame uses every acting direction. In dimension three, a single degree two compact root is sufficient: the acting algebra is two dimensional, and a separate rigidity argument shows that this one-parameter control is enough. This is the source of Theorem~\ref{thm:oneRootSL3rigidity}.

Assume $n=3$ and use normalized coordinates 
\[
V=U_2=E_{12}.
\]
Set
\[
K\coloneq E_{21}-E_{12}.
\]
Then
\begin{equation}\label{eq:KadRes}
    \frac{1}{2}\ad_V^2K=-V,\qquad \ad_V^3K=0.
\end{equation}
For $s,t\in\R$, write
\[
\Delta_s(x,t)\coloneq\Delta_s^{(2)}(x,t)=\beta(\exp(sK)x,t)-\beta(x,t).
\]

\begin{Proposition}\label{prop:oneRtCovCri}
Suppose that, for some $s_0,C>0$, every map 
\[
t\mapsto\Delta_s(\cdot,t),\qquad \abs{s}\leq s_0,
\]
is locally absolutely continuous as an $L^2(\nu,\h_3)$-valued map. Let $D_s(\cdot,t)$ denote its derivative, defined for almost every $t\geq 0$, and $D_s(x,t)$ the chosen jointly measurable representative. Assume that, for some $\chi>0$ and $\delta>0$,
\[
\abs{\int_X\left\langle D_s(x,u),D_s(x,w)\right\rangle d\nu(x)}\leq C\abs{s}^{2\delta}(1+\abs{u-w})^{-\chi}
\]  
for $\abs{s}\leq s_0$ and almost every $u,w\geq 0$. Then there is $C'>0$ such that
\[
\norm{\Delta_s(\cdot,t)}_{L^2(\nu)}\leq C'\abs{s}^{\delta} M_{\chi}(t)\qquad \abs{s}\leq s_0,\quad t\geq0,
\]
where $M_{\chi}$ is defined in \eqref{eq:Mchi}. 

Consequently, Definition~\ref{def:oneRootMeanSquare} holds whenever
\[
\delta>\frac{2-\chi}{4}\quad \text{for }0<\chi<1,\qquad \delta>\frac{1}{4}\quad \text{for }\chi\geq 1.
\]
\end{Proposition}
\begin{proof}
Fix $\abs{s}\leq s_0$. Since $\Delta_s(\cdot,0)=0$ and the map $t\mapsto \Delta_s(\cdot,t)$ is locally absolutely continuous in $L^2(\nu,\h_3)$,  Lemma~\ref{lem:covInt} applies with
\[
\begin{aligned}
    Y=X,\qquad \mathcal{H}=\h_3,\qquad F(t)=\Delta_s(\cdot,t),\\
    D(x,t)=D_s(x,t),\qquad A=C^{1/2}\abs{s}^{\delta}.
\end{aligned}
\]
This gives the stated estimate with $C'=C_{\chi}C^{1/2}$.

\medskip

If $0<\chi<1$, take $\gamma=1-\chi/2$. Then $2\delta>\gamma$ is equivalent to $\delta>(2-\chi)/4$. If $\chi>1$, take $\gamma=1/2$, so $2\delta>\gamma$ is equivalent to $\delta>1/4$. 

Finally, suppose $\chi=1$ and $\delta>1/4$. Choose $1/2<\gamma<2\delta$. Since 
\[
\log^{1/2}(2+t)\ll_{\gamma}(1+t)^{\gamma-1/2},
\]
our stated estimate gives
\[
\norm{\Delta_s(\cdot,t)}_{L^2(\nu)}\ll_{\gamma}\abs{s}^{\delta}(1+t)^{\gamma}.
\]
Since $\Delta_s=\Delta_s^{(2)}$, this verifies Definition~\ref{def:oneRootMeanSquare} in every case.
\end{proof}

\begin{Theorem}[Centralizer for flows in $\mathcal{M}_2(v,\nu)$]\label{thm:oneRtGraPro}
    Assume that the flow $(T_t)_{t\in\R}$ is in $\mathcal{M}_2(v,\nu)$ according to Definition~\ref{def:oneRootMeanSquare}. Then there is a conull set $E\subset X$ such that for every $x\in E$ and every $r,t\in \R$
    \[
    \exp(rV)T_tx=T_t\exp(rV)x.
    \]
\end{Theorem}
\begin{proof}
    Let $C,s_0,\delta,\gamma$ be the constants in Definition~\ref{def:oneRootMeanSquare}. Fix $r\in\R$. For $\eta\in(0,1)$, set
    \[
    g_{\eta,r}\coloneq\exp(-r\eta^2K),\qquad \Delta_{\eta,r}(x,t)\coloneq\beta(g_{\eta,r}x,t)-\beta(x,t),
    \]
    and define
    \[
    Q_r(u)\coloneq\exp(ru^2V),\qquad u\in[1/2,3/2].
    \]
    
    For $\eta>0$ sufficiently small that $\abs{r}\eta^2\leq s_0$, Definition~\ref{def:oneRootMeanSquare} gives
    \[
    \begin{aligned}
         \sup_{\frac{1}{2\eta}\leq t\leq \frac{3}{2\eta}}\norm{\Delta_{\eta,r}(\cdot,t)}_{L^2(\nu)}=& \sup_{\frac{1}{2\eta}\leq t\leq \frac{3}{2\eta}}\norm{\Delta_{-r\eta^2}(\cdot,t)}_{L^2(\nu)}\\
         \leq & C\abs{r}^{\delta}\eta^{2\delta}\left(\frac{5}{2\eta}\right)^{\gamma}\ll_r \eta^{2\delta-\gamma}\longrightarrow0.
    \end{aligned}
    \]
    The last step follows from $2\delta>\gamma$.
    
    Now let $\eta_j\to0^+$ and suppose that $x_j\in X$ satisfy
    \[
    \sup_{u\in[1/2,3/2]}\norm{\frac{\eta_j}{u}\beta\left(x_j,\frac{u}{\eta_j}\right)-V}\longrightarrow0.
    \]
    Apply Lemma~\ref{lem:polyPacLim} to the algebraic frame of Lemma~\ref{lem:firstRowCptRatner} with $B=rV$. Its coefficients relative to $(-V,-U_3)$ are $(-r,0)$, so
    \[
    Z_{\eta}(rV)=-r\eta^2K,\qquad P_{rV}(u)=ru^2V.
    \]
    The lemma therefore gives
    \[
    \exp\left(\beta\left(x_j,\frac{u}{\eta_j}\right)\right)g_{\eta_j,r}\exp\left(-\beta\left(x_j,\frac{u}{\eta_j}\right)\right)\longrightarrow \exp(ru^2V)=Q_r(u)
    \]
    uniformly for $u\in[1/2,3/2]$. Since $Q_r(1)=\exp(rV)$, Proposition~\ref{prop:scaleGraCri}, applied to the present normalized system, gives a conull set $E_r$ on which 
    \[
    \exp(rV)T_tx=T_t\exp(rV)x,\qquad t\in\R.
    \]
    Intersecting over $r\in\Q$ and using the continuity in $r$ gives one conull set on which the identity holds for all $r,t\in\R$.
\end{proof}

\medskip

Unlike Theorem~\ref{thm:fullGraRigid}, the next result assumes only that the directional flow commutes almost everywhere with the algebraic flow generated by the drift $V$.

\begin{Theorem}[Rigidity for $\mathfrak{h}_3$]\label{thm:RigidGraGeneDrift}
   Suppose there is a conull set $E\subset X$ such that for every $x\in E$ and every $s,t\in\R$
   \[
   \exp(sV)T_tx=T_t\exp(sV)x.
   \]
   Then for $\nu$-a.e. $x$ and every $t\in\R$
   \[
   \beta(x,t)=tV.
   \]
   Consequently, $T_tx=\exp(tV)x$ modulo $\nu$, and $\nu$ is a homogeneous probability measure.
\end{Theorem}
\begin{proof}
    Set $S_sx=\exp(sV)x$. Since $E$ is conull,  $E\cap\supp\nu$ is dense in $\supp\nu$. For fixed $s,t\in\R$, both maps
    \[
    x\longmapsto S_sT_tx,\qquad x\longmapsto T_tS_sx
    \]
    are continuous. Hence the graph relation extends from $E$ to $\supp\nu$:
    \[
    S_sT_tx=T_tS_sx\qquad x\in\supp\nu,\ s,t\in\R.
    \]
    Notice also that $\supp\nu$ is $T_t$-invariant. 

    Fix $x\in\supp\nu$ and $t\in\R$. By assumption, and since $\h_3$ is abelian, we have
    
    \[
    \begin{aligned}
        \exp\bigl(\beta(S_sx,t)\bigr)S_sx=T_tS_sx=S_sT_tx=\exp\bigl(\beta(x,t)\bigr)S_sx.
    \end{aligned}
    \]
    Consequently,
    \[
    \beta(S_sx,t)-\beta(x,t)\in\Lambda_{S_sx},
    \]
    where $\Lambda_x=\{B\in\h_3:\exp(B)x=x\}$. Lemma~\ref{lem:algStab} gives $\Lambda_{S_sx}=\Lambda_x$. Hence the continuous map
    \[
    s\mapsto\beta(S_sx,t)-\beta(x,t)
    \]
    takes values in the fixed discrete group $\Lambda_x$. Since it vanishes at $s=0$, it vanishes identically. Therefore
    \begin{equation}\label{eq:SInv}
        \beta(S_sx,t)=\beta(x,t),\qquad x\in\supp\nu,\ s,t\in\R.
    \end{equation}

    Choose $Q\in\h_3$ such that $(V,Q)$ is a basis of $\h_3$ and write 
    \[
    \beta(x,t)=a(x,t)V+c(x,t)Q.
    \]
    Let $\pi_Q:\h_3\longrightarrow\R$ denote the $Q$-coordinate:
    \[
    \pi_Q(aV+cQ)=c.
    \]

    Fix $x\in\supp\nu$ and define
    \[
    \mathcal{L}_t\coloneq\beta(x,t)+\R V.
    \]
    Since the $V$-component is absorbed into $\R V$,
    \begin{equation}\label{eq:lineC}
        \mathcal{L}_t=c(x,t)Q+\R V.
    \end{equation}
    In particular
    \[
    \mathcal{L}_t=\mathcal{L}_u\Longleftrightarrow c(x,t)=c(x,u).
    \]

    For $r\in\R$, define
    \[
    F_r(z)\coloneq z+\beta(\exp(z)x,r),\quad z\in\h_3.
    \]
    Since $\h_3$ is abelian,
    \[
    \exp(F_r(z))x=T_r\exp(z)x.
    \]
    Hence the cocycle identity gives that $F_r$ is a homeomorphism with inverse $F_{-r}$. If~$z=\beta(x,t)+sV$, then $\exp(z)x=S_sT_tx$, so \eqref{eq:SInv} and the cocycle identity give
    \[
    F_r(z)=\beta(x,t+r)+sV.
    \]
    Hence
    \begin{equation}\label{eq:LInv}
        F_r(\mathcal{L}_t)=\mathcal{L}_{t+r},\qquad r,t\in\R.
    \end{equation}

    Define 
    \[
    P_x\coloneq\{t\in\R:c(x,t)=0\}=\{t\in\R:\mathcal{L}_t=\mathcal{L}_0\}.
    \]
    Here the second equality follows from \eqref{eq:lineC}. Moreover, for $t,u\in\R$, applying the homeomorphism $F_{-u}$ and using \eqref{eq:LInv}, we obtain
    \[
    \begin{aligned}
        c(x,t)=c(x,u)&\Longleftrightarrow \mathcal{L}_t=\mathcal{L}_u\Longleftrightarrow F_{-u}(\mathcal{L}_t)=F_{-u}(\mathcal{L}_u)\\
        &\Longleftrightarrow\mathcal{L}_{t-u}=\mathcal{L}_0\Longleftrightarrow t-u\in P_x.
    \end{aligned}
    \]
    In particular, $0\in P_x$, and if $t,u\in P_x$, then $t-u\in P_x$. Thus $P_x$ is a subgroup of $\R$. It is closed as $P_x=c(x,\cdot)^{-1}(\{0\})$ and $t\mapsto c(x,t)$ is continuous.

    Every closed subgroup of $\R$ is either $\R$, $\{0\}$ or $b\Z$ for some $b>0$. If $P_x=b\Z$, then
    \[
    \overline{c}_x:\R/b\Z\longrightarrow\R,\qquad \overline{c}_x([t])=c(x,t),
    \]
    is a well-defined continuous injection, which is impossible. Therefore
    \[
    P_x=\R\qquad\text{or}\qquad P_x=\{0\}.
    \]
    In the first case,
    \[
    c(x,t)=0\qquad t\in\R.
    \]
    In the second case, $t\mapsto c(x,t)$ is a continuous injection from $\R$ to $\R$, and is therefore strictly monotone.

    It follows that every $x\in\supp\nu$ belongs to exactly one of the sets
    \[
    \begin{aligned}
        E_0&\coloneq\{x\in\supp\nu: c(x,t)=0\text{ for every }t\in\R\},\\
        E_+&\coloneq\{x\in\supp\nu: c(x,t)>0\text{ for every }t>0\},\\
        E_-&\coloneq\{x\in\supp\nu: c(x,t)<0\text{ for every }t>0\}.
    \end{aligned}
    \]
    Since $c(x,\cdot)$ is either identically zero, strictly increasing, or strictly decreasing, the defining conditions may be tested on rational times. Hence $E_0,E_+$, and $E_-$ are Borel.

    Taking the $Q$-coordinate in the cocycle identity gives
    \begin{equation}\label{eq:cCocycle}
        c(T_ux,t)=c(x,u+t)-c(x,u).
    \end{equation}
    Equation~\eqref{eq:cCocycle} shows that if $c(x,\cdot)$ is identically zero, strictly increasing, or strictly decreasing, then the same is true of $c(T_ux,\cdot)$. Hence  $E_0,E_+$, and $E_-$ are $T_t$-invariant. Since they partition $\supp\nu$, ergodicity implies that exactly one of them is conull.

    We now rule out $E_+$. Suppose, towards a contradiction, that $E_+$ is conull, and define
    \[
    f_+(x)\coloneq\min\Big(1,\max\big(c(x,1),0\big)\Big).
    \]
    Then $0\leq f_+(x)\leq 1$ and $f_+(x)>0$ for $\nu$-a.e. $x$. Since $E_+$ is invariant, for $x\in E_+$ and every $k\geq 0$,
    \[
    0<f_+(T_kx)\leq c(T_kx,1).
    \]
    Therefore, by \eqref{eq:cCocycle}
    \[
    0\leq \frac{1}{n}\sum_{k=0}^{n-1}f_+(T_kx)\leq\frac{1}{n}\sum_{k=0}^{n-1}c(T_kx,1)=\frac{c(x,n)}{n}.
    \]
    Since the $Q$-coordinate of $V$ is zero, the asymptotic drift gives 
    \[
    \frac{c(x,n)}{n}=\pi_Q\left(\frac{\beta(x,n)}{n}\right)\longrightarrow\pi_Q(V)=0, \qquad \text{$\nu$-a.e. $x$}.
    \]
    Hence
    \[
    \frac{1}{n}\sum_{k=0}^{n-1}f_+(T_kx)\longrightarrow0\qquad \text{$\nu$-a.e.}
    \]
    These averages are bounded by $1$. Dominated convergence and the $T_k$-invariance of $\nu$ therefore give
    \[
    0=\lim_{n\to\infty}\int_X\frac{1}{n}\sum_{k=0}^{n-1}f_+(T_kx)d\nu(x)=\lim_{n\to\infty}\frac{1}{n}\sum_{k=0}^{n-1}\int_Xf_+(T_kx)d\nu(x)=\int_Xf_+d\nu.
    \]
    This is impossible as $f_+>0$ almost everywhere. Thus $\nu(E_+)=0$. Applying the same argument to 
    \[
    f_-(x)\coloneq\min\Big(1,\max\big(-c(x,1),0\big)\Big)
    \]
    gives $\nu(E_-)=0$. Hence $E_0$ is conull and 
    \begin{equation}\label{eq:betaDriftPure}
        \beta(x,t)=a(x,t)V\qquad x\in E_0,\ t\in\R.
    \end{equation}

    Let $E_*\subset E_0$ be the conull set on which the asymptotic drift is $V$. Fix $x\in E_*$. For $u\in\R$, equation~\eqref{eq:betaDriftPure} gives
    \[
    T_ux=\exp(a(x,u)V)x=S_{a(x,u)}x.
    \]
    Hence equation~\eqref{eq:SInv} and the cocycle identity imply
    \[
    \beta(x,u+t)=\beta(x,u)+\beta(x,t).
    \]
    The function $t\mapsto \beta(x,t)$ is a continuous homomorphism from $(\R,+)$ to~$\h_3$, so
    \[
    \beta(x,t)=A(x)t
    \]
    for some $A(x)\in\h_3$. The asymptotic drift gives $A(x)=V$. Therefore
    \[
    \beta(x,t)=tV\qquad x\in E_*,\ t\in\R.
    \]

    Thus $T_t$ agrees modulo $\nu$ with the algebraic unipotent flow $x\mapsto\exp(tV)x$. Consequently, $\nu$ is invariant and ergodic for this algebraic flow. Ratner's measure classification theorem \cites{ratner1990measure,ratner1991measure} gives a closed connected subgroup $L<G$ and a closed finite-volume orbit $Lz$ such that
    \[
    \exp(\R V)\leq L,\qquad \nu=m_{Lz}.
    \]
    This completes the proof.
\end{proof}

\begin{proof}[Proof of Theorem~\ref{thm:oneRootSL3rigidity}]
    Pass to the normalized system associated with~$g_V$. Apply Theorem~\ref{thm:oneRtGraPro} and Theorem~\ref{thm:RigidGraGeneDrift}, and transport the conclusion back to the original system.
\end{proof}

\begin{proof}[Proof of Corollary~\ref{cor:RootRigiditySL3}]
    Since $\nu$ is a probability, the pointwise estimate implies the one root mean square estimate with the same exponents. Apply Theorem~\ref{thm:oneRootSL3rigidity}.
\end{proof}

\section{Higher-dimensional subactions}\label{sec:HigheerDimSubactions}
The one dimensional theorem is also a building block for restrictions to a parameter subspace. Once it is known on finitely many ergodic directions forming a basis, the cocycle identity and discreteness of stabilizers recover the full subaction. We first identify the drift on each line and then carry out this finite basis argument.

We retain the notation and hypotheses preceding Theorem~\ref{thm:finiteSubactionRigidity}. Set
\[
\mathcal{B}(x,p)=\iota\tau(x,p),\qquad p\in P.
\]

\begin{Lemma}\label{lem:dirDrift}
    For every $x\in X$, $p\in P$, and $t\in\R$,
    \[
    \mathcal{B}(x,tp)=\int_0^t\iota\big(D_P\tau(\alpha_{sp}x,0)(p)\big)ds.
    \]
    If the line action $t\mapsto\alpha_{tp}$ is $\nu$-ergodic, then
    \[
    \lim_{t\to+\infty}\frac{\mathcal{B}(x,tp)}{t}=\mathcal{V}_{\nu}(p)\qquad \text{for $\nu$-a.e. $x$.}
    \]
\end{Lemma}
\begin{proof}
    For $r\neq0$, the cocycle identity gives
    \[
    \frac{\mathcal{B}(x,(t+r)p)-\mathcal{B}(x,tp)}{r}=\frac{\iota\tau(\alpha_{tp}x,rp)}{r}.
    \]
    Letting $r\to0$ gives
    \[
    \frac{d}{dt}\mathcal{B}(x,tp)=\iota\big(D_P\tau(\alpha_{tp}x,0)(p)\big).
    \]
    The fundamental theorem of calculus gives the displayed formula. The last assertion follows from the pointwise ergodic theorem applied to the integrable function $x\mapsto\iota\big(D_P\tau(x,0)(p)\big)$, where the integrability follows from the assumption $\iota\circ D_P\tau(\cdot,0)\in L^1\big(\nu,\operatorname{Hom}(P,\h_n)\big)$.
\end{proof}

\begin{proof}[Proof of Theorem~\ref{thm:finiteSubactionRigidity}]
    Recall that $p_1,\ldots,p_k$ form a basis of $P$ and for each $i$, the line action $t\mapsto\alpha_{tp_i}$ is $\nu$-ergodic. By Lemma~\ref{lem:dirDrift},  $\mathcal{V}_{\nu}(p_i)$ is the drift of the $i$-th line action. The corresponding one-dimensional theorem therefore gives, on a conull set independent of $t$,
    \[
    \alpha_{tp_i}x=\exp\big(t\mathcal{V}_{\nu}(p_i)\big)x\qquad t\in\R.
    \]
    Hence
    \begin{equation}\label{eq:BVlambda}
        \mathcal{B}(x,tp_i)-t\mathcal{V}_{\nu}(p_i)\in\Lambda_x\qquad t\in\R,
    \end{equation}
    where $\Lambda_x=\{B\in\h_n:\exp(B)x=x\}$. By Lemma~\ref{lem:algStab}, $\Lambda_x$ is discrete. Since the LHS of \eqref{eq:BVlambda} is continuous in $t$ and vanishes at $t=0$, it vanishes identically. Thus
    \begin{equation}\label{eq:1Dcase}
        \mathcal{B}(x,tp_i)=t\mathcal{V}_{\nu}(p_i)\qquad t\in\R.
    \end{equation}

    Set $P_{\Q}\coloneq\sum_{i=1}^k\Q p_i$. For each $i$, let $E_i$ be a conull set on which \eqref{eq:1Dcase} holds for every $t$. Since $\nu$ is $P$-invariant,
    \[
    E=\bigcap_{i=1}^k\bigcap_{q\in P_{\Q}}\alpha_q^{-1}E_i
    \]
    is conull. Let $q=\sum_{i=1}^kr_ip_i\in P_{\Q}$ and set $q_0=0$, $q_i=\sum_{\ell=1}^ir_{\ell}p_{\ell}$. For $x\in E$, repeated use of the cocycle identity gives
    \[
    \mathcal{B}(x,q)=\sum_{i=1}^k\mathcal{B}(\alpha_{q_{i-1}}x,r_ip_i)=\sum_{i=1}^kr_i\mathcal{V}_{\nu}(p_i)=\mathcal{V}_{\nu}(q).
    \]
    Both sides are continuous in $q$, so the identity extends to every $p\in P$. Hence the $P$-subaction agrees modulo $\nu$ with the algebraic action of $\exp(\mathcal{V}_{\nu}(P))$. Ratner's theorem applied to any one of the ergodic basis lines shows that $\nu$ is homogeneous.
\end{proof}

We use the following consequence of Pugh and Shub's theorem on ergodic elements \cite{PughShub71}*{Theorem~1}.
\begin{Lemma}\label{lem:almostLineErg} 
   Let $\alpha$ be an ergodic probability-preserving $\R^k$-action on a standard probability space. Then, for almost every $p\in\mathbb{S}^{k-1}$, the one-parameter subaction
   \[
   t\longmapsto\alpha_{tp}
   \]
   is ergodic.
\end{Lemma}
\begin{proof}
By Pugh and Shub's theorem, $\alpha_p$ is ergodic for Lebesgue almost every $p\in\R^k$. Each such $p$ generates an ergodic line action. Ergodicity of the line is unchanged by nonzero rescaling of $p$; polar coordinates therefore give the assertion on the sphere.    
\end{proof}

Recall that $\mathcal{V}_{\nu}:P\to \h_n$ is linear. The one-dimensional rigidity theorems used in Theorem~\ref{thm:finiteSubactionRigidity} require a nonzero asymptotic drift. This is the only reason for excluding directions in $\ker\mathcal{V}_{\nu}$. If $\mathcal{V}_{\nu}\neq0$, then $\ker\mathcal{V}_{\nu}$ is a proper linear subspace of $P$, and hence has Lebesgue measure zero. Although $\mathcal{V}_{\nu}$ need not be injective, one can therefore choose a basis of $P$ entirely outside its kernel.
\begin{Corollary}[Almost every direction criterion]\label{cor:almostEveryDirCri}
    Suppose that $\nu$ is ergodic for the $P$-action and that 
    \[
    \mathcal{V}_{\nu}:P\longrightarrow\h_n
    \]
    is nonzero. Assume that, for Lebesgue almost every $p\in P$ satisfying
    \[
    \mathcal{V}_{\nu}(p)\neq0,
    \]
    the corresponding line restriction satisfies one of the one-dimensional regularity hypotheses listed in Theorem~\ref{thm:finiteSubactionRigidity}. Then the conclusion of that theorem holds.
\end{Corollary}
\begin{proof}
    By Lemma~\ref{lem:almostLineErg} and polar coordinates, the vectors $p\in P$ for which
    \[
    t\longmapsto\alpha_{tp}
    \]
    is $\nu$-ergodic form a full measure subset of $P$. Moreover, since $\mathcal{V}_{\nu}\neq0$, $\ker\mathcal{V}_{\nu}$ is a proper linear subspace and therefore has Lebesgue measure zero.

    Hence there is a full measure set $\mathcal{G}\subset P$ such that, for every $p\in\mathcal{G}$,
    \[
    t\mapsto\alpha_{tp}\quad\text{is $\nu$-ergodic},\qquad \mathcal{V}_{\nu}(p)\neq0,
    \]
    and the corresponding line restriction satisfies one of the required hypotheses.

    Let $d=\dim P$. Choose $p_1,\ldots,p_d\in\mathcal{G}$ inductively. Once $p_1,\ldots,p_r$ have been chosen with $r<d$, the proper subspace
    \[
    \operatorname{span}\{p_1,\ldots,p_r\}
    \]
    has Lebesgue measure zero, so $\mathcal{G}$ contains a vector $p_{r+1}$ outside this span. Thus $p_1,\ldots,p_d$ form a basis of $P$. Apply Theorem~\ref{thm:finiteSubactionRigidity}.
\end{proof}

The finite basis argument is not specific to the first row action. We record its abstract form in the notation of \S\ref{sec:cptRatnerFra} and \S\ref{sec:RigidityFullGRaFam}.

Let $\alpha$ be a continuous time change of the $N$-action on $X_0$, $\{0\}\neq P<\mathfrak{u}$ a linear subspace, and $\nu$ a $P$-invariant probability measure. Denote the inverse reparametrization cocycle by
\[
 \mathcal{B}:X_0\times \mathfrak{u}\longrightarrow\mathfrak{u},\qquad \alpha_px=\exp\big(\mathcal{B}(x,p)\big)x.
\]
Assume that, for every $x\in X_0$, the map
\[
P\ni p\longmapsto \mathcal{B}(x,p)
\]
is Fr\'echet differentiable at $0$, that
\[
x\longmapsto D_P\mathcal{B}(x,0)
\]
is continuous, and that
\[
D_P\mathcal{B}(\cdot,0)\in L^1\big(\nu,\operatorname{Hom}(P,\mathfrak{u})\big).
\]
Define the linear drift map
\[
\mathcal{V}_{\nu}(p)\coloneq\int_{X_0}D_P\mathcal{B}(x,0)(p)d\nu(x),\qquad p\in P.
\]
For $p\in P$, write
\[
T_t^px\coloneq \alpha_{tp}x.
\]
\begin{Corollary}[Abstract finite basis criterion]\label{cor:absFinBasisCri}
    Suppose that $p_1,\ldots,p_k$ form a basis of $P$ and that, for every $1\leq i\leq k$, $T^{p_i}$ is $\nu$-ergodic, $\mathcal{V}_{\nu}(p_i)\neq0$, and the pair $\big(\mathfrak{u},\mathcal{V}_{\nu}(p_i)\big)$ admits a compact Ratner frame which is either pointwise regular or correlation regular for $(T^{p_i},\nu)$. Then
    \[
    \mathcal{B}(x,p)=\mathcal{V}_{\nu}(p),\qquad \text{for $\nu$-a.e. $x$ and every $p\in P$}.
    \]
    Consequently, the $P$-subaction is algebraic modulo $\nu$ and $\nu$ is a homogeneous probability measure.
\end{Corollary}
\begin{proof}
    The proof of Lemma~\ref{lem:dirDrift} applies verbatim with $\mathcal{B}$ in place of $\iota\tau$. Hence $\mathcal{V}_{\nu}(p_i)$ is the asymptotic drift of the $i$-th line restriction. Theorem~\ref{thm:absCptRigid} gives for $\nu$-a.e. $x$ and every $t\in\R$,
    \[
    \mathcal{B}(x,tp_i)=t\mathcal{V}_{\nu}(p_i).
    \]

    The reconstruction argument in the proof of Theorem~\ref{thm:finiteSubactionRigidity} then gives for $\nu$-a.e. $x$ and every $p\in P$
    \[
    \mathcal{B}(x,p)=\mathcal{V}_{\nu}(p).
    \]
    The homogeneity of $\nu$ follows from Theorem~\ref{thm:absCptRigid} applied to any one of the basis directions.
\end{proof}

\section{Sharpness of the scale conditions}\label{sec:sharpnessScaleCond}

We begin with examples that meet the closed scale conditions but fail to be algebraic in the original coordinates. They are obtained by applying a translation on each orbit by a small smooth displacement; in particular, they remain smoothly conjugate to the algebraic action.

Equip $G$ with a complete right-invariant Riemannian metric, denote its distance by $d_G$, let $d_X$ be the induced quotient metric on $X$, and put $\bar{d}_X=\min\{1,d_X\}$. Let $m_X$ be Haar probability on $X$ and choose a nonconstant smooth compactly supported map $F:X\to\h_n$. For $F$ sufficiently small in $C^1$, set
\[
\Psi_F(x)\coloneq \exp(F(x))x.
\]
Then the homotopy $\Psi_{rF}$, $0\leq r\leq 1$, consists of proper local diffeomorphisms: every map is the identity outside a compact set and its derivative is uniformly close to the identity. Hence $\Psi_F$ is a covering map homotopic to the identity. Its induced map on the fundamental group is an isomorphism, so the covering has one sheet and $\Psi_F$ is a diffeomorphism. Define
\[
\alpha_w^F\coloneq\Psi_F^{-1}\circ\phi_w\circ\Psi_F.
\]
Its reparametrization cocycle is
\begin{equation}\label{eq:cocycleIdRhoF}
    \rho_F(x,p)=p+\iota^{-1}\left(F(\phi_px)-F(x)\right).
\end{equation}
Indeed, this is a smooth cocycle and
\[
\phi_{\rho_F(x,p)}\Psi_F(x)=\Psi_F(\phi_px).
\]
Since $F$ is sufficiently small, every fiber map is a global diffeomorphism. In fact, if 
\[
G_x(p)=\iota^{-1}\left(F(\phi_px)-F(x)\right),
\]
then $\sup_{x,p}\norm{D_pG_x(p)}\leq\kappa<1$, and hence
\[
(1-\kappa)\norm{p-q}\leq\norm{\rho_F(x,p)-\rho_F(x,q)}\leq(1+\kappa)\norm{p-q}.
\]
Thus each fiber map is injective, proper, and a local diffeomorphism; its image is open and closed in $\R^{n-1}$. The parameter dependent inverse function theorem gives a jointly smooth inverse family.

Let $\tau_F$ be the inverse cocycle and set
\[
T_t^F\coloneq\alpha_{tv}^F,\qquad \beta_F(x,t)\coloneq\iota\tau_F(x,tv),\qquad \nu_F\coloneq(\Psi_F^{-1})_*m_X.
\]
\begin{Lemma}\label{lem:OrbShiftId}
    For every $x\in X$ and $t\in\R$,
    \[
    \beta_F(x,t)=t\iota(v)+F(x)-F(T_t^Fx).
    \]
    The measure $\nu_F$ has full support, is $T_t^F$-ergodic, and has asymptotic drift $\iota(v)$.
\end{Lemma}
\begin{proof}
    If $p=\tau_F(x,tv)$, then $\rho_F(x,p)=tv$ and $\phi_px=T_t^Fx$. Equation~\eqref{eq:cocycleIdRhoF} gives the stated cocycle identity. Moore's ergodic theorem makes the algebraic flow $\phi_{tv}$ ergodic for $m_X$, and $\Psi_F$ conjugates it to $T_t^F$ with invariant measure $\nu_F$. Full support follows since $\Psi_F$ is a diffeomorphism. Since $F$ is bounded, division of the stated cocycle identity by $t$ gives the drift.
\end{proof}

\begin{Lemma}\label{lem:CptRtEst}
    Assume the drift is normalized to $V=E_{12}$. For every choice of $\theta_2,\ldots,\theta_n\in(0,1]$ there are $C,s_0>0$ such that for every $t\in\R$,
    \begin{equation*}
       \begin{aligned}
           &[\beta_F(\cdot,t)]_{K_2,\theta_2,s_0}\leq C(1+\abs{t})^{2\theta_2},\\
        [\beta_F(\cdot,t)&]_{K_j,\theta_j,s_0}\leq C(1+\abs{t})^{\theta_j}\qquad \text{for }3\leq j\leq n.  
       \end{aligned}
    \end{equation*}
\end{Lemma}
\begin{proof}
After decreasing $s_0$ if necessary, assume $s_0\leq 1$. For small $s$, define 
\[
\begin{aligned}
    R_{j,s}(x)&\coloneq \exp(F(\exp(sK_j)x))\exp(sK_j)\exp(-F(x)),\\
    W_{j,s}(x)&\coloneq\log R_{j,s}(x).
\end{aligned}
\]
Since $R_{j,0}(x)=\id$ and $W_{j,0}(x)=0$, uniform Taylor expansion at $s=0$ gives
\[
W_{j,s}(x)=sY_j(x)+s^2Z_{j,s}(x),
\]
where $Y_j$ and $Z_{j,s}$ are uniformly bounded. 

Let $X_{K_j}$ denote the vector field on $X$ generated by $K_j$. Set
\[
F_s=F(\exp(sK_j)x),\qquad F_0=F(x).
\]
Since $\h_n$ is abelian, 
\[
R_{j,s}(x)=\exp(F_s-F_0)\exp\left(s\Ad_{\exp(F_0)}K_j\right).
\]
Moreover, uniformly in $x$,
\[
F_s-F_0=s(dF)_x(X_{K_j}(x))+O(s^2).
\]
The Baker-Campbell-Hausdorff formula therefore gives
\[
Y_j(x)=(dF)_x(X_{K_j}(x))+\Ad_{\exp(F(x))}K_j.
\]

The first term belongs to $\h_n$ and hence commutes with $V$. Since $[F(x),V]=0$, the operators $\Ad_{\exp(F(x))}$ and $\ad_V$ commute. Thus for $j\geq 3$, 
\[
\ad_V^2Y_j(x)=\Ad_{\exp(F(x))}\ad_V^2K_j=0\qquad j\geq 3.
\]

Since $V^2=0$, we have $\ad_V^3=0$ on $\sl_n(\R)$. It follows that
\[
\norm{\Ad_{\exp(tV)}W_{2,s}(x)}\ll\abs{s}(1+\abs{t})^2,
\]
and, for $j\geq3$, the linear Taylor coefficient $Y_j$ has adjoint degree at most one, while the quadratic remainder is estimated using the general degree-two bound. Therefore
\[
\norm{\Ad_{\exp(tV)}W_{j,s}(x)}\ll\abs{s}(1+\abs{t})+s^2(1+\abs{t})^2.
\]

The maps $F,\Psi_F$ and $\Psi_F^{-1}$ are globally Lipschitz with respect to the truncated quotient metric. Moreover,
\[
\Psi_F(\exp(sK_j)x)=R_{j,s}(x)\Psi_F(x).
\]
Consequently, 
\[
\begin{aligned}
    \exp(tV)\Psi_F(\exp(sK_j)x)=\exp\left(\Ad_{\exp(tV)}W_{j,s}(x)\right)\exp(tV)\Psi_F(x).
\end{aligned}
\]
The right-invariance of $d_G$ and definition of the quotient metric give
\[
\bar{d}_X(\exp(A)y,y)\ll\min(1,\norm{A}).
\]
Since
\[
T_t^F=\Psi_F^{-1}\exp(tV)\Psi_F,
\]
the preceding estimates and the global Lipschitz bound for $\Psi_F^{-1}$ give
\[
\bar{d}_X(T_t^F\exp(sK_2)x,T_t^Fx)\ll\min\{1,\abs{s}(1+\abs{t})^2\},
\]
\[
\bar{d}_X(T_t^F\exp(sK_j)x,T_t^Fx)\ll\min\{1,\abs{s}(1+\abs{t})+s^2(1+\abs{t})^2\}.
\]
Lemma~\ref{lem:OrbShiftId} gives
\begin{equation*}
    \beta_F(\exp(sK_j)x,t)-\beta_F(x,t)
    =F(\exp(sK_j)x)-F(x)-F(T_t^F\exp(sK_j)x)+F(T_t^Fx).
\end{equation*}
The Lipschitz bound for $F$ and $\min(1,a)\leq a^{\theta}$ for $0<\theta\leq 1$ now give the asserted estimates. For the degree-one roots, use
\[
\min(1,z+z^2)\leq 2\min(1,z)\leq 2z^{\theta},\qquad z=\abs{s}(1+\abs{t}),\qquad 0<\theta\leq1.
\]
This proves the lemma.
\end{proof}

\begin{proof}[Proof of Proposition~\ref{prop:SharpProp}]
    Choose $F$ as above. Lemma~\ref{lem:OrbShiftId} gives full support, ergodicity, and the drift, while Lemma~\ref{lem:CptRtEst} gives the endpoint estimates after normalization.

    Suppose that, for some $W\in\h_n$, the flow $T_t^F$ agreed modulo $\nu_F$ with the algebraic $H_n$-flow $x\mapsto\exp(tW)x$. Intersecting the conull equality sets first over $t\in\Q$, and then using continuity in $t$, gives a conull set on which the equality holds for every $t\in\R$. Hence 
    \[
    \beta_F(x,t)-tW\in\Lambda_x\qquad t\in\R.
    \]
    The LHS is continuous in $t$, takes values in the discrete group $\Lambda_x$, and vanishes at $t=0$; hence it vanishes identically. The asymptotic drift then gives $W=\iota(v)$. The cocycle identity in Lemma~\ref{lem:OrbShiftId} therefore yields
    \[
    F(T_t^Fx)=F(x)\qquad t\in\R.
    \]
    Ergodicity makes $F$ constant $\nu_F$-almost everywhere, and full support and continuity make it constant everywhere, a contradiction.

    In dimension three, the first endpoint estimate implies the corresponding one-root $L^2$ estimate with 
    \[
    \delta=\theta_2,\qquad \gamma=2\theta_2,
    \]
    and hence $2\delta=\gamma$.
\end{proof}

\subsection{Compact examples and optimal exponents}\label{sec:cptExamplesOptCoe}
Let $n=3$ and assume that $\Gamma$ is cocompact. Let $m$ be the Haar probability on $X$. In the parameter coordinates corresponding to $(U_2,U_3)$, write $e_1=(1,0)^{\mathrm{T}}$. Choose a nonconstant real $b\in C^{\infty}(X)$. If $U_2b$ vanished identically, then $b(\exp(tU_2)x)=b(x)$ for all $x,t$; Haar ergodicity and continuity would make $b$ constant. Since $\int_XU_2b\,dm=0$, the function $U_2b$ is itself nonconstant.

Apply the preceding construction with $F(x)=\epsilon b(x)U_2$, where $\epsilon\neq0$ is sufficiently small; in particular, take $1+\epsilon U_2b>0$. Since $\iota^{-1}(cU_2)=ce_1$, \eqref{eq:cocycleIdRhoF} gives
\[
\rho_{\epsilon}(x,(r,s))=\big(r+\epsilon(b(\phi_{(r,s)}x)-b(x)),s\big),\qquad \Psi_{\epsilon}(x)=\exp(\epsilon b(x)U_2)x.
\]
The same $\Psi_{\epsilon}$ conjugates the full actions. For $T_t^{\epsilon}=\alpha_{te_1}^{\epsilon}$, the inverse parameter satisfies
\begin{equation}\label{eq:8.1inverseCocycle}
    \tau_{\epsilon}(x,te_1)=\big(t+\epsilon b(x)-\epsilon b(T_t^{\epsilon}x)\big)e_1.
\end{equation}
Its generator is $(1+\epsilon U_2b)^{-1}U_2$. Since $U_2b$ is nonconstant and the point stabilizers in $G$ are discrete, this is not a constant algebraic vector field.

The invariant probability is
\[
d\nu_{\epsilon}=(1+\epsilon U_2b)dm=(\Psi_{\epsilon}^{-1})_*m.
\]
Indeed, in fundamental Lie-algebra frames, factoring $\Ad_{\exp(\epsilon bU_2)}$ out of $D\Psi_{\epsilon}$ leaves 
\[
Y\longmapsto Y+\epsilon( Yb)U_2.
\]
This map scales $\R U_2$ by $1+\epsilon U_2b(x)$ and induces the identity on the quotient by $\R U_2$. Since the adjoint map has determinant one, the Haar Jacobian of $\Psi_{\epsilon}$ is  $1+\epsilon U_2b>0$. The measure $\nu_{\epsilon}$ has full support, and the conjugacy gives ergodicity. Equation~\eqref{eq:8.1inverseCocycle} gives uniformly bounded inverse deviations and parameter drift $e_1$, hence Lie algebra drift $U_2$. The cocycle formula and the inverse function theorem show that both $\rho_{\epsilon}(x,p)$ and $\tau_{\epsilon}(x,p)$ converge to $p$ in $C^{\infty}$ on compact parameter sets as $\epsilon\to0$.

For every fixed sufficiently small $\epsilon\neq0$ and every $0<\delta\leq 1$, Lemma~\ref{lem:CptRtEst} gives for $\abs{s}\leq s_0$ and $t\in\R$,
\[
\norm{\tau_{\epsilon}(\exp(sK_2)\cdot,te_1)-\tau_{\epsilon}(\cdot,te_1)}_{\infty}\leq C_{\epsilon,\delta}\abs{s}^{\delta}(1+\abs{t})^{2\delta}.
\]
If an $L^2(\nu_{\epsilon})$ estimate with the same $\delta$ and a nonnegative exponent $\gamma<2\delta$ held, Theorem~\ref{thm:oneRootSL3rigidity} would give $T^{\epsilon}_tx=\exp(tU_2)x$ algebraic modulo $\nu_{\epsilon}$. Full support and continuity would give equality everywhere, contradicting the generator above. Thus $2\delta$ is the optimal time exponent, attained in both $L^2(\nu_{\epsilon})$ and the uniform norm. Each fixed example realizes $(3/4, 3/2)$, $(1/\sqrt{2},\sqrt{2})$, and indeed all pairs $(\delta,2\delta)$ with $0<\delta\leq 1$.

\section{Mean square sharpness on the modular  \texorpdfstring{$\SL_3(\R)$}{SL3R} quotient}\label{sec:sl3Sharpness}
The examples in the previous section are smoothly conjugate to algebraic actions. In this section, we provide an example of a rank-two time change of the first row action of $\SL_3(\R)$ on the modular quotient $\SL_3(\R) / \SL_3(\Z)$, which is not smoothly conjugate to the algebraic action by a smooth leaf preserving diffeomorphism and satisfies the conclusion of Theorem~\ref{thm:sharpnessSL3rigidityOneRoot}. In order to construct our example, we exploit the primitive Siegel transform to obtain an explicit rank-two closed cocycle which is not a coboundary.

Let
\[
G=\SL_3(\R),\qquad \Gamma=\SL_3(\Z),\qquad X=G/\Gamma,
\]
and let $m$ be Haar probability measure on $G/\Gamma$. For the construction it is convenient to use the commuting nilpotents
\[
V=E_{21},\qquad Z=E_{23},\qquad \h=\R V\oplus\R Z.
\]
The integral matrix
\begin{equation}
    w_0=\begin{pmatrix}
        0 & 1 & 0\\
        1 & 0 & 0 \\
        0 & 0 & -1
    \end{pmatrix}\in \SL_3(\Z)
\end{equation}
satisfies
\[
\Ad_{w_0}V=E_{12},\qquad \Ad_{w_0}Z=-E_{13},\qquad \Ad_{w_0}(E_{12}-E_{21})=E_{21}-E_{12}.
\]

Let $C_{w_0}(g\Gamma)=w_0g\Gamma$. Then
\[
C_{w_0}\big(\exp(rV+sZ)x\big)=\exp(rE_{12}-sE_{13})C_{w_0}(x).
\]
We first carry out the construction for the action generated by $V$ and $Z$. Once the time-changed action $\alpha^{\epsilon}$
has been constructed, define
\begin{equation}\label{eq:transportNonRigidity}
\widehat{\alpha}_{(r,s)}^{\epsilon}=C_{w_0}\circ\alpha_{(r,-s)}^{\epsilon}\circ C_{w_0}^{-1}.
\end{equation}
This is a time change of the standard first row action. The parameter reflection fixes the selected axis $(t,0)$, while
\[
\Ad_{w_0}(E_{12}-E_{21})=K_2.
\]
We return to the transported system after proving the compact root estimate.

By Lemma~\ref{lem:smmothDisplacement}, every smooth leaf preserving diffeomorphism for the $\exp(\h)$-action has the form
\[
\Psi_B(x)=\exp\big(B_1(x)V+B_2(x)Z\big)x,\qquad B=(B_1,B_2)\in C^{\infty}(X,\R^2).
\]

\subsection{The primitive Siegel cocycle}
Let 
\[
\Z^3_{\operatorname{prim}}\coloneq\{v\in\Z^3:\gcd(\abs{v_1},\abs{v_2},\abs{v_3})=1\}.
\]
For $F\in C_c^{\infty}(\R^3)$, define its primitive Siegel transform by
\begin{equation}\label{eq:primSiegelTrans}
    \Theta_F^{\operatorname{prim}}(g\Gamma)=\sum_{v\in\Z^3_{\operatorname{prim}}}F(gv).
\end{equation}
The sum is pointwise finite. It defines a smooth function on $X$, because locally only finitely many lattice vectors meet the fixed compact support of $F$. It is right $\Gamma$-invariant because $\Gamma$ acts bijectively on $\Z^3_{\operatorname{prim}}$.

For $Y\in\sl_3(\R)$, we use the same symbol for the infinitesimal generator of the left $G$-action on $X$ and write $Y^{\mathrm{lin}}$ for the corresponding linear vector field on $\R^3$:
\begin{equation}\label{eq:LieDeriX}
    (Y\varphi)(g\Gamma)\coloneq\left.\frac{d}{ds}\right|_{s=0}\varphi(\exp(sY)g\Gamma),\qquad (Y^{\mathrm{lin}}F)(y)=\left.\frac{d}{ds}\right|_{s=0}F(\exp(sY)y)=dF_y(Yy).
\end{equation}
If 
\[
D=Y_1\cdots Y_r,\qquad Y_1,\ldots,Y_r\in\sl_3(\R),
\]
is a finite composition of such Lie derivatives, set $D^{\operatorname{lin}}\coloneq Y_1^{\operatorname{lin}}\cdots Y_r^{\operatorname{lin}}$. Local finiteness of \eqref{eq:primSiegelTrans} gives 
\begin{equation}\label{eq:SiegelLieInter}
    D\Theta_F^{\operatorname{prim}}=\Theta_{D^{\operatorname{lin}}F}^{\operatorname{prim}}.
\end{equation}
Moreover, $D^{\operatorname{lin}}F\in C_c^{\infty}(\R^3)$.

\begin{Lemma}\label{lem:priFirstTwoMoments}
    For every $F\in C_c^{\infty}(\R^3)$,
    \[
    \int_X\Theta_F^{\operatorname{prim}}dm=\frac{1}{\zeta(3)}\int_{\R^3}F(y)dy,
    \]
    and $\Theta_F^{\operatorname{prim}}\in L^2(X,m)$. Consequently, for every fixed finite composition $D$ of Lie derivatives,
    \[
    D\Theta_F^{\operatorname{prim}}\in L^2(X,m).
    \]
\end{Lemma}
\begin{proof}
    Let 
    \[
    \widehat{F}(g\Gamma)\coloneq\sum_{v\in\Z^3\setminus\{0\}}F(gv)
    \]
    be the ordinary Siegel transform and set 
    \[
    F_k(y)=F(ky),\qquad k\geq 1.
    \]
    Let $\mu$ denote the arithmetic M\"{o}bius function. Since every nonzero $w\in\Z^3$ is uniquely a positive integer multiple of a primitive vector,
    \[
    \begin{aligned}
        \sum_{k\geq 1}\mu(k)\widehat{F_k}(g\Gamma)=&\sum_{w\in\Z^3\setminus\{0\}}\left(\sum_{k\mid\gcd(\abs{w_1},\abs{w_2},\abs{w_3})}\mu(k)\right)F(gw)\\
        =&\Theta_F^{\operatorname{prim}}(g\Gamma),
    \end{aligned}
    \]
    because
    \[
    \sum_{k\mid n}\mu(k)=\mathbf 1_{\{n=1\}}.
    \]
    For fixed $g$, only finitely many $w\in\Z^3$ satisfy $gw\in\supp F$, and every such $w$ has only finitely many common divisors. Thus the sums above are pointwise finite.

    Siegel's mean value theorem \cite{SiegelMeanValue1945}*{Eq.~(2), p.~341}, applied to $\abs{F_k}$, gives
    \[
    \begin{aligned}
        \sum_{k\geq 1}\abs{\mu(k)}\norm{\widehat{F_k}}_{L^1(X,m)}&\leq\sum_{k\geq1}\int_{\R^3}\abs{F(ky)}dy =\sum_{k\geq1}k^{-3}\norm{F}_1<\infty.
    \end{aligned}
    \]
    Hence the M\"{o}bius series converges absolutely in $L^1(X,m)$. Termwise integration and another application of Siegel's mean value theorem give
    \[
    \begin{aligned}
       \int_X\Theta_F^{\operatorname{prim}}dm=&\sum_{k\geq 1}\mu(k)\int_X\widehat{F_k}dm=\left(\sum_{k\geq 1}\frac{\mu(k)}{k^3}\right)\int_{\R^3}F(y)dy \\
       =&\frac{1}{\zeta(3)}\int_{\R^3}F(y)dy.
    \end{aligned}
    \]

    For the second moment, Rogers' primitive second moment formula \cite{RogersMeanValue1955}*{Theorem~5} applied to
    \[
    \rho(y_1,y_2)=\abs{F(y_1)}\abs{F(y_2)},
    \]
    gives
    \[
    \begin{aligned}
        \int_X\left(\Theta_{\abs{F}}^{\operatorname{prim}}\right)^2 dm=&\frac{1}{\zeta(3)^2}\norm{F}_1^2+\frac{1}{\zeta(3)}\norm{F}_2^2+\frac{1}{\zeta(3)}\int_{\R^3}\abs{F(y)}\abs{F(-y)} dy\\
        \leq &\frac{1}{\zeta(3)^2}\norm{F}_1^2+\frac{2}{\zeta(3)}\norm{F}_2^2.
    \end{aligned}
    \]
    Since
    \[
    \abs{\Theta_F^{\operatorname{prim}}}\leq\Theta_{\abs{F}}^{\operatorname{prim}},
    \]
    it follows that
    \[
    \Theta_F^{\operatorname{prim}}\in L^2(X,m).
    \]

    Finally, for every fixed finite composition $D$ of Lie derivatives, the function $D^{\operatorname{lin}}F$ belongs to $C_c^{\infty}(\R^3)$. Applying the preceding $L^2$ conclusion to $D^{\operatorname{lin}}F$ and using \eqref{eq:SiegelLieInter} proves the final assertion.
\end{proof}

\subsubsection{The local closed pair}
Choose a number
\[
0<s_*<\frac{1}{4}\log 2.
\]
Choose real functions
\begin{equation}\label{eq:psiPdomain}
    \psi\in C_c^{\infty}((-s_*/2,s_*/2)),\qquad p\in C_c^{\infty}((-d,d)),
\end{equation}
where 
\[
0<d<\frac{1}{2}e^{-s_*},
\]
such that $\psi,p$ are nonnegative, $p$ is even, and
\[
\psi(0)>0,\qquad p(0)>0,\qquad \int_{\R}p(u)du>0.
\]
For $y=(y_1,y_2,y_3)$, define
\begin{equation}\label{eq:f0def}
    F_0(y)=\begin{cases}
        y_1^{-3/2}\psi(\log y_1)p(y_2/y_1)p(y_3/y_1), & y_1>0,\\
        0, & y_1\leq 0
    \end{cases}
\end{equation}
and
\[
F_1(y)=\begin{cases}
    \frac{y_3}{y_1}F_0(y), &y_1>0,\\
    0, &y_1\leq 0.
\end{cases}
\]
Both functions are smooth and compactly supported: on their support, $y_1$ stays in a compact subinterval of $(0,\infty)$. 

On the chart $y_1>0$, write
\[
r=y_1, \qquad x_1=\frac{y_2}{y_1},\qquad x_2=\frac{y_3}{y_1}.
\]
The linear actions of $V$ and $Z$ are
\[
V^{\mathrm{lin}}=\partial_{x_1},\qquad Z^{\mathrm{lin}}=x_2\partial_{x_1}.
\]
Consequently,
\begin{equation}\label{eq:VF1ZF0}
    V^{\mathrm{lin}}F_1=Z^{\mathrm{lin}}F_0.
\end{equation}

Let 
\[
\mu_0\coloneq\int_X\Theta_{F_0}^{\operatorname{prim}}dm=\frac{1}{\zeta(3)}\int_{\R^3}F_0(y)dy>0,
\]
and define
\[
f=\Theta_{F_0}^{\operatorname{prim}}-\mu_0,\qquad g=\Theta_{F_1}^{\operatorname{prim}}.
\]
By the definition of $\mu_0$,
\begin{equation}\label{eq:fHaarMeanZero}
    \int_Xfdm=0.
\end{equation}
Since $p$ is even, $F_1$ is odd in the $y_3$-variable. Hence Lemma~\ref{lem:priFirstTwoMoments} gives
\[
\int_Xgdm=\frac{1}{\zeta(3)}\int_{\R^3}F_1(y)dy=0.
\]
Moreover, $F_0\geq 0$ and therefore
\[
f\geq-\mu_0.
\]
By \eqref{eq:SiegelLieInter} and \eqref{eq:VF1ZF0},
\begin{equation}\label{eq:VgZf}
    Vg=\Theta_{ V^{\mathrm{lin}}F_1}^{\operatorname{prim}}=\Theta_ {Z^{\mathrm{lin}}F_0}^{\operatorname{prim}}=Zf.
\end{equation}

The restrictions on $s_*$ and $d$ ensure that only primitive vectors of the form $(1,n,0)^{\mathrm{T}}$ can contribute in the periodic-orbit calculation below. 

\subsection{The smooth time change}
For $u=(r,s)\in\R^2$, write
\[
\phi_ux\coloneq\exp(rV+sZ)x.
\]
We record the integration of a closed pair in a form also used in Section~\ref{sec:smoothTCcptSL3}.

\begin{Lemma}\label{lem:intScaCocycle}
    Let $\phi$ be a smooth $\R^2$ action with generators $V,Z$ on a smooth manifold $X$, and let $f,g\in C^{\infty}(X,\R)$ satisfy $Vg=Zf$. There is a unique smooth additive cocycle $c:X\times\R^2\to\R$ with infinitesimal generators $f$ and $g$. It is given by
    \[
    c(x,(r,s))=\int_0^1\Big(rf\big(\phi_{t(r,s)}x\big)+sg\big(\phi_{t(r,s)}x\big)\Big)dt,
    \]
    and satisfies
    \[
    \partial_r c(x,(r,s))=f(\phi_{(r,s)}x),\qquad \partial_s c(x,(r,s))=g(\phi_{(r,s)}x).
    \]
\end{Lemma}
\begin{proof}
    On the parameter plane, the form
    \[
    f(\phi_{(r,s)}x)dr+g(\phi_{(r,s)}x)ds
    \]
    is closed by $Vg=Zf$. It has a unique primitive vanishing at the origin. Integration along a radial segment gives the displayed formula and joint smoothness. Path independence and concatenation give
    \[
    c(x,p+q)=c(x,p)+c(\phi_px,q).
    \]
    Conversely, a smooth cocycle with the specified generators has this parameter differential and vanishes at the origin, so it coincides with the normalized primitive.
\end{proof}

Write $c(x,r,s)=c(x,(r,s))$ for the cocycle associated with the pair constructed above.

For $0<\epsilon<\mu_0^{-1}$, set $e_1=(1,0)$ and define
\[
\rho_{\epsilon}(x,u)\coloneq u+\epsilon c(x,u)e_1,\qquad x\in X,\quad u\in\R^2.
\]
Equivalently,
\[
\rho_{\epsilon}(x,(r,s))=\big(r+\epsilon c(x,r,s),s\big).
\]
\begin{Proposition}\label{prop:gloParaDiff}
For every $x\in X$, the parameter map $\rho_{\epsilon,x}:u\mapsto\rho_{\epsilon}(x,u)$ is a global smooth diffeomorphism of $\R^2$. Its inverse $\tau_{\epsilon}(x,\cdot)=\rho_{\epsilon,x}^{-1}$ depends smoothly on $(x,w)$, and, for $w\in\R^2$,
\[
\alpha_w^{\epsilon}x=\exp(\tau_{\epsilon,1}(x,w)V+\tau_{\epsilon,2}(x,w)Z)x
\]
defines a smooth $\R^2$-action with the same orbit partition as the algebraic action $\phi_u$.
\end{Proposition}
\begin{proof}
    The cocycle identity for $c$ (cf.~Lemma~\ref{lem:intScaCocycle}) gives
    \[
    \begin{aligned}
        \rho_{\epsilon}(x,u+v)=u+v+\epsilon\big(c(x,u)+c(\phi_ux,v)\big)e_1=\rho_{\epsilon}(x,u)+\rho_{\epsilon}(\phi_ux,v).
    \end{aligned}
    \]

    Fix $x\in X$ and $s\in\R$, and set
    \[
    h_{x,s}(r)\coloneq r+\epsilon c(x,r,s).
    \]
    By Lemma~\ref{lem:intScaCocycle},
    \[
    h'_{x,s}(r)=1+\epsilon f(\phi_{(r,s)}x)\geq 1-\epsilon\mu_0>0.
    \]
    Thus $h_{x,s}$ is strictly increasing. The same uniform lower bound gives
    \[
    h_{x,s}(r)\longrightarrow\pm\infty\qquad \text{as }r\longrightarrow\pm\infty,
    \]
    so $h_{x,s}$ is a smooth diffeomorphism of $\R$.

    Since the second coordinate of $\rho_{\epsilon,x}$ is unchanged, the unique preimage of $(R,S)\in\R^2$ is
    \[
    \left(h_{x,S}^{-1}(R),S\right).
    \]
    Hence $\rho_{\epsilon,x}$ is a global diffeomorphism. Joint smoothness of the inverse follows from the implicit function theorem, since
    \[
    \partial_r\big(r+\epsilon c(x,r,s)-R\big)\geq 1-\epsilon\mu_0>0,
    \]
    and uniqueness makes the local inverse branches agree globally.

    The inverse cocycle identity, as in Lemma~\ref{lem:contSameOrbPar}, gives the action law. Since $\tau_{\epsilon,x}$ is onto $\R^2$ for every $x$, the new action has the same orbits as $\phi_u$.
\end{proof}

\subsection{The restricted flow and its invariant measure}
Set
\[
T_t^{\epsilon}x\coloneq\alpha_{(t,0)}^{\epsilon}x,
\]
and
\[
a_{\epsilon}(x)\coloneq 1+\epsilon f(x)=(1-\epsilon\mu_0)+\epsilon\Theta_{F_0}^{\operatorname{prim}}(x)>0.
\]
For $x\in X$, set
\begin{equation}\label{eq:Rdef}
    R_{x,\epsilon}(q)\coloneq\rho_{\epsilon,1}(x,(q,0))=q+\epsilon c(x,(q,0))=\int_0^qa_{\epsilon}(\exp(rV)x)dr,
\end{equation}
and
\[
q_{\epsilon}(x,t)\coloneq \tau_{\epsilon,1}(x,(t,0)).
\]
Since the second coordinate of $\rho_{\epsilon}$ is unchanged, $
\tau_{\epsilon}(x,(t,0))=(q_{\epsilon}(x,t),0)$. Hence 
\begin{equation}\label{eq:qRidentity}
    q_{\epsilon}(x,t)=R_{x,\epsilon}^{-1}(t)
\end{equation} 
and therefore
\[
T_t^{\epsilon}x=\exp(q_{\epsilon}(x,t)V)x.
\]
\begin{Proposition}\label{prop:InvProDri}
The probability 
\[
d\nu_{\epsilon}=a_\epsilon dm
\]
is invariant under the full action $\alpha^{\epsilon}$, equivalent to Haar, and has full support. The restricted flow $T_t^{\epsilon}$ is ergodic; consequently, the full action $\alpha^{\epsilon}$ is ergodic. Moreover,
\[
\lim_{t\to+\infty}\frac{q_{\epsilon}(x,t)}{t}=1,\qquad \text{for $\nu_{\epsilon}$-a.e. $x$.}
\]  
\end{Proposition}
\begin{proof}
The integral of $a_{\epsilon}$ is one because $f$ has Haar mean zero, cf.~\eqref{eq:fHaarMeanZero}. The generator of $T_t^{\epsilon}$ is $a_{\epsilon}^{-1}V$. For every compactly supported smooth test function $\varphi$,
\[
\int_Xa_{\epsilon}^{-1}V\varphi d\nu_{\epsilon}=\int_X V\varphi dm=0,
\]
where the last equality follows by differentiating at $t=0$ the Haar-invariance identity
\[
\int_X\varphi(\exp(tV)x)dm(x)=\int_X \varphi dm.
\]
Thus $\nu_{\epsilon}$ is $T_t^{\epsilon}$-invariant. It is equivalent to Haar because 
\[
a_{\epsilon}\geq 1-\epsilon\mu_0>0,
\]
and hence has full support. Since the time-changed and algebraic flows have the same oriented $V$-orbits and $\nu_{\epsilon}\sim m$, Moore ergodicity of the $V$-flow gives ergodicity of $T_t^{\epsilon}$.

To prove the invariance under the full action, observe that
\[
D_u\rho_{\epsilon}(x,0)=\begin{pmatrix}
    a_{\epsilon} & \epsilon g(x)\\
    0 & 1
\end{pmatrix}.
\]
Hence the infinitesimal generators of $\alpha^{\epsilon}$ are
\[
\widetilde{V}=a_{\epsilon}^{-1}V,\qquad \widetilde{Z}=Z-\epsilon ga_{\epsilon}^{-1}V.
\]
The $\widetilde{V}$-invariance of $\nu_{\epsilon}$ was proved above. Moreover, for every $\varphi\in C_c^{\infty}(X)$,
\[
\int_X\widetilde{Z}\varphi d\nu_{\epsilon}=\int_X(a_{\epsilon}Z\varphi-\epsilon g V\varphi)dm=-\epsilon\int_X(Zf-Vg)\varphi dm=0
\]
by \eqref{eq:VgZf}. Thus $\nu_{\epsilon}$ is invariant under the full action $\alpha^{\epsilon}$. Since every $\alpha^{\epsilon}$-invariant measurable set is $T_t^{\epsilon}$-invariant, the full action is ergodic.

Finally, since
\[
R_{x,\epsilon}(q)/q=\frac{1}{q}\int_{0}^qa_{\epsilon}(\exp(rV)x)dr,
\]
the ergodic theorem gives
\[
\lim_{q\to\infty}\frac{R_{x,\epsilon}(q)}{q}=\int_X a_{\epsilon}dm=1
\]
for Haar-almost every $x$, hence for $\nu_{\epsilon}$-almost every $x$. Since $R_{x,\epsilon}$ is an increasing diffeomorphism, $q_{\epsilon}(x,t)\to+\infty$ as $t\to+\infty$. Equation~\eqref{eq:qRidentity} therefore gives the stated limit.
\end{proof}
\medskip

\subsection{Periodic data and smooth nonconjugacy}
For $s\in\R$, set
\[
a_s\coloneq \diag(e^s,e^{-s},1),\qquad x_s\coloneq a_s\Gamma,\qquad L_s\coloneq e^{-2s}.
\]
Since
\[
a_s^{-1}\exp(tV)a_s=\exp(te^{2s}V)=\id+te^{2s}E_{21},
\]
we have
\[
\exp(tV)x_s=x_s\Longleftrightarrow a_s^{-1}\exp(tV)a_s\in\Gamma\Longleftrightarrow te^{2s}\in\Z.
\]
Thus the stabilizer of $x_s$ in the $V$-flow is $e^{-2s}\Z=L_s\Z$, and its least positive period is $L_s$.
\begin{Lemma}\label{lem:disDiffNorPeriod}
    For every $s\in[0,s_*]$ and $t\in\R$,
    \[
    \Theta_{F_0}^{\operatorname{prim}}(\exp(tV)x_s)=e^{-3s/2}\psi(s)p(0)\sum_{n\in\Z}p(t+L_sn).
    \]
    Consequently,
    \[
    \frac{1}{L_s}\int_0^{L_s}\Theta_{F_0}^{\operatorname{prim}}(\exp(tV)x_s)dt=e^{s/2}\psi(s)p(0)\int_{\R}p(u)du.
    \]
    Moreover,
    \[
    \frac{R_{x_s,\epsilon}(L_s)}{L_s}=1-\epsilon\mu_0+\epsilon e^{s/2}\psi(s)p(0)\int_{\R}p(u)du.
    \]
    In particular, the normalized periods of 
    the orbits of $x_0$ and $x_{s_*}$ are different.
\end{Lemma}
\begin{proof}
Fix $s\in[0,s_*]$. For a primitive integer vector $v=(m,n,k)^{\mathrm{T}}\in\Z^3$,
\[
\exp(tV)a_sv=(e^sm,e^{-s}n+te^sm,k)^{\mathrm{T}}.
\]
If
\[
F_0\big(\exp(tV)a_sv\big)\neq 0,
\]
then \eqref{eq:f0def} implies that
\[
m>0,\qquad s+\log m\in(-s_*/2,s_*/2).
\]
Since $s\geq 0$, every $m\geq 2$ would give
\[
s+\log m\geq \log 2>s_*/2.
\]
Hence $m=1$. The second $p$-factor in the definition of $F_0$ is then $p(ke^{-s})$. Its nonvanishing implies
\[
\abs{k}<de^{s}\leq de^{s_*}<\frac{1}{2},
\]
and therefore $k=0$. Thus the only possible contributing primitive vectors are
\[
(1,n,0)^{\mathrm{T}}\qquad n\in\Z.
\]
Substitution into the definition of $F_0$ (cf.~\eqref{eq:f0def}) gives
\[
\Theta_{F_0}^{\operatorname{prim}}\big(\exp(tV)x_s\big)=e^{-3s/2}\psi(s)p(0)\sum_{n\in\Z}p(t+L_sn).
\]

Since $p$ is compactly supported (cf.~\eqref{eq:psiPdomain}),
\[
\int_0^{L_s}\sum_{n\in\Z}p(t+L_sn)dt=\sum_{n\in\Z}\int_{L_sn}^{L_s(n+1)}p(u)du=\int_{\R}p(u)du.
\]
Dividing by $L_s=e^{-2s}$ gives the second identity.

Finally,
\[
R_{x_s,\epsilon}(L_s)=\int_0^{L_s}a_{\epsilon}(\exp(tV)x_s)dt
\]
and 
\[
a_{\epsilon}=1-\epsilon\mu_0+\epsilon\Theta_{F_0}^{\operatorname{prim}}.
\]
This gives the formula for the normalized periods.
Its additional term is positive at $s=0$ and vanishes at $s=s_*$ because $\psi(0)>0$ and  $\psi(s_*)=0$.
\end{proof}

\begin{Proposition}[Failure of smooth algebraic conjugacy]\label{prop:nonSmoothConjugation}
    The flow $T_t^{\epsilon}$ is not smoothly conjugate by a leaf preserving diffeomorphism to any algebraic flow $x\mapsto \exp(tW)x$, $W\in\h$. Consequently, the full action $\alpha^{\epsilon}$ is not smoothly conjugate by such a diffeomorphism to a linear algebraic $\R^2$-action.
\end{Proposition}
\begin{proof}
    Suppose that 
    \[
    \Psi_BT_t^{\epsilon}=\exp(tW)\Psi_B,\qquad W=w_1V+w_2Z,
    \]
    for some smooth leaf preserving diffeomorphism $\Psi_B$, written as above. Identifying $\h$ with $\R^2$ through the ordered basis $(V,Z)$, set
    \[
    \mathbf{w}=(w_1,w_2).
    \]
    Since the $\exp(\h)$-stabilizers are discrete, differentiating the conjugacy at $t=0$ in infinitesimal leaf coordinates gives
    \begin{equation}\label{eq:conjPerData}
        e_1+VB=a_\epsilon\mathbf{w},
    \end{equation}
    where 
    \[
    VB=(VB_1,VB_2).
    \]

    Fix $s\in\{0,s_*\}$ and integrate \eqref{eq:conjPerData} along the closed algebraic $V$-orbit
    \[
    t\longmapsto\exp(tV)x_s,\qquad 0\leq t\leq L_s.
    \]
    Since $\exp(L_sV)x_s=x_s$ and $B$ is a globally defined smooth function on $X$,
    \[
    \begin{aligned}
        \int_0^{L_s}VB(\exp(tV)x_s)dt=B(\exp(L_sV)x_s)-B(x_s)=0.
    \end{aligned}
    \]
    Hence
    \[
    L_se_1=\mathbf{w}\int_{0}^{L_s}a_{\epsilon}(\exp(tV)x_s)dt=\mathbf{w} R_{x_s,\epsilon}(L_s).
    \]
    Since $a_{\epsilon}>0$, we have
    \[
    R_{x_s,\epsilon}(L_s)>0.
    \]
    Taking the two coordinates of the preceding identity gives
    \[
    w_2=0,\qquad w_1=\frac{L_s}{R_{x_s,\epsilon}(L_s)}=\left(\frac{R_{x_s,\epsilon}(L_s)}{L_s}\right)^{-1}.
    \]
    Thus $W=w_1V$, and its coefficient is the reciprocal of the normalized period of the orbit of $x_s$.

    Lemma~\ref{lem:disDiffNorPeriod} gives different 
    normalized periods at $s=0$ and $s=s_*$, contradicting the fixed value of $w_1$. A full action conjugacy would restrict to a conjugacy of $T^{\epsilon}$ with an algebraic flow, so it is excluded as well.
\end{proof}

\subsection{The compact root endpoint estimate}
Set
\[
K\coloneq E_{12}-E_{21},\qquad H\coloneq[V,K]=E_{22}-E_{11}.
\]
Then $\exp(\R K)$ is compact and
\begin{equation}\label{eq:AdrVK}
    \Ad_{\exp(rV)}K=K+rH-r^2V.
\end{equation}
Set
\[
c_{\epsilon}\coloneq 1-\epsilon\mu_0>0.
\]
Since $a_{\epsilon}\geq c_{\epsilon}$ and
\[
t=\int_0^{q_{\epsilon}(x,t)}a_{\epsilon}(\exp(rV)x)dr,
\]
we shall repeatedly use the uniform bound
\begin{equation}\label{eq:qEpsilonUpbound}
    \abs{q_{\epsilon}(x,t)}\leq c_{\epsilon}^{-1}\abs{t},\qquad x\in X,\quad t\in\R.
\end{equation}
\begin{Proposition}\label{prop:tanQuaEst}
For every $t\in\R$, the derivative
\[
D_Kq_{\epsilon}(x,t)\coloneq\left.\frac{d}{ds}\right|_{s=0}q_{\epsilon}(\exp(sK)x,t)
\]
exists, belongs to $L^2(\nu_{\epsilon})$, and satisfies
\[
\norm{D_Kq_{\epsilon}(\cdot,t)}_{L^2(\nu_{\epsilon})}\leq C_{\epsilon}(1+\abs{t})^2.
\]
\end{Proposition}
\begin{proof}
Define
\[
a\coloneq a_{\epsilon},\qquad q_t(x)\coloneq q_{\epsilon}(x,t),
\]
and write $u_r=\exp(rV)$ and $k_s=\exp(sK)$. All integrals below are understood as oriented integrals.

The joint smoothness of the inverse clock gives the existence of the derivative. By \eqref{eq:Rdef}, its defining equation is
\[
t=\int_0^{q_t(x)}a(u_rx)dr=q_t(x)+\epsilon\int_0^{q_t(x)}f(u_rx)dr.
\]
Differentiating the corresponding identity at $k_sx$ and $s=0$ gives
\[
0=a(T_t^{\epsilon}x)D_Kq_{\epsilon}(x,t)+\epsilon\mathcal{J}_K(x,q_t(x)),
\]
and therefore
\begin{equation}\label{eq:DkqEpsilon}
    D_Kq_{\epsilon}(x,t)=-\epsilon\frac{\mathcal{J}_K(x,q_t(x))}{a(T_t^{\epsilon}x)},
\end{equation}
where
\[
\mathcal{J}_K(x,q)\coloneq\frac{d}{ds}\big|_{s=0}\int_0^qf(u_rk_sx)dr.
\]

Since
\[
u_rk_s=\exp(s\Ad_{u_r}K)u_r,
\]
equation~\eqref{eq:AdrVK} gives
\[
\mathcal{J}_K(x,q)=\int_0^q\big((K+rH-r^2V)f\big)(u_rx)dr.
\]
Using $\frac{d}{dr}f(u_rx)=(Vf)(u_rx)$ and integration by parts, we obtain
\begin{equation}\label{eq:Jkxq}
    \mathcal{J}_K(x,q)=\int_0^q(Kf)(u_rx)dr+\int_0^qr(Hf)(u_rx)dr-q^2f(u_qx)+2\int_0^qrf(u_rx)dr.
\end{equation}

Moreover, Lemma~\ref{lem:priFirstTwoMoments} gives $f, Kf, Hf\in L^2(m)$. Hence, since $a\geq c_{\epsilon}$, for $h\in\{f,Kf,Hf\}$, 
\begin{equation}\label{eq:h/aL2estimate}
    \norm{h/a}_{L^2(\nu_{\epsilon})}^2=\int_X\frac{\abs{h}^2}{a}dm\leq c_{\epsilon}^{-1}\norm{h}_{L^2(m)}^2.
\end{equation}
In particular, $(Hf+2f)/a\in L^2(\nu_{\epsilon})$.

Make the change of variables
\[
u=R_{x,\epsilon}(r),\qquad u_rx=T_u^{\epsilon}x,\qquad r=q_{\epsilon}(x,u),\qquad dr=\frac{du}{a(T^{\epsilon}_ux)}.
\]
Equation~\eqref{eq:Jkxq} then becomes
\[
\mathcal{J}_K(x,q_t(x))=\int_0^t\frac{Kf}{a}(T_u^{\epsilon}x)du+\int_0^tq_{\epsilon}(x,u)\frac{Hf+2f}{a}(T_u^{\epsilon}x)du-q_t(x)^2f(T_t^{\epsilon}x).
\]
By Minkowski's inequality, the $T_u^{\epsilon}$-invariance of $\nu_{\epsilon}$, and \eqref{eq:qEpsilonUpbound},
\[
\norm{\int_0^t\frac{Kf}{a}(T_u^{\epsilon}\cdot)du}_{L^2(\nu_{\epsilon})}\leq \abs{t}\norm{Kf/a}_{L^2(\nu_{\epsilon})}
\]
and
\[
\norm{\int_0^tq_{\epsilon}(\cdot,u)\frac{Hf+2f}{a}(T_u^{\epsilon}\cdot)du}_{L^2(\nu_{\epsilon})}\leq\frac{\abs{t}^2}{2c_{\epsilon}}\norm{\frac{Hf+2f}{a}}_{L^2(\nu_{\epsilon})}.
\]
For the boundary term, invariance gives
\[
\norm{q_t^2\frac{f}{a}(T_t^{\epsilon}\cdot)}_{L^2(\nu_{\epsilon})}\leq c_{\epsilon}^{-2}\abs{t}^2\norm{f/a}_{L^2(\nu_{\epsilon})}.
\]
Finally, $a^{-1}\leq c_{\epsilon}^{-1}$ controls the terminal denominator in the two integral terms of \eqref{eq:DkqEpsilon}. Together with \eqref{eq:h/aL2estimate}, these estimates prove the proposition.
\end{proof}

\subsection{The finite compact root increment}
To pass from the derivative estimate to finite increments, we use $d\nu_{\epsilon}=a_{\epsilon}\,dm$ and estimate compact root translates of both $\sqrt{a_{\epsilon}}$ and $\sqrt{a_{\epsilon}}q_{\epsilon}(\cdot,t)$ in $L^2(m)$. Throughout this subsection, set
\[
k_s\coloneq\exp(sK),\qquad a\coloneq a_{\epsilon},\qquad q_t(x)\coloneq q_{\epsilon}(x,t).
\]
\begin{Lemma}\label{lem:halfDensityCptRtEst}
    We have
    \begin{equation*}
        \begin{aligned}
            \norm{K\sqrt{a}}_{L^2(m)}&\leq C_{\epsilon},\\
            \norm{K(\sqrt{a}q_t)}_{L^2(m)}&\leq C_{\epsilon}(1+\abs{t})^2.
        \end{aligned}
    \end{equation*}
\end{Lemma}
\begin{proof}
    Since $a\geq c_{\epsilon}=1-\epsilon\mu_0>0$, 
    \[
    K\sqrt{a}=\frac{\epsilon Kf}{2\sqrt{a}}.
    \]
    Hence, by Lemma~\ref{lem:priFirstTwoMoments},
    \[
        \norm{K\sqrt{a}}_{L^2(m)}^2\leq\frac{\epsilon^2}{4c_{\epsilon}}\norm{Kf}_{L^2(m)}^2<\infty. 
    \]
    This proves the first estimate.

    \medskip

    Since $\int_Xa\,dm=1$, we have $\sqrt{a}\in L^2(m)$. Then \eqref{eq:qEpsilonUpbound} gives
    \[
    \norm{\sqrt{a}q_t}_{L^2(m)}^2\leq c_{\epsilon}^{-2}\abs{t}^2\int_Xa\,dm=c_{\epsilon}^{-2}\abs{t}^2.
    \]
    Thus $\sqrt{a}q_t\in L^2(m)$. By Proposition~\ref{prop:tanQuaEst}, 
    \[
    Kq_t=D_Kq_{\epsilon}(\cdot,t)\in L^2(\nu_{\epsilon}),\qquad \norm{Kq_t}_{L^2(\nu_{\epsilon})}\leq C_{\epsilon}(1+\abs{t})^2.
    \]
    Since $a>0$ and $q_t$ is smooth, the product rule gives
    \[
    K(\sqrt{a}q_t)=\sqrt{a}Kq_t+q_tK\sqrt{a}.
    \]
    Therefore
    \[
    \begin{aligned}
        \norm{K(\sqrt{a}q_t)}_{L^2(m)}\leq&\norm{\sqrt{a}Kq_t}_{L^2(m)}+\norm{q_tK\sqrt{a}}_{L^2(m)}\\
        \leq& \norm{Kq_t}_{L^2(\nu_{\epsilon})}+c_{\epsilon}^{-1}\abs{t}\norm{K\sqrt{a}}_{L^2(m)}\leq C_{\epsilon}(1+\abs{t})^2.
    \end{aligned}
    \]
    This proves the second estimate.
\end{proof}

\begin{Theorem}\label{thm:finCptRtInreClosedBd}
For every $s,t\in\R$,
\[
\norm{q_{\epsilon}(k_s\cdot,t)-q_{\epsilon}(\cdot,t)}_{L^2(\nu_{\epsilon})}\leq C_{\epsilon}\abs{s}(1+\abs{t})^2.
\]    
\end{Theorem}
\begin{proof}
For smooth $F$ with $F,KF\in L^2(m)$, Haar invariance and Minkowski's inequality applied to the fundamental theorem of calculus give
\begin{equation}\label{eq:ksDiff}
    \norm{F\circ k_s-F}_{L^2(m)}\leq\abs{s}\norm{KF}_{L^2(m)}.
\end{equation}

We have
\[
\begin{aligned}
    \sqrt{a}(q_{t}\circ k_s-q_t)=\left((\sqrt{a} q_t)\circ k_s-\sqrt{a}q_t\right)+\left(\sqrt{a}-(\sqrt{a})\circ k_s\right)(q_{t}\circ k_s).
\end{aligned}
\]
By \eqref{eq:ksDiff} and Lemma~\ref{lem:halfDensityCptRtEst},
\[
\norm{(\sqrt{a}q_t)\circ k_s-\sqrt{a}q_t}_{L^2(m)}\leq C_{\epsilon}\abs{s}(1+\abs{t})^2.
\]
Moreover, \eqref{eq:qEpsilonUpbound} gives
    \[
    \abs{q_{t}(k_sx)}\leq c_{\epsilon}^{-1}\abs{t},
    \] 
and hence, again by \eqref{eq:ksDiff},
\[
\norm{(\sqrt{a}-(\sqrt{a})\circ k_s)(q_{t}\circ k_s)}_{L^2(m)}\leq c_{\epsilon}^{-1}\abs{t}\norm{\sqrt{a}-(\sqrt{a})\circ k_s}_{L^2(m)}\leq C_{\epsilon}\abs{s}\abs{t}.
\]
Finally, since $d\nu_{\epsilon}=a\,dm$,
\[
\norm{q_{t}\circ k_s-q_t}_{L^2(\nu_{\epsilon})}=\norm{\sqrt{a}(q_{t}\circ k_s-q_t)}_{L^2(m)}\leq C_{\epsilon}\abs{s}(1+\abs{t})^2.
\]
\end{proof}

\begin{proof}[Proof of Theorem~\ref{thm:sharpnessSL3rigidityOneRoot}]
    Lemma~\ref{lem:intScaCocycle} and Proposition~\ref{prop:gloParaDiff} construct the smooth family $\alpha^{\epsilon}$. Since
    \[
    \rho_{\epsilon}(x,u)=u+\epsilon c(x,u)e_1
    \]
    and 
    \[
    \partial_r\rho_{\epsilon,1}(x,r,s)=1+\epsilon f(\phi_{(r,s)}x)\geq 1-\epsilon\mu_0,
    \]
    the parameter-dependent inverse theorem gives $\tau_{\epsilon}\to\id$ smoothly on compact sets. Hence $\alpha^{\epsilon}$ converges to the algebraic $\h$-action as $\epsilon\to0$.
    
    Proposition~\ref{prop:InvProDri} gives the invariant measure and drift of the selected flow. Theorem~\ref{thm:finCptRtInreClosedBd} gives the boundary estimate, while Proposition~\ref{prop:nonSmoothConjugation} gives both nonconjugacy statements. 

    Let $\widehat{\alpha}^{\epsilon}$ be the transported action defined in \eqref{eq:transportNonRigidity}, and set
    \[
    \widehat T_t^{\epsilon}\coloneq\widehat\alpha_{(t,0)}^{\epsilon},\qquad \widehat\nu_{\epsilon}\coloneq (C_{w_0})_*\nu_{\epsilon},\qquad \widehat{q_{\epsilon}}(y,t)\coloneq q_{\epsilon}(C_{w_0}^{-1}y,t).
    \]
    Then
    \[
    \widehat T_t^{\epsilon}y=\exp\big(\widehat q_{\epsilon}(y,t)U_2\big)y
    \]
    and the drift is $U_2$. Since
    \[
    C_{w_0}\circ k_s\circ C_{w_0}^{-1}=\exp(sK_2),
    \]
    a change of variables gives
    \[
    \norm{\widehat q_{\epsilon}(\exp(sK_2)\cdot,t)-\widehat q_{\epsilon}(\cdot,t)}_{L^2(\widehat{\nu_{\epsilon}})}=\norm{q_{\epsilon}(k_s\cdot,t)-q_{\epsilon}(\cdot,t)}_{L^2(\nu_{\epsilon})}\leq C_{\epsilon}\abs{s}(1+\abs{t})^2.
    \]
    The fixed algebraic conjugation $C_{w_0}$ carries leaf preserving diffeomorphisms  for the $\exp(\h)$-action to leaf preserving diffeomorphisms for the first row action, and the reflection of the second parameter is linear. Hence the two nonconjugacy conclusions also pass to the transported action. This proves assertions \ref{item:sharpSL31}--\ref{item:sharpSL33} of Theorem~\ref{thm:sharpnessSL3rigidityOneRoot}.

    The displayed estimate has $(\delta,\gamma)=(1,2)$. If $\widehat{T}_t^{\epsilon}$ satisfied a one-root estimate with $2\delta>\gamma$, then Theorem~\ref{thm:oneRootSL3rigidity} would make it agree with its algebraic drift on a full $\widehat{\nu}_{\epsilon}$-measure set. Full support and continuity would extend the equality to all of $X$, contradicting the transported form of Proposition~\ref{prop:nonSmoothConjugation}.
\end{proof}

\section{Pointwise sharpness for compact Ratner frames}\label{sec:PtSharpCptRatnerFra}
The final example addresses the pointwise abstract theorem. 
We take the product of two smooth time changes of horocycle flows on compact quotients of $\PSL_2(\R)$. In each factor, the invariant distributions of Flaminio and Forni \cite{FlaminioForni03InvariantDistributions} provide a smooth mean zero function which is not a smooth horocycle coboundary, while neither the full action nor its diagonal restriction admits a smooth leaf preserving conjugacy to an algebraic action.

This section proves Theorem~\ref{thm:sharpnessAbsCptRigid}. For $i=1,2$, let $G_i=\PSL_2(\R)$, $\Gamma_i<G_i$ a torsion-free and cocompact lattice, and set $X_i=G_i/\Gamma_i$. Let $m_i$ be  Haar probability measure on $X_i$, and set
\[
X=X_1\times X_2,\qquad m=m_1\otimes m_2.
\]

In $\sl_2(\R)$, define
\[
U\coloneq E_{12},\qquad K\coloneq E_{21}-E_{12},\qquad H\coloneq E_{11}-E_{22}.
\]
For $i=1,2$, let $\mathcal{U}_i$, $\mathcal{K}_i$ and $\mathcal{H}_i$ denote the infinitesimal generators on $X_i$ associated with $U$, $K$ and $H$, respectively. Let
\[
U^{(i)},\qquad K^{(i)},\qquad H^{(i)}
\]
denote the corresponding elements in the $i$-th factor of 
\[
\mathfrak{g}=\sl_2(\R)\oplus\sl_2(\R).
\]
Set 
\[
\mathfrak{u}=\R U^{(1)}\oplus \R U^{(2)},\qquad N=\exp(\mathfrak{u}).
\]

For $i=1,2$, define
\[
h_t^{(i)}x_i\coloneq\exp(tU)x_i
\]
for the horocycle flow on $X_i$. We use the same symbol $U^{(i)}$ for the induced vector field on the product $X$. The algebraic $N$-action is
\[
\phi_{(p_1,p_2)}(x_1,x_2)=\big(\exp(p_1U)x_1,\exp(p_2U)x_2\big).
\]

\subsection{The two nontrivial time changes}
\begin{Lemma}\label{lem:choicesClockFunctions}
    For each $i=1,2$, there is a real function $f_i\in C^{\infty}(X_i)$ such that
    \[
    \int_{X_i}f_i\,dm_i=0
    \]
    and the cohomological equation
    \[
    \mathcal{U}_ig=f_i
    \]
    has no smooth solution $g\in C^{\infty}(X_i)$.
\end{Lemma}
\begin{proof}
    Inversion identifies $X_i=G_i/\Gamma_i$ with the unit tangent bundle model $\Gamma_i\backslash\PSL_2(\R)$. Under this identification, left translation by $\exp(tU)$ corresponds to right translation by $\exp(-tU)$. Since replacing $U$ by $-U$ does not change the space of invariant distributions, Flaminio and Forni's classification \cite{FlaminioForni03InvariantDistributions}*{Theorem~1.1} gives, after pulling back by inversion, a complex $\mathcal{U}_i$-invariant distribution which is not a complex multiple of $m_i$. Taking a real or imaginary part if necessary, choose a real $\mathcal{U}_i$-invariant distribution $\mathcal{D}_i$ which is not a real multiple of $m_i$. 

    The two continuous linear functionals
    \[
    f\mapsto \mathcal{D}_i(f),\qquad f\mapsto \int_{X_i}f\,dm_i
    \]
    are therefore linearly independent on $C^{\infty}(X_i,\R)$. Hence there is a real $f_i\in C^{\infty}(X_i)$ such that
    \[
    \int_{X_i}f_i\,dm_i=0,\qquad \mathcal{D}_i(f_i)\neq0.
    \]
    
    The invariance of $\mathcal{D}_i$ means equivalently that
    \[
    \mathcal{D}_i(\mathcal{U}_i\varphi)=0\qquad \text{for every }\varphi\in C^{\infty}(X_i,\R).
    \]
    If $\mathcal{U}_ig=f_i$ for some smooth $g$, then
    \[
    \mathcal{D}_i(f_i)=\mathcal{D}_i(\mathcal{U}_ig)=0,
    \]
    contradicting the choice of $f_i$.
\end{proof}

Fix the functions $f_1,f_2$. Choose $\epsilon_0>0$ such that, for $i=1,2$,
\[
a_{i,\epsilon}\coloneq1+\epsilon f_i>0
\]
on $X_i$ whenever $\abs{\epsilon}<\epsilon_0$. Since $f_i$ has zero mean,
\begin{equation}\label{eq:zeroAmean}
    \int_{X_i}a_{i,\epsilon}\,dm_i=1.
\end{equation}
For $x_i\in X_i$ and $r\in\R$, define 
\[
R_{i,x_i}^{\epsilon}(r)\coloneq\int_0^ra_{i,\epsilon}(h_s^{(i)}x_i)\,ds.
\]
Then
\[
\partial_rR_{i,x_i}^{\epsilon}(r)=a_{i,\epsilon}(h_r^{(i)}x_i).
\]
Since $X_i$ is compact, there are constants
\[
0<m_{i,\epsilon}\leq M_{i,\epsilon}<\infty
\]
such that
\begin{equation}\label{eq:boundPartialDeriR}
    m_{i,\epsilon}\leq\partial_r R_{i,x_i}^{\epsilon}(r)\leq M_{i,\epsilon}
\end{equation}
uniformly in $x_i$ and $r$. In particular, $R_{i,x_i}^{\epsilon}(r)\to\pm\infty$ as $r\to\pm\infty$. Thus every $R_{i,x_i}^{\epsilon}$ is a smooth increasing diffeomorphism of $\R$.
\subsection{The product time change}
Define
\begin{equation}\label{eq:productCocycle}
    \rho_{\epsilon}((x_1,x_2),(p_1,p_2))=(R_{1,x_1}^{\epsilon}(p_1),R_{2,x_2}^{\epsilon}(p_2)).
\end{equation}
We regard functions on $X_i$ as functions on $X$ through the $i$-th coordinate projection.

\begin{Lemma}\label{lem:productRepara}
    The map $\rho_{\epsilon}$ is a smooth reparametrization cocycle for $\phi$. Every parameter map is a smooth diffeomorphism of $\R^2$, and the inverse family is jointly smooth. The action induced by $\rho_{\epsilon}$ is the product of two time changed horocycle flows whose product generators are
    \[
    \widetilde{U}_{\epsilon}^{(i)}=a_{i,\epsilon}^{-1}U^{(i)},\qquad i=1,2.
    \]
    Moreover,
    \[
    \widetilde{U}_{\epsilon}^{(i)}\longrightarrow U^{(i)} \quad \text{in } C^{\infty},\qquad i=1,2,
    \]
    as $\epsilon\to0$.
\end{Lemma}
\begin{proof}
    For each factor and all $p,q\in\R$,
    \[
    R_{i,x_i}^{\epsilon}(p+q)=R_{i,x_i}^{\epsilon}(p)+R_{i,h_p^{(i)}x_i}^{\epsilon}(q).
    \]
    Hence \eqref{eq:productCocycle} satisfies the cocycle identity. Every parameter map is the product of two smooth diffeomorphisms of $\R$. Its inverse is
    \[
    \tau_{\epsilon}\big((x_1,x_2),(w_1,w_2)\big)=\left((R_{1,x_1}^{\epsilon})^{-1}(w_1),(R_{2,x_2}^{\epsilon})^{-1}(w_2)\right),
    \]
    which is jointly smooth by \eqref{eq:boundPartialDeriR} and the implicit-function theorem.

    Since 
    \[
    \partial_rR_{i,x_i}^{\epsilon}(0)=a_{i,\epsilon}(x_i),
    \]
    the $i$-th factor of the induced action is generated by $a_{i,\epsilon}^{-1}U^{(i)}$. Finally,
    \[
    a_{i,\epsilon}^{-1}=(1+\epsilon f_i)^{-1}\longrightarrow1\quad \text{in }C^{\infty},
    \]
    which proves the last assertion.
\end{proof}

Denote the induced action by $\alpha^{\epsilon}$, so that
\[
\alpha^{\epsilon}_{\rho_{\epsilon}(x,p)}x=\phi_px.
\]
Set
\[
T_t^{\epsilon}\coloneq\alpha^{\epsilon}_{(t,t)}.
\]

For $i=1,2$, let $q_{i,\epsilon}(x_i,\cdot)$ be the inverse 
of $R_{i,x_i}^{\epsilon}$, so that
\begin{equation}\label{eq:inverseClockq}
    R_{i,x_i}^{\epsilon}(q_{i,\epsilon}(x_i,t))=t.
\end{equation}
Then
\[
T_t^{\epsilon}(x_1,x_2)=\left(h_{q_{1,\epsilon}(x_1,t)}^{(1)}x_1,h_{q_{2,\epsilon}(x_2,t)}^{(2)}x_2\right).
\]
The physical cocycle of the diagonal restriction is therefore
\begin{equation}\label{eq:productCocycleInverse}
    \beta_{\epsilon}(x,t)=q_{1,\epsilon}(x_1,t)U^{(1)}+q_{2,\epsilon}(x_2,t)U^{(2)}.
\end{equation}

\subsection{Invariant measure, mixing, and drift}
For each factor set
\[
d\nu_{i,\epsilon}\coloneq a_{i,\epsilon}\,dm_i,
\]
and let
\[
\nu_{\epsilon}\coloneq \nu_{1,\epsilon}\otimes\nu_{2,\epsilon}.
\]
\begin{Proposition}\label{prop:measureDrift}
    The measure $\nu_{\epsilon}$ is a full-support probability invariant under the full $\R^2$-action $\alpha^{\epsilon}$. The diagonal flow $T^{\epsilon}$ is mixing with respect to $\nu_{\epsilon}$. Moreover,
    \[
    \lim_{t\to+\infty}\frac{\beta_{\epsilon}(x,t)}{t}=U^{(1)}+U^{(2)},\qquad \text{$\nu_{\epsilon}$-a.e. $x$.}
    \]
\end{Proposition}
\begin{proof}
    Equation~\eqref{eq:zeroAmean} shows that each $\nu_{i,\epsilon}$ is a probability. Since the $i$-th factor flow is generated by $a_{i,\epsilon}^{-1}\mathcal{U}_i$,
    \[
    \int_{X_i}a_{i,\epsilon}^{-1}\mathcal{U}_ig\,d\nu_{i,\epsilon}=\int_{X_i}\mathcal{U}_ig\,dm_i=0
    \]
    for every smooth $g$ on $X_i$. Thus $\nu_{i,\epsilon}$ is invariant under the $i$-th time-changed flow. Since $a_{i,\epsilon}>0$, it is equivalent to $m_i$ and has full support. Hence $\nu_{\epsilon}$ is a full-support probability invariant under the product action.

    By the quantitative mixing theorem for smooth time changes of compact horocycle flows \cite{forni2012time}, each flow is mixing. Therefore their product $T^{\epsilon}$ is mixing.

    Differentiating \eqref{eq:inverseClockq} and using $q_{i,\epsilon}(x_i,0)=0$ gives
    \[
    q_{i,\epsilon}(x_i,t)=\int_0^t a_{i,\epsilon}^{-1}\left(h_{q_{i,\epsilon}(x_i,s)}^{(i)}x_i\right)\,ds.
    \]
    Since the $i$-th factor flow is ergodic, the ergodic theorem gives
    \[
    \lim_{t\to+\infty}\frac{q_{i,\epsilon}(x_i,t)}{t}=\int_{X_i}a_{i,\epsilon}^{-1}d\nu_{i,\epsilon}=\int_{X_i}1dm_i=1
    \]
    for $\nu_{i,\epsilon}$-a.e. $x_i$. Equation \eqref{eq:productCocycleInverse} then completes the proof.
\end{proof}

\subsection{The compact Ratner frame}
The $\sl_2$-relations are
\begin{equation}\label{eq:sl2product}
    [U,K]=H,\qquad [U,H]=-2U.
\end{equation}
Consequently,
\begin{equation}\label{eq:adUcomputations}
    \frac{1}{2}\ad_U^2K=-U,\qquad \ad_U^3K=0.
\end{equation}
\begin{Lemma}\label{lem:productCptFra}
    For $V=U^{(1)}+U^{(2)}$, the family
    \[
    (K^{(1)},2,-U^{(1)}),\qquad (K^{(2)},2,-U^{(2)})
    \]
    is a compact Ratner frame for $(\mathfrak{u},V)$.
\end{Lemma}
\begin{proof}
    For each $i=1,2$, the element $K^{(i)}$ generates the compact rotation subgroup in the $i$-th factor. Since the two factors commute, \eqref{eq:adUcomputations} gives
    \[
    \frac{1}{2}\ad_V^{2}K^{(i)}=-U^{(i)}.
    \]
    Moreover,
    \[
    \ad_{Z_1}\ad_{Z_2}\ad_{Z_3}K^{(i)}=0
    \]
    for all $Z_1,Z_2,Z_3\in\mathfrak{u}$. Finally, $(-U^{(1)},-U^{(2)})$ is a basis of $\mathfrak{u}$.
\end{proof}

\subsection{The pointwise estimate}
For $i=1,2$, define
\[
A_i(x_i,q)\coloneq \int_0^qf_i(h_r^{(i)}x_i)dr,
\]
so that
\[
R_{i,x_i}^{\epsilon}(q)=q+\epsilon A_i(x_i,q).
\]
The relations in \eqref{eq:sl2product} give
\begin{equation}\label{eq:adKHUDeco}
    \Ad_{\exp(rU)}K=K+rH-r^2U.
\end{equation}
\begin{Lemma}\label{lem:quadraticCptRtDis}
    For each $i=1,2$, there is $C_i>0$ such that
    \[
    \abs{A_i(\exp(sK)x_i,q)-A_i(x_i,q)}\leq C_i\abs{s}(1+\abs{q})^2
    \]
    for every $s,q\in\R$, and $x_i\in X_i$.
\end{Lemma}
\begin{proof}
    Set $x_{i,\sigma}\coloneq\exp(\sigma K)x_i$. Since
    \[
    \exp(rU)\exp\big((\sigma+\eta) K\big)=\exp(\eta\Ad_{\exp(rU)}K)\exp(rU)\exp(\sigma K),
    \]
    equation~\eqref{eq:adKHUDeco} gives
    \[
    \frac{d}{d\sigma}f_i(h_r^{(i)}x_{i,\sigma})=\bigl((\mathcal{K}_i+r\mathcal{H}_i-r^2\mathcal{U}_i)f_i\bigr)(h_r^{(i)}x_{i,\sigma}).
    \]
    Differentiating under the oriented integral, we obtain
    \begin{equation}\label{eq:dSigmaAi}
        \frac{d}{d\sigma}A_i(x_{i,\sigma},q)=\int_0^q\bigl((\mathcal{K}_i+r\mathcal{H}_i-r^2\mathcal{U}_i)f_i\bigr)(h_r^{(i)}x_{i,\sigma})dr.
    \end{equation}
     
    The first two terms on the RHS of \eqref{eq:dSigmaAi} are bounded respectively by
    \[
    \norm{\mathcal{K}_if_i}_{\infty}\abs{q}\qquad\text{and}\qquad \frac{1}{2}\norm{\mathcal{H}_if_i}_{\infty}q^2.
    \]
    For the last term, since
    \[
    \frac{d}{dr}f_i(h_r^{(i)}x_{i,\sigma})=(\mathcal{U}_if_i)(h_r^{(i)}x_{i,\sigma}),
    \]
    integration by parts gives
    \[
    -\int_0^qr^2(\mathcal{U}_if_i)(h_r^{(i)}x_{i,\sigma})dr=-q^2f_i(h_q^{(i)}x_{i,\sigma})+2\int_0^qrf_i(h_r^{(i)}x_{i,\sigma})dr.
    \]
    Hence its absolute value is at most
    \[
    2\norm{f_i}_{\infty}q^2.
    \]
    It follows from \eqref{eq:dSigmaAi} that
    \[
    \abs{\frac{d}{d\sigma}A_i(x_{i,\sigma},q)}\leq C_i(1+\abs{q})^2
    \]
    uniformly in $x_i$, $\sigma$ and $q$. Therefore,
    \[
    A_i(\exp(sK)x_i,q)-A_i(x_i,q)=\int_0^s\frac{d}{d\sigma}A_i(x_{i,\sigma},q)d\sigma,
    \]
    which proves the lemma.
\end{proof}

\begin{Proposition}\label{prop:endptEstPhyCoc}
For each $i=1,2$, there is $C_{i,\epsilon}>0$ such that
\[
\abs{q_{i,\epsilon}(\exp(sK)x_i,t)-q_{i,\epsilon}(x_i,t)}\leq C_{i,\epsilon}\abs{s}(1+\abs{t})^2
\]
for every $s,t\in\R$, and $x_i\in X_i$. In particular, the compact root estimate in part~\ref{item:sharpnessAbstract2} of Theorem~\ref{thm:sharpnessAbsCptRigid} holds.
\end{Proposition}
\begin{proof}
Set
\[
x'_i\coloneq \exp(sK)x_i,\qquad q=q_{i,\epsilon}(x_i,t),\qquad q'=q_{i,\epsilon}(x'_i,t).
\]
By the definition of the inverse clocks,
\begin{equation}\label{eq:diffInverseClocks}
    R_{i,x_i}^{\epsilon}(q)=t=R_{i,x'_i}^{\epsilon}(q').
\end{equation}
The derivative bound \eqref{eq:boundPartialDeriR} gives
\[
m_{i,\epsilon}\abs{q-q'}\leq\abs{R_{i,x'_i}^{\epsilon}(q')-R_{i,x'_i}^{\epsilon}(q)}.
\]
Using \eqref{eq:diffInverseClocks} and $R_{i,x_i}^{\epsilon}(r)=r+\epsilon A_i(x_i,r)$, the RHS equals
\[
\abs{R_{i,x_i}^{\epsilon}(q)-R_{i,x'_i}^{\epsilon}(q)}=\abs{\epsilon}\abs{A_i(x_i,q)-A_i(x'_i,q)}\leq\abs{\epsilon}C_i\abs{s}(1+\abs{q})^2
\]
by Lemma~\ref{lem:quadraticCptRtDis}.

Moreover, since $R_{i,x_i}^{\epsilon}(0)=0$, 
\[
m_{i,\epsilon}\abs{q}\leq\abs{R_{i,x_i}^{\epsilon}(q)-R_{i,x_i}^{\epsilon}(0)}=\abs{t}.
\]
Consequently,
\[
\abs{q'-q}\leq\frac{\abs{\epsilon}C_i}{m_{i,\epsilon}}\abs{s}\left(1+\frac{\abs{t}}{m_{i,\epsilon}}\right)^2\leq C_{i,\epsilon}\abs{s}(1+\abs{t})^2.
\]

Finally, $K^{(i)}$ acts only on the $i$-th factor, while the other inverse clock in \eqref{eq:productCocycleInverse} is unchanged. Thus the last estimate implies the $K^{(i)}$-seminorm bound in \ref{item:sharpnessAbstract2} of Theorem~\ref{thm:sharpnessAbsCptRigid}.
\end{proof}

\subsection{Nonconjugacy to algebraic actions}
The algebraic $N$-action on $X$ is free. Indeed, a nontrivial stabilizer would give a nontrivial unipotent element in one of the cocompact lattices $\Gamma_i$, which is impossible by the contraction argument used in Lemma~\ref{lem:algStab}.

\begin{Proposition}[Diagonal flow nonconjugacy]\label{prop:nonSmoothAlgDiag}
    For every sufficiently small nonzero $\epsilon$, the flow $T^{\epsilon}$ is not smoothly conjugate by a leaf preserving diffeomorphism to any algebraic flow $x\mapsto\exp(tW)x$, $W\in\mathfrak{u}$. Consequently, the action $\alpha^{\epsilon}$ is not smoothly conjugate by such a diffeomorphism to a linear algebraic $N$-action.
\end{Proposition}
\begin{proof}
   By Lemma~\ref{lem:smmothDisplacement}, every smooth leaf preserving conjugacy has the form $\Psi_F$ with $F\in C^{\infty}(X,\mathfrak{u})$. Identifying $\mathfrak{u}$ with $\R^2$ through the ordered basis $(U^{(1)},U^{(2)})$, suppose that 
   \[
   \Psi_FT_t^{\epsilon}=\exp(tW)\Psi_F,\qquad W=c_1U^{(1)}+c_2U^{(2)},
   \]
   for some smooth $F=(F_1,F_2):X\to\R^2$. Freeness and \eqref{eq:productCocycleInverse} give
   \[
   F(T_t^{\epsilon}x)+(q_{1,\epsilon}(x_1,t),q_{2,\epsilon}(x_2,t))=F(x)+t(c_1,c_2).
   \]
   Differentiating at $t=0$ gives
   \begin{equation}\label{eq:diffFResFlow}
       X_{\epsilon}F+(a_{1,\epsilon}^{-1},a_{2,\epsilon}^{-1})=(c_1,c_2),
   \end{equation}
   where $X_{\epsilon}=a_{1,\epsilon}^{-1}U^{(1)}+a_{2,\epsilon}^{-1}U^{(2)}$ is the generator of $T^{\epsilon}$.
   
   Integrate with respect to the $T^{\epsilon}$-invariant measure $\nu_{\epsilon}$. Invariance gives $\int_XX_{\epsilon}F_jd\nu_{\epsilon}=0$, while
   \[
   \int_Xa^{-1}_{1,\epsilon}d\nu_{\epsilon}=\int_{X_2}a_{2,\epsilon}dm_2=1,\qquad \int_Xa_{2,\epsilon}^{-1}d\nu_{\epsilon}=\int_{X_1}a_{1,\epsilon}dm_1=1.
   \]
Hence $(c_1,c_2)=(1,1)$.

The first coordinate of \eqref{eq:diffFResFlow} becomes
\[
a_{1,\epsilon}^{-1}U^{(1)}F_1+a_{2,\epsilon}^{-1}U^{(2)}F_1=1-a_{1,\epsilon}^{-1}=\frac{\epsilon f_1}{a_{1,\epsilon}}.
\]
Multiplying by $a_{1,\epsilon}a_{2,\epsilon}$ and integrating over $X_2$ gives
\[
\mathcal{U}_1\overline{F}_1=\epsilon f_1,\qquad \overline{F}_1(x_1)=\int_{X_2}a_{2,\epsilon}(x_2)F_1(x_1,x_2)dm_2(x_2).
\]
Here the term containing $U^{(2)}F_1$ integrates to zero because $m_2$ is $h_t^{(2)}$-invariant, while $a_{1,\epsilon}$ is independent of $x_2$. Since $\overline{F}_1$ is smooth and $\epsilon\neq0$, the identity $\mathcal{U}_1(\epsilon^{-1}\overline{F}_1)=f_1$ contradicts Lemma~\ref{lem:choicesClockFunctions}. A full action conjugacy would restrict to the diagonal line and contradict this conclusion.
\end{proof}

\begin{proof}[Proof of Theorem~\ref{thm:sharpnessAbsCptRigid}]
    Lemma~\ref{lem:choicesClockFunctions} supplies the functions $f_i$,
    and Lemma~\ref{lem:productRepara} constructs the arbitrarily small smooth time changes. Proposition~\ref{prop:measureDrift} gives the invariant measure, mixing, full support, and the drift. Lemma~\ref{lem:productCptFra} and Proposition~\ref{prop:endptEstPhyCoc} give the compact frame and the closed-boundary estimates. Proposition~\ref{prop:nonSmoothAlgDiag} gives both nonconjugacy conclusions.

    For both frame directions,
    \[
    d_i=2,\qquad \delta_i=1,\qquad \gamma_i=2,
    \]
    and hence
    \[
    d_i\delta_i=\gamma_i=2\qquad i=1,2.
    \]
    Thus the strict inequalities $d_i\delta_i>\gamma_i$ in pointwise compact-frame rigidity cannot in general be replaced by $d_i\delta_i\geq \gamma_i$.
\end{proof}

\section{Smooth time changes on compact \texorpdfstring{$\SL_3(\R)$}{SL3R} quotients}\label{sec:smoothTCcptSL3}

In this section, we prove Theorem~\ref{thm:smoothConjugacyCri} and Theorem~\ref{thm:NonConjTCGivenQuo}. Throughout this section and Section~\ref{sec:cptRtAsyOpExp}, let $G=\SL_3(\R)$ and $X=G/\Gamma$, where $\Gamma<G$ is a cocompact lattice. Let $m$ be Haar probability on $X$. Write
\[
U=U_{2}=E_{12},\qquad Z=U_3=E_{13},\qquad K=K_2=E_{21}-E_{12}.
\]
As in the compact setting of the introduction, we write parameter vectors as columns $p=(p_1,p_2)^{\mathrm{T}}$ and set $\iota(p)=p_1U+p_2Z$. Thus $\phi_px=\exp(\iota(p))x$, and $e_1=(1,0)^{\mathrm{T}}$, $e_2=(0,1)^{\mathrm{T}}$ form the standard basis of the parameter space. Vector norms on $\R^2$ are Euclidean, and matrix norms are the associated operator norms. Let $\rho$ be the reparametrization cocycle of a smooth time change of $\phi$, let $\tau$ be its parameter inverse, and $\alpha$ the induced action. 

\subsection{Parameter derivatives and the invariant probability}

Recall that 
\[
M(x)=D_p\rho(x,0)\qquad \text{and}\qquad A=\int_XM\,dm.
\]
The matrix $M(x)$ converts infinitesimal algebraic parameters to parameters of $\alpha$. For a Lie algebra matrix $Y$, we write
\[
(Yf)(x)=\left.\frac{d}{ds}\right|_{s=0}f(\exp(sY)x),
\]
with derivatives taken componentwise for vector and matrix valued functions.

Differentiating $\rho(x,p+he_j)-\rho(x,p)=\rho(\phi_px,he_j)$ in $h$ at zero gives
\begin{equation}\label{eq:deriCocycle}
    D_p\rho(x,p)=M(\phi_px).
\end{equation}
Equality of mixed parameter derivatives then gives
\begin{equation}\label{eq:mixedDeriEq}
    U(Me_2)=Z(Me_1).
\end{equation}
Thus each row $(f_1,f_2)$ of $M$ is closed along the leaves: $Uf_2=Zf_1$. Since $\rho(x,0)=0$, integration along a straight parameter segment and Haar invariance give
\begin{equation}\label{eq:inteMrho}
    \rho(x,p)=\int_0^1M(\phi_{rp}x)p\,dr,\qquad \int_X\rho(x,p)dm(x)=Ap.
\end{equation}

Recall that the algebraic flows generated by $U$ and $Z$ are ergodic with respect to the Haar measure by Moore's ergodic theorem.  

\begin{Lemma}\label{lem:averagingClosedRows}
    Suppose $f=(f_1,f_2)$ and $g=(g_1,g_2)$ are smooth real rows with $Uf_2=Zf_1$ and $Ug_2=Zg_1$. Then
    \begin{equation*}
        \int_X(f_1g_2-f_2g_1)dm=\left(\int_Xf_1dm\right)\left(\int_Xg_2dm\right)-\left(\int_Xf_2dm\right)\left(\int_Xg_1dm\right).
    \end{equation*}
    Consequently, if $R:X\to\operatorname{Mat}_{2\times 2}(\R)$ is smooth and its rows are closed, then
    \begin{equation*}
        \int_X R \adj M \,dm=\left(\int_XRdm\right)\adj A,\qquad \int_X\det M dm=\det A.
    \end{equation*}
    Here $\adj Q$ is the adjugate matrix of $Q$.
\end{Lemma}
\begin{proof}
    Write $u_t=\exp(tU)$ and define
    \[
    I(t)=\int_X\big(f_1(u_tx)g_2(x)-f_2(u_tx)g_1(x)\big)dm(x).
    \]
    Integration by parts and commutativity of $U$ and $Z$ imply that 
    \[
    \begin{aligned}
        I'(t)=&-\int_X(f_1\circ u_t)Ug_2dm-\int_X((Zf_1)\circ u_t)g_1dm\\
        =&-\int_X(f_1\circ u_t)Zg_1dm+\int_X(f_1\circ u_t)Zg_1dm=0,
    \end{aligned}
    \]
    where we also used the fact that the rows are closed. 
    
    The mean ergodic theorem and ergodicity of the $U$-flow give
    \[
    \frac{1}{T}\int_0^Tf_i\circ u_tdt\longrightarrow\int_Xf_i\,dm
    \text{ in $L^2(m)$.}
    \]
    Since $g_j\in L^2(m)$, Fubini's theorem and Cauchy-Schwarz give
    \[
    \begin{aligned}
        \frac{1}{T}\int_0^T\int_X(f_i\circ u_t)g_j\,dm\,dt&=\int_X\left(\frac{1}{T}\int_0^Tf_i\circ u_t\,dt\right)g_j\,dm\longrightarrow\left(\int_Xf_i\,dm\right)\left(\int_Xg_jdm\right).
    \end{aligned}
    \]
    Averaging $I(t)=I(0)$ and applying this limit to the pairs $(i,j)=(1,2)$ and $(2,1)$ proves the first identity.

    Each entry of $R\adj M$ is, up to sign, the determinant formed from one row of $R$ and one row of $M$. Applying the scalar identity to these rows proves the matrix identity. Applying it to the two rows of $M$ gives $\int_X\det M\,dm=\det A$.
\end{proof}

\begin{Proposition}\label{prop:compactMeasure}
    The matrix $A$ is invertible, and $d\mu=\frac{\det M}{\det A}dm$ defines a smooth full support probability invariant under $\alpha$. Moreover,
    \[
    \int_XM^{-1}d\mu=A^{-1}.
    \]
\end{Proposition}
\begin{proof}
    The connectedness of $X$ and invertibility of $M(x)$ imply that $\det M$ has constant nonzero sign. Lemma~\ref{lem:averagingClosedRows} therefore gives
    \[
    \det A=\int_X\det M\,dm\neq0.
    \]
    Hence $h=\det M/\det A$ is smooth and positive, with integral one. Compactness makes $h$ bounded above and below away from zero, so $\mu$ is a full support probability.

    Fix $\eta\in\R^2$. Differentiating $\rho(x,\tau(x,r\eta))=r\eta$ at $r=0$, using $\tau(x,0)=0$, gives
    \[
    \left.\frac{d}{dr}\right|_{r=0}\tau(x,r\eta)=M^{-1}(x)\eta.
    \]
    Since $\alpha_{r\eta}x=\phi_{\tau(x,r\eta)}x$, its generator $X_{\eta}$ satisfies
    \[
    (X_{\eta}f)(x)=(M^{-1}(x)\eta)_1(Uf)(x)+(M^{-1}(x)\eta)_2(Zf)(x),\qquad f\in C^{\infty}(X).
    \]
    
    Write $M=(m_{ij})$. Equation~\eqref{eq:mixedDeriEq} gives $Um_{12}=Zm_{11}$ and $Um_{22}=Zm_{21}$, and hence
    \[
    U\big((\adj M)_{1j}\big)+Z\big((\adj M)_{2j}\big)=0,\qquad j=1,2.
    \]
    Define the smooth scalar functions 
    \[
    b_j(x)=h(x)(M^{-1}(x)\eta)_{j},\qquad j=1,2.
    \]
    Since $h(x)M^{-1}(x)\eta=(\det A)^{-1}\adj M(x) \eta$, the preceding identities imply  
    $Ub_1+Zb_2=0$. For every smooth test function $f$,  integration by parts therefore gives
    \[
    \int_X X_{\eta}fd\mu=\int_X(b_1Uf+b_2Zf)dm=-\int_Xf(Ub_1+Zb_2)dm=0.
    \]
    Applying this identity to $f\circ \alpha_{r\eta}$ proves invariance under the flow, and hence under the full action.

    \medskip
    
    Finally, the adjugate is linear for $2\times 2$ matrices, so
    \[
    \int_XM^{-1}d\mu=\frac{1}{\det A}\int_X\adj Mdm=\frac{\adj A}{\det A}=A^{-1}.
    \]
\end{proof}

Since the density $h=d\mu/dm$ is bounded above and bounded away from zero, the $L^2(m)$ and $L^2(\mu)$ norms are comparable. We will also use 
\begin{equation}\label{eq:rotationNorm}
    \norm{f\circ k_s}_{L^2(\mu)}\leq B_{\mu}\norm{f}_{L^2(\mu)},\qquad k_s=\exp(sK),\quad s\in\R,
\end{equation}
for scalar or vector valued $f$, with $B_{\mu}=(\max_Xh/\min_Xh)^{1/2}$. In fact, Haar invariance gives
\[
\norm{f\circ k_s}^2_{L^2(\mu)}=\int_X\abs{f(y)}^2h(k_{-s}y)dm(y)\leq B_{\mu}^2\norm{f}^2_{L^2(\mu)}.
\]

\subsection{Smooth conjugacy and bounded deviations}

Define the centered columns
\begin{equation}\label{eq:defGj}
    g_j=Me_j-Ae_j,\qquad j=1,2.
\end{equation}
They satisfy $Ug_2=Zg_1$ and have zero Haar mean. Consider the unitary representation
\[
\pi(g)f(x)=f(g^{-1}x),\qquad g\in G,\quad f\in \mathcal{H}=L_0^2(X,m;\C).
\]
It has no nonzero $G$-fixed vector: such a vector would be $U$-invariant, hence constant by ergodicity, and therefore zero,  since it has zero mean.

Let $\mathcal{H}^{\infty}$ denote the smooth vectors of $\pi$: a vector $f\in\mathcal{H}$ belongs to $\mathcal{H}^{\infty}$ if $g\mapsto\pi(g)f$ is smooth in the $L^2$ norm. For the Haar representation on compact $X$, these are exactly the $L^2$ classes having smooth mean zero representatives. This follows from the Lie algebra Sobolev characterization \cite{Wang2015Cohomological}*{\S 3.4} and Sobolev embedding, since a Lie algebra basis gives a frame of $TX$. We use these smooth representatives below.

With our derivative convention, $d\pi(U)=-U$. A $U$-invariant distribution is a continuous linear functional $\mathcal{D}$ on $\mathcal{H}^{\infty}$, equipped with its usual smooth vector topology, such that
\[
\mathcal{D}(U\varphi)=0\qquad \text{for every }\varphi\in\mathcal{H}^{\infty}.
\]
Wang's theorem \cite{Wang2015Cohomological}*{Theorem~2.1}, applied coordinatewise to $d\pi(U)b=-g_1$, shows that every $b\in L^2_0(X,m,\R^2)$ satisfying $Ub=g_1$ in the infinitesimal generator sense has a smooth representative. It also supplies such a solution whenever
\[
\mathcal{D}\big((g_1)_j\big)=0,\qquad j=1,2,
\]
for every $U$ invariant distribution $\mathcal{D}$; taking real parts gives a real valued solution.

The conditions of Theorem~\ref{thm:smoothConjugacyCri} are also equivalent to this vanishing condition. We prove this characterization together with the theorem.

\begin{proof}[Proof of Theorem~\ref{thm:smoothConjugacyCri}]
    \emph{\ref{item:thm:smoothConjugacyCri1}$\Longleftrightarrow$\ref{item:thm:smoothConjugacyCri2}.}
    The identities
    \[
    D_p\rho(x,p)=M(\phi_px),\qquad D_{\eta}\tau(x,\eta)=M^{-1}(\alpha_{\eta}x)
    \]
    give uniform Lipschitz constants $\norm{M}_{\infty}$ and $\norm{M^{-1}}_{\infty}$ for the parameter maps and their inverses. Consequently,
    \[
    \norm{\rho(x,te_1)-tAe_1}\leq\norm{M}_{\infty}\norm{te_1-\tau(x,tAe_1)},\quad \norm{te_1-\tau(x,tAe_1)}\leq\norm{M^{-1}}_{\infty}\norm{\rho(x,te_1)-tAe_1}.
    \]
    The equivalence follows from the comparison of the $L^2(m)$ and $L^2(\mu)$ norms.
    
    \medskip\noindent\emph{\ref{item:thm:smoothConjugacyCri1}$\Longleftrightarrow$\ref{item:thm:smoothConjugacyCri3}.}
    Assume \ref{item:thm:smoothConjugacyCri1}. On $L^2_0(X,m,\R^2)$, set
    \[
    \mathscr{U}_tf=f\circ\phi_{te_1},\qquad C_t=\rho(\cdot,te_1)-tAe_1.
    \]
    The operators $\mathscr{U}_t$ form a strongly continuous unitary group, and \eqref{eq:inteMrho} gives $\int_XC_t\,dm=0$. The cocycle identity and \eqref{eq:deriCocycle} give
    \[
    C_{s+t}=C_s+\mathscr{U}_sC_t,\qquad C_t=\int_0^t\mathscr{U}_ug_1\,du,
    \]
    where $g_1=Me_1-Ae_1$ is defined in \eqref{eq:defGj}.

    Let $K_0=\sup_{t}\norm{C_t}_{L^2(m)}$. The averages
    \[
    b_R=-\frac{1}{R}\int_0^RC_t\,dt
    \]
    have $L^2$ norm at most $K_0$. Choose $R_j\to\infty$ such that $b_{R_j}\rightharpoonup b$ in $L^2_0(X,m,\R^2)$. For each fixed $s\in\R$,
    \[
    \mathscr{U}_sb_R-b_R=C_s-\frac{1}{R}\left(\int_R^{R+s}C_t\,dt-\int_{0}^sC_t\,dt\right).
    \]
    The error has $L^2(m)$ norm at most $2\abs{s}K_0/R$, with oriented integrals when $s<0$. Passing to the same weak limit gives
    \[
    \mathscr{U}_sb-b=C_s\qquad \text{in }L^2(m),\quad s\in\R.
    \]
    Strong continuity then gives
    \[
    \frac{\mathscr{U}_sb-b}{s}=\frac{1}{s}\int_0^s\mathscr{U}_ug_1\, du\longrightarrow g_1\qquad \text{in }L^2(m),\quad s\to0.
    \]
    Thus $b$ belongs to the generator domain and $Ub=g_1$. Wang's regularity theorem \cite{Wang2015Cohomological}*{Theorem~2.1~(1)} gives a smooth representative of $b$, which we continue to denote by $b$. Thus $b\in C^{\infty}(X,\R^2)$ and $Ub=g_1$ holds everywhere.

    Since $Ub=g_1$ and $Ug_2=Zg_1$, commutativity gives
    \[
    U(Zb-g_2)=Zg_1-Ug_2=0.
    \]
    Thus the smooth vector valued function $Zb-g_2$ is $U$-invariant, so Haar ergodicity makes it constant. Haar invariance gives $\int_XZb\,dm=0$, while $\int_Xg_2\,dm=0$ by centering. Hence $Zb=g_2$. For $p=p_1e_1+p_2e_2$, the straight segment formula \eqref{eq:inteMrho} and \eqref{eq:defGj} give 
    \[
    \rho(x,p)-Ap=\int_0^1(M(\phi_{rp}x)-A)p\,dr=\int_0^1\big(p_1Ub+p_2Zb\big)(\phi_{rp}x)dr=b(\phi_px)-b(x).
    \]
    This proves \ref{item:thm:smoothConjugacyCri3}.

    \medskip

    Conversely, under \ref{item:thm:smoothConjugacyCri3}, Haar invariance gives
    \[
    \norm{C_t}_{L^2(m)}=\norm{b\circ \phi_{te_1}-b}_{L^2(m)}\leq2\norm{b}_{L^2(m)},
    \]
    which proves \ref{item:thm:smoothConjugacyCri1}.

    \medskip

    The invariant distribution characterization follows from the same argument. Vanishing of the distributions supplies an $L^2_0$ solution of $Ub=g_1$ by Wang's existence theorem \cite{Wang2015Cohomological}*{Theorem~2.1~(2)}. Regularity and the preceding compatibility argument then give \ref{item:thm:smoothConjugacyCri3}. Conversely, subtracting the Haar mean from a smooth transfer gives
    \[
    \mathcal{D}\big((g_1)_j\big)=\mathcal{D}(Ub_j)=0,\qquad j=1,2,
    \]
    for every $U$-invariant distribution $\mathcal{D}$.
    
    \medskip\noindent\emph{\ref{item:thm:smoothConjugacyCri3}$\Longleftrightarrow$\ref{item:thm:smoothConjugacyCri4}.}
    Assume \ref{item:thm:smoothConjugacyCri3}, and set
    \[
    F=A^{-1}b,\qquad \Psi(x)=\phi_{F(x)}x.
    \]
    The cohomology identity gives
    \[
    \Psi(\phi_px)=\phi_{p+F(\phi_px)}x=\phi_{A^{-1}\rho(x,p)}\Psi(x).
    \]
    Taking $p=\tau(x,\eta)$ gives $\Psi\alpha_{\eta}=\phi_{A^{-1}\eta}\Psi$.

    Define $\Xi(x)=\alpha_{-b(x)}x$. The intertwining identity gives
    \[
    \Psi(\Xi(x))=\phi_{-F(x)}\Psi(x)=x.
    \]
    For the other composition, set $y=\Psi(x)$. Item~\ref{item:thm:smoothConjugacyCri3} at $p=F(x)$ gives $\rho(x,F(x))=b(y)$, so $y=\alpha_{b(y)}x$ and $\Xi(y)=x$. Thus $\Psi$ is a smooth diffeomorphism with inverse $\Xi$, proving \ref{item:thm:smoothConjugacyCri4} and \eqref{eq:compactConjugacyMaps}.

    \medskip

    Conversely, suppose that \ref{item:thm:smoothConjugacyCri4} holds, so
    \[
    \Psi\alpha_{\eta}=\phi_{B^{-1}\eta}\Psi,\qquad B\in \GL_2(\R),
    \]
    for a smooth leaf preserving diffeomorphism $\Psi$. By Lemma~\ref{lem:smmothDisplacement}, write $\Psi(x)=\phi_{F(x)}x$ with $F$ smooth. Substitution of $\eta=\rho(x,p)$ shows that
    \[
    p+F(\phi_px)-F(x)-B^{-1}\rho(x,p)
    \]
    belongs to the discrete parameter stabilizer of $x$. For fixed $x$, this expression is continuous in $p$ and vanishes at zero, so it vanishes identically. Haar integration gives $p=B^{-1}Ap$, hence $B=A$. The same identity proves \ref{item:thm:smoothConjugacyCri3} with $b=AF$ and identifies the parameter matrix of every full conjugacy.

    \medskip\noindent\emph{\ref{item:thm:smoothConjugacyCri4}$\Longleftrightarrow$\ref{item:thm:smoothConjugacyCri5}.}
    A full conjugacy restricts to every parameter line and hence \ref{item:thm:smoothConjugacyCri4} implies \ref{item:thm:smoothConjugacyCri5}.

    \medskip

    Conversely, assume \ref{item:thm:smoothConjugacyCri5}, let $\Psi$ be a conjugacy as in \ref{item:thm:smoothConjugacyCri5}, with
    \[
    \Psi\alpha_{tv}\Psi^{-1}=\phi_{tq},\qquad v\neq0.
    \]
    Set 
    \[
    \widehat{\alpha}_{\eta}=\Psi\alpha_{\eta}\Psi^{-1},
    \]
    and let $Y_1,Y_2$ be its generators. They are linearly independent everywhere, since the generators of $\alpha$ are independent and $\Psi$ is a diffeomorphism. Thus $v_1Y_1+v_2Y_2$ never vanishes, and $q\neq0$. By Lemma~\ref{lem:smmothDisplacement}, $D\Psi$ preserves the leaf tangent planes. We may therefore write
    \[
    Y_j(x)=c_{1j}(x)U_x+c_{2j}(x)Z_x,\qquad j=1,2,
    \]
    with smooth coefficients.

    Commutativity of $\widehat{\alpha}$ gives
    \[
    \phi_{tq}\widehat{\alpha}_{re_j}=\widehat{\alpha}_{re_j}\phi_{tq}.
    \]
    Differentiating in $r$ at zero gives
    \[
    D\phi_{tq}(x)Y_j(x)=Y_j(\phi_{tq}x).
    \]
    Since $\phi_{tq}$ preserves $U$ and $Z$, comparison of coefficients gives $c_{ij}\circ\phi_{tq}=c_{ij}$. Moore's ergodic theorem and continuity make all $c_{ij}$ constant. Their matrix $C=(c_{ij})$ is invertible, so the generators determine the full action as
    \[
    \widehat{\alpha}_{\eta}=\phi_{C\eta},\qquad \eta\in\R^2.
    \]
    The same $\Psi$ is therefore a full conjugacy, and the preceding part gives $C=A^{-1}$.

    \medskip\noindent
    \emph{Uniqueness and the invariant measure.} Two smooth transfers differ by a $U$-invariant function, hence by a constant vector. For any full conjugacy $\Psi$, the probability $\Psi_*\mu$ has a smooth density relative to Haar measure and is invariant under the full algebraic action. Its density is therefore $U$-invariant; Haar ergodicity and normalization give $\Psi_*\mu=m$. This also applies to every conjugacy in \ref{item:thm:smoothConjugacyCri5}, since the same map conjugates the full action.
\end{proof}

\subsection{The cohomological obstruction}
\begin{proof}[Proof of Theorem~\ref{thm:NonConjTCGivenQuo}]
\emph{\ref{item:thm:NonCOnjTCGivenQuo1}$\Longrightarrow$\ref{item:thm:NonCOnjTCGivenQuo3}} 
    Let $(f,g)$ be a smooth pair as in \ref{item:thm:NonCOnjTCGivenQuo1}. Lemma~\ref{lem:intScaCocycle}, applied to the generators $U, Z$, gives a smooth scalar cocycle $\mathcal{C}$ with 
    \[
    \begin{aligned}
        \partial_r\mathcal{C}(x,(r,s))=f(\phi_{(r,s)}x),\qquad
        \partial_s\mathcal{C}(x,(r,s))=g(\phi_{(r,s)}x).
    \end{aligned}
    \]
    For sufficiently small $\epsilon\neq0$, set
    \[
    \rho_{\epsilon}(x,(r,s))=(r+\epsilon\mathcal{C}(x,(r,s)),s),\qquad a_{\epsilon}=1+\epsilon f\geq \frac{1}{2}.
    \]
    For fixed $x$ and $s$, the first component has $r$-derivative $a_{\epsilon}(\phi_{(r,s)}x)\geq\frac{1}{2}$ and tends to $\pm\infty$ as $r\to\pm\infty$. It is therefore a smooth diffeomorphism of $\R$. Since the second component is unchanged, each full parameter map is a diffeomorphism of $\R^2$. The implicit function theorem gives a jointly smooth inverse, including its dependence on $\epsilon$. Let $\alpha^{\epsilon}$ denote the induced action. Its infinitesimal and mean matrices are
    \[
    M_{\epsilon}=\begin{pmatrix}
        1+\epsilon f & \epsilon g\\
        0 & 1 
    \end{pmatrix},\qquad A_{\epsilon}=\id.
    \]
    A full conjugacy would give, by Theorem~\ref{thm:smoothConjugacyCri}, a smooth transfer whose first coordinate $b_1$ satisfies $Ub_1=\epsilon f$ and $Zb_1=\epsilon g$. Then $b_1/\epsilon$ would be a common primitive for $(f,g)$, a contradiction. The same theorem rules out a leaf preserving smooth algebraic conjugacy of any nonzero restriction.

    The first restriction $T_t^{\epsilon}=\alpha^{\epsilon}_{te_1}$ is generated by $a_{\epsilon}^{-1}U$ and preserves the smooth full support probability $a_{\epsilon}m$ by Proposition~\ref{prop:compactMeasure}. For a bounded $T^{\epsilon}$ invariant function $h$ and smooth $\varphi$, differentiate $\int_Xh(\varphi\circ T_t^{\epsilon})a_{\epsilon}dm=\int_Xh\varphi a_{\epsilon}\,dm$ at zero to obtain
    \[
    \int_XhU\varphi\,dm=0.
    \]
    Applying this identity to $\varphi\circ\exp(tU)$ shows that $\int_Xh(\varphi\circ\exp(tU))\,dm$ is independent of $t$. Haar invariance and density of smooth functions give
    \[
    h\circ \exp(tU)=h\qquad \text{$m$-a.e., for each $t\in\R$.}
    \]
    Moore's ergodic theorem makes $h$ constant $m$-a.e., and hence $a_{\epsilon}m$-a.e. Finally, as $\epsilon\to0$,  $\rho_{\epsilon}$, its inverse, and the induced action converge smoothly to their algebraic counterparts on compact parameter sets. This proves \ref{item:thm:NonCOnjTCGivenQuo3}.

\medskip\noindent
\emph{\ref{item:thm:NonCOnjTCGivenQuo3}$\Longrightarrow$\ref{item:thm:NonCOnjTCGivenQuo2}} A full conjugacy would restrict to a conjugacy of the first parameter line, contrary to \ref{item:thm:NonCOnjTCGivenQuo3}.

\medskip\noindent
\emph{\ref{item:thm:NonCOnjTCGivenQuo2}$\Longrightarrow$\ref{item:thm:NonCOnjTCGivenQuo1}} Take a time change as in \ref{item:thm:NonCOnjTCGivenQuo2}, and suppose \ref{item:thm:NonCOnjTCGivenQuo1} fails. Every smooth closed pair with zero Haar means then has a common smooth primitive. The two rows of $M-A$ are such pairs. Their primitives form a smooth vector $b$ satisfying $Ub=g_1$ and $Zb=g_2$. As in the preceding proof, integration along a straight parameter segment gives
\[
\rho(x,p)-Ap=b(\phi_px)-b(x).
\]
Theorem~\ref{thm:smoothConjugacyCri} gives a full conjugacy, contradicting \ref{item:thm:NonCOnjTCGivenQuo2}.
\end{proof}

\section{Compact root asymptotics and optimal exponents}\label{sec:cptRtAsyOpExp}

We prove Theorems~\ref{thm:CptRtAsy} and \ref{thm:OptCptRtExp}, retaining the compact setting of Section~\ref{sec:smoothTCcptSL3}. Fix $v\neq0$ and set
\[
T_t=\alpha_{tv},\qquad \xi=M^{-1}v,\qquad k_s=\exp(sK),\qquad D_s(x,t)=\tau(k_sx,tv)-\tau(x,tv).
\]
Differentiating $\rho(x,\tau(x,tv))=tv$ gives
\begin{equation}\label{eq:derivativeSetup}
    \partial_t\tau(x,tv)=\xi(T_tx),\qquad \tau(x,tv)=\int_0^t\xi(T_ux)du,\qquad \norm{\tau(\cdot,tv)}_{\infty}\leq L\abs{t},\qquad L=\norm{\xi}_{\infty}.
\end{equation}
Consequently, $\norm{D_s(\cdot,t)}_{\infty}\leq 2L\abs{t}$ for all $s,t$. Integrals with negative upper endpoints are oriented.

\subsection{The quadratic compact root estimate}
For $p=(p_1,p_2)^{\mathrm{T}}$ and $B=\iota(p)$, the first row structure gives $B^2=0$ and $BKB=p_1B$. Hence
\begin{equation}\label{eq:adexpBexpand}
    \Ad_{\exp(B)}K=K+[B,K]-p_1B.
\end{equation}
The quadratic term is a scalar multiple of $B$. Closedness converts its contribution into a derivative along the inverse parameter path, allowing integration by parts and giving quadratic rather than cubic growth.

\begin{Lemma}\label{lem:derivativeAlongOrbit}
    Let $\tau_1$ and $\xi_1$ be the first coordinates of $\tau$ and $\xi$, respectively. For every $x\in X$ and $t\in\R$,
    \begin{equation*}
        \begin{aligned}
            M(T_tx)K\tau(x,tv)=&\tau_1(x,tv)M(T_tx)\tau(x,tv)-\int_0^t(KM)(T_ux)\xi(T_ux)du\\
            &-\int_0^t\Big(\big([\iota(\tau(x,uv)),K]M\big)(T_ux)\xi(T_ux)+\xi_1(T_ux)M(T_ux)\tau(x,uv)+\tau_1(x,uv)v\Big)du.
        \end{aligned}
    \end{equation*}
    In the commutator term, the vector $\tau(x,uv)$ is held fixed when the Lie derivative acts on $M$.
\end{Lemma}
\begin{proof}
Fix $x,t$. Differentiating the inverse equation $\rho(k_sx,\tau(k_sx,tv))=tv$ in the base point gives
\[
M(T_tx)K\tau(x,tv)=-\left.\frac{d}{ds}\right|_{s=0}\rho(k_sx,\tau(x,tv)).
\]
To evaluate the RHS, keep the path $u\mapsto\tau(x,uv)$ and its velocity $\xi(T_ux)$ fixed. For a separate base point $y$, the parameter derivative formula gives
\[
\rho(y,\tau(x,tv))=\int_0^tM(\phi_{\tau(x,uv)}y)\xi(T_ux)du.
\]
Differentiate in $y=k_sx$ at $s=0$. Group conjugation and \eqref{eq:adexpBexpand} give
\[
\begin{aligned}
    M(T_tx)K\tau(x,tv)=&-\int_0^t\big((K+[\iota(\tau(x,uv)),K])M\big)(T_ux)\xi(T_ux)du\\
    &+\int_0^t\tau_1(x,uv)\big(\iota(\tau(x,uv))M\big)(T_ux)\xi(T_ux)du.
\end{aligned}
\]
Closedness \eqref{eq:mixedDeriEq} gives $(\iota(a)M)c=(\iota(c)M)a$ for constant $a,c\in\R^2$. Since $T$ has generator $\xi_1U+\xi_2Z$, 
\[
\frac{d}{du}M(T_ux)=\xi_1(T_ux)(UM)(T_ux)+\xi_2(T_ux)(ZM)(T_ux).
\]
Applying the closedness identity at $T_ux$, with $a=\tau(x,uv)$ and $c=\xi(T_ux)$, therefore gives
\[
\big(\iota(\tau(x,uv))M\big)(T_ux)\xi(T_ux)=\left(\frac{d}{du}M(T_ux)\right)\tau(x,uv).
\]
The derivative on the right acts only on $M$. Using $\partial_u\tau_1(x,uv)=\xi_1(T_ux)$ and  $M(T_ux)\partial_u\tau(x,uv)=v$, the product rule gives
\[
\frac{d}{du}\big(\tau_1(x,uv)M(T_ux)\tau(x,uv)\big)=\tau_1(x,uv)\left(\frac{d}{du}M(T_ux)\right)\tau(x,uv)+\xi_1(T_ux)M(T_ux)\tau(x,uv)+\tau_1(x,uv)v.
\]
Integrating from $0$ to $t$ and using $\tau(x,0)=0$ proves the lemma.
\end{proof}

\begin{proof}[Proof of Theorem~\ref{thm:CptRtAsy}~\ref{item:CptRtAsyitem1}]
Compactness bounds $M$ and $M^{-1}$, and the fixed direction derivatives $KM$, $[U,K]M$ and $[Z,K]M$. Since
\[
[\iota(p),K]M=p_1[U,K]M+p_2[Z,K]M,
\]
we have
\[
\norm{([\iota(p),K]M)(y)}\leq C\norm{p},\qquad p\in\R^2,\quad y\in X.
\]
Together with \eqref{eq:derivativeSetup}, this bounds each summand in the last integrand of Lemma~\ref{lem:derivativeAlongOrbit} by $C\abs{u}$, uniformly in $x$. Its integral is therefore $O(t^2)$ for both signs of $t$. The endpoint term is also $O(t^2)$, and the first integral is $O(\abs{t})$. Multiplying by the uniformly bounded inverse $M^{-1}(T_tx)$ gives
\[
\norm{K\tau(\cdot,tv)}_{\infty}\leq C(\abs{t}+t^2).
\]
Finally,
\[
D_s(x,t)=\int_0^s(K\tau(\cdot,tv))(k_rx)dr
\]
proves the increment bound for all real $s,t$.
\end{proof}

\subsection{The leading quadratic term}
For the rest of this section, suppose that $T$ is $\mu$-ergodic. Use Lemma~\ref{lem:normalization} to transport the action, lattice, and measure to coordinates in which 
\[
A^{-1}v=e_1.
\]
By Proposition~\ref{prop:compactMeasure},
\[
\int_X\xi\,d\mu=A^{-1}v=e_1.
\]
The ergodic theorem applied to \eqref{eq:derivativeSetup}, followed by dominated convergence using the uniform clock bound, gives
\[
\frac{\tau(x,tv)}{t}\longrightarrow e_1\quad \text{$\mu$-a.e. $x$,}\qquad \norm{\frac{\tau(\cdot,tv)}{t}-e_1}_{L^2(\mu)}\longrightarrow0\quad t\to+\infty.
\]
The next identity determines the average of the integral terms in Lemma~\ref{lem:derivativeAlongOrbit}.
\begin{Lemma}\label{lem:identityDerivativeAlongOrbit}
    For $H=[U,K]=E_{11}-E_{22}$, 
    \[
    \int_X\big((HM)\xi+\xi_1Me_1+v\big)d\mu=2v.
    \]
\end{Lemma}
\begin{proof}
    The relation 
    \[
    \exp(rH)\phi_p=\phi_{\diag(e^{2r},e^r)p}\exp(rH)
    \]
    shows that
    \[
    (x,p)\longmapsto\rho(\exp(rH)x,\diag(e^{2r},e^r)p)
    \]
    is a cocycle over $\phi$. Its parameter derivative at zero is $M(\exp(rH)x)\diag(e^{2r},e^r)$, whose rows are therefore closed. 
    
    Define the smooth matrix valued function 
    \[
    N(x)\coloneq\left.\frac{d}{dr}\right|_{r=0}\big(M(\exp(rH)x)\diag(e^{2r},e^r)\big)= HM(x)+M(x)\diag(2,1).
    \]
    Differentiating the closedness identities in $r$ shows that $N$ has closed rows. Haar invariance gives $\int_XHM\,dm=0$, and hence
    \[
    \int_XN\,dm=A\diag(2,1).
    \]
    Since $\xi=M^{-1}v$ and $d\mu=\frac{\det M}{\det A}dm$, Lemma~\ref{lem:averagingClosedRows} gives
    \[
    \int_XN\xi\, d\mu=\frac{1}{\det A}\left(\int_XN\adj M\,dm\right)v=\left(\int_XN\,dm\right)A^{-1}v=A\diag(2,1)A^{-1}v=2v.
    \]
    Finally, 
    \[
    N\xi=(HM)\xi+M\diag(2,1)\xi=(HM)\xi+v+\xi_1Me_1,
    \]
    which proves the lemma.
\end{proof}
\medskip

For $t>0$, define
\begin{equation}\label{eq:Etx}
    E_t(x)\coloneq\frac{K\tau(x,tv)}{t^2}-w(T_tx),
\end{equation}
where
\begin{equation}\label{eq:defwSection12}
   w=e_1-\xi=e_1-M^{-1}v. 
\end{equation}
\begin{proof}[Proof of \eqref{eq:CptRtAsyitem2eq1} and \eqref{eq:CptRtAsyitem2eq2} in Theorem~\ref{thm:CptRtAsy}]
For $f\in L^2(\mu,\R^2)$, set 
\[
B_t\coloneq t^{-1}\int_0^tf\circ T_u\,du\qquad \text{for $t>0$},\qquad \overline{f}=\int_Xf\,d\mu.
\]
The mean ergodic theorem gives $B_t\to\int_Xf\,d\mu$ in $L^2(\mu)$, and integration by parts gives
\[
\frac{1}{t^2}\int_0^tuf(T_u\cdot)\,du=B_t-\frac{1}{t^2}\int_0^tuB_u\,du.
\]
Consequently,
\begin{equation}\label{eq:weightedLimit}
    \norm{\frac{1}{t^2}\int_0^tuf\circ T_u\,du-\frac{1}{2}\overline{f}}_{L^2(\mu)}\leq\norm{B_t-\overline{f}}_{L^2(\mu)}+\int_0^1r\norm{B_{tr}-\overline{f}}_{L^2(\mu)}\,dr\longrightarrow0\quad \text{ as $t\to+\infty$.}
\end{equation}
The last integral tends to zero by dominated convergence, since $\norm{B_u-\overline{f}}_{L^2(\mu)}\leq\norm{f-\overline{f}}_{L^2(\mu)}$ for $u>0$.

\medskip

We estimate the three contributions in Lemma~\ref{lem:derivativeAlongOrbit} after division by $t^2$. For the endpoint term, the uniform clock bound gives
\[
\norm{\frac{\tau_1(\cdot,tv)M(T_t\cdot)\tau(\cdot,tv)}{t^2}-M(T_t\cdot)e_1}_{L^2(\mu)}\leq\norm{M}_{\infty}(L+1)\norm{\frac{\tau(\cdot,tv)}{t}-e_1}_{L^2(\mu)}\longrightarrow0.
\]
The first integral satisfies
\[
\norm{\frac{1}{t^2}\int_0^t(KM)(T_u\cdot)\xi(T_u\cdot)\,du}_{L^2(\mu)}\leq\frac{\norm{KM}_{\infty}L}{t}\longrightarrow0.
\]

In the remaining integral, replacing $\tau(x,uv)$ by $ue_1$ changes its integrand to
\[
u\big((HM)\xi+\xi_1Me_1+v\big)(T_ux),
\]
since $[\iota(ue_1),K]=uH$. With $a=\tau(x,uv)-ue_1$, the original integrand minus this replacement is exactly
\[
\big([\iota(a),K]M\big)(T_ux)\xi(T_ux)+\xi_1(T_ux)M(T_ux)a+a_1v.
\]
Its norm is at most $C\norm{a}$, uniformly in $x,u$: in fact, $[\iota(a),K]M=a_1HM+a_2[Z,K]M$, and all coefficients are bounded on $X$. Minkowski's inequality therefore bounds the normalized replacement error in $L^2(\mu)$ by
\[
\frac{C}{t^2}\int_0^t\norm{\tau(\cdot,uv)-ue_1}_{L^2(\mu)}\,du=\frac{C}{t^2}\int_0^tu\norm{\frac{\tau(\cdot,uv)}{u}-e_1}_{L^2(\mu)}\,du\longrightarrow0.
\]
Here we use the normalized clock convergence and its uniform bound. The weighted limit \eqref{eq:weightedLimit} and Lemma~\ref{lem:identityDerivativeAlongOrbit} give
\[
\frac{1}{t^2}\int_0^tu\big((HM)\xi+\xi_1Me_1+v\big)(T_ux)\,du\longrightarrow v\qquad \text{in $L^2(\mu)$.}
\]
Thus the remaining integral, divided by $t^2$, converges to $v$. Combining the three contributions in the clock identity gives
\[
\norm{\frac{M(T_t\cdot)K\tau(\cdot,tv)}{t^2}-M(T_t\cdot)e_1+v}_{L^2(\mu)}\longrightarrow0.
\]
Multiplication by the uniformly bounded matrix $M^{-1}(T_t\cdot)$ gives
\[
\norm{E_t}_{L^2(\mu)}=\norm{\frac{K\tau(\cdot,tv)}{t^2}-w\circ T_t}_{L^2(\mu)}\longrightarrow0,
\]
proving \eqref{eq:CptRtAsyitem2eq1}. Finally, $T$-invariance of $\mu$ gives $\norm{w\circ T_t}_{L^2(\mu)}=\norm{w}_{L^2(\mu)}$. The reverse triangle inequality therefore gives
\[
\abs{\frac{\norm{K\tau(\cdot,tv)}_{L^2(\mu)}}{t^2}-\norm{w}_{L^2(\mu)}}=\abs{\norm{\frac{K\tau(\cdot,tv)}{t^2}}_{L^2(\mu)}-\norm{w\circ T_t}_{L^2(\mu)}}\leq \norm{E_t}_{L^2(\mu)}\longrightarrow0,
\]
which proves \eqref{eq:CptRtAsyitem2eq2}.
\end{proof}

\subsection{The two scales of the compact root increment}
Lemma~\ref{lem:exactRelId} gives the exact endpoint identity
\begin{equation}\label{eq:endPtIdKs}
    T_tk_sx=\exp(\iota(D_s(x,t)))\exp(s\Ad_{\exp(\iota(\tau(x,tv)))}K)T_tx.
\end{equation}
Let $d_X$ be the quotient Riemannian distance. By \eqref{eq:derivativeSetup}, \eqref{eq:adexpBexpand}, and Theorem~\ref{thm:CptRtAsy}~\ref{item:CptRtAsyitem1}, the Lie algebra elements in the two exponential factors have norm at most $C\abs{s}t^2$ for $t\geq 1$. For $Y\in\mathfrak{g}$, compactness of $X$ bounds the speed of $r\mapsto\exp(rY)z$, $0\leq r\leq 1$, by $C_X\norm{Y}$, uniformly in $z$. Hence
\[
d_X(\exp(Y)z,z)\leq C_X\norm{Y}.
\]
Applying this estimate successively to the two factors in the endpoint identity and using the triangle inequality gives 
\begin{equation}\label{eq:disQuaControl}
    \sup_{x\in X}d_X(T_tk_sx,T_tx)\leq C\abs{s}t^2,\qquad t\geq 1,\quad \abs{s}t^2\leq1.
\end{equation}
\begin{proof}[Proof of \eqref{eq:CptRtAsyitem2eq3} of Theorem~\ref{thm:CptRtAsy}]
    For $t\geq 1$ and $0<\abs{s}t^2\leq 1$, the fundamental theorem of calculus and the substitution $\sigma=rs$ give
    \[
    D_s(x,t)=\int_0^s(K\tau)(k_{\sigma}x,tv)\,d\sigma=s\int_0^1(K\tau)(k_{rs}x,tv)\,dr.
    \]
    Hence \eqref{eq:Etx} gives
    \[
    \frac{D_s(x,t)}{st^2}=\int_0^1(E_t(k_{rs}x)+w(T_tk_{rs}x))dr.
    \]
    Minkowski's inequality and the composition estimate \eqref{eq:rotationNorm} control the first term; Lipschitz continuity of $w$ and \eqref{eq:disQuaControl} control the second. Thus
    \[
    \norm{\frac{D_{s}(\cdot,t)}{st^2}-w\circ T_t}_{L^2(\mu)}\leq B_{\mu}\norm{E_t}_{L^2(\mu)}+C\abs{s}t^2. 
    \]
    Taking $s=s(t)\neq0$ with $\abs{s(t)}t^2\to0$ proves the relative expansion for either sign of $s(t)$. Since $\mu$ is $T$-invariant, $\norm{w\circ T_t}_{L^2(\mu)}=\norm{w}_{L^2(\mu)}$. The reverse triangle inequality therefore gives
    \[
    \frac{\norm{D_{s(t)}(\cdot,t)}_{L^2(\mu)}}{\abs{s(t)}t^2}\longrightarrow\norm{w}_{L^2(\mu)}\qquad t\to+\infty.
    \]
\end{proof}
\medskip

At the critical scale, differentiating the inverse parameter gives
\begin{equation}\label{eq:sigmaTauInt}
    \sigma e_1+\tau(y,-\sigma v)=\int_0^{\sigma}w(T_{-u}y)du.
\end{equation}
Besides Theorem~\ref{thm:CptRtAsy}~\ref{item:CptRtAsyitem3}, we also prove the trajectory comparison:
\begin{equation}\label{eq:trajectoryComparisonDx}
    \norm{\sup_{\abs{\sigma}\leq S}\bar{d}_X(T_tk_{\sigma/t^2}\cdot,T_{t-\sigma}\cdot)}_{L^2(\mu)}\longrightarrow0,\qquad \bar{d}_X=\min(1,d_X),
\end{equation}
for every $S>0$ as $t\to+\infty$.
\begin{proof}[Proof of Theorem~\ref{thm:CptRtAsy}~\ref{item:CptRtAsyitem3} and \eqref{eq:trajectoryComparisonDx}] 
Fix $S>0$ and let 
\[
d_t(\sigma,x)\coloneq D_{\sigma/t^2}(x,t).
\]
The quadratic bound (cf.~Theorem~\ref{thm:CptRtAsy}~\ref{item:CptRtAsyitem1}) gives $\norm{d_t(\sigma,x)}\leq C_S$ for $t\geq 1$ and $\abs{\sigma}\leq S$. Integrating \eqref{eq:Etx} in the rescaled parameter gives
\[
d_t(\sigma,x)=\int_0^{\sigma}\big(w(T_tk_{u/t^2}x)+E_t(k_{u/t^2}x)\big)du.
\]
The integral of the second term, uniformly in $\sigma$, is bounded by $\int_{-S}^S\norm{E_t(k_{u/t^2}x)}du$. By \eqref{eq:rotationNorm}, its $L^2$ norm is at most $2SB_{\mu}\norm{E_t}_{L^2(\mu)}$.

\medskip

For $s=\sigma/t^2$, equation~\eqref{eq:adexpBexpand} shows that the second exponent in \eqref{eq:endPtIdKs} is
\[
\sigma\left(\frac{K}{t^2}+\frac{[\iota(\tau(x,tv)),K]}{t^2}-\frac{\tau_1(x,tv)}{t}\iota\left(\frac{\tau(x,tv)}{t}\right)\right).
\]
The norm of its difference from $-\sigma U$ satisfies
\[
\norm{\frac{\sigma}{t^2}\Ad_{\exp(\iota(\tau(x,tv)))}K+\sigma U}\leq C_S\left(t^{-1}+\norm{\frac{\tau(x,tv)}{t}-e_1}\right),\qquad \abs{\sigma}\leq S.
\]
All Lie algebra elements in the exponential factors remain in bounded sets. The map
\[
(P,Q,y)\longmapsto\exp(P)\exp(Q)y
\]
therefore has uniformly bounded derivative in $Q$ on the relevant sets, since $X$ is compact. Keeping the first exponent fixed and comparing the second with $-\sigma U$ gives
\[
\sup_{\abs{\sigma}\leq S}d_X(T_tk_{\sigma/t^2}x,\phi_{d_t(\sigma,x)-\sigma e_1}T_tx)\leq C_S\left(t^{-1}+\norm{\frac{\tau(x,tv)}{t}-e_1}\right).
\]
Combining this with the preceding integral identity and the Lipschitz bound for $w$, we obtain
\begin{equation}\label{eq:dtwCompare}
   \sup_{\abs{\sigma}\leq S}\norm{d_t(\sigma,x)-\int_0^{\sigma}w(\phi_{d_t(u,x)-ue_1}T_tx)du}\leq r_t(x), 
\end{equation}
where
\[
r_t(x)=\int_{-S}^S\norm{E_t(k_{u/t^2}x)}du+C_S\left(t^{-1}+\norm{\frac{\tau(x,tv)}{t}-e_1}\right)
\]
satisfies $\norm{r_t}_{L^2(\mu)}\to0$.

\medskip

For fixed $t,x$, set $y=T_tx$, and define 
\[
z(\sigma)=\sigma e_1+\tau(y,-\sigma v).
\]
Since
\[
    z'(\sigma)=w(T_{-\sigma}y),\qquad \phi_{z(\sigma)-\sigma e_1}y=\phi_{\tau(y,-\sigma v)}y=T_{-\sigma}y,
\]
this function solves
\begin{equation}\label{eq:z'formula}
z'(\sigma)=w(\phi_{z(\sigma)-\sigma e_1}y),\qquad z(0)=0.
\end{equation}
The RHS is globally Lipschitz in the parameter variable, with a constant $C_0$ independent of $y,\sigma$. Indeed, for $p,a\in\R^2$,
\[
\left.\frac{d}{dh}\right|_{h=0}w(\phi_{p+ha-\sigma e_1}y)=(a_1Uw+a_2Zw)(\phi_{p-\sigma e_1}y),
\]
whose norm is at most $C_0\norm{a}$ by compactness.

Set $e(\sigma)=\norm{d_t(\sigma,x)-z(\sigma)}$. Subtracting the exact and approximate integral equations, cf.~\eqref{eq:dtwCompare} and~\eqref{eq:z'formula}, gives
\[
e(\sigma)\leq r_t(x)+C_0\int_0^{\sigma}e(u)\,du,\qquad 0\leq \sigma\leq S.
\]
For the Gronwall estimate, set
\[
F(\sigma)=r_t(x)+C_0\int_0^{\sigma}e(u)\,du.
\]
Then $e\leq F$, $F(0)=r_t(x)$, and $F'=C_0e\leq C_0F$. Hence
\[
\frac{d}{d\sigma}\big(e^{-C_0\sigma}F(\sigma)\big)\leq 0,\qquad e(\sigma)\leq F(\sigma)\leq e^{C_0\sigma}r_t(x).
\]
Applying the same argument to $\sigma\mapsto e(-\sigma)$ on $[0,S]$ gives
\[
\sup_{\abs{\sigma}\leq S}\norm{d_t(\sigma,x)-\sigma e_1-\tau(T_tx,-\sigma v)}\leq e^{C_0S}r_t(x).
\]
Taking $L^2(\mu)$ norms and using $\norm{r_t}_{L^2(\mu)}\to0$ proves Theorem~\ref{thm:CptRtAsy}~\ref{item:CptRtAsyitem3}.

\medskip

Finally, integrating the bounded generating fields of $\phi$ along a straight parameter segment gives
\[
d_X(\phi_py,\phi_qy)\leq C_X\norm{p-q},\qquad p,q\in\R^2,\quad y\in X.
\]
Since
\[
\phi_{\tau(T_tx,-\sigma v)}T_tx=T_{t-\sigma}x,
\]
the endpoint comparison and the preceding error bound give
\[
\sup_{\abs{\sigma}\leq S}d_X(T_tk_{\sigma/t^2}x,T_{t-\sigma}x)\leq C_S\left(t^{-1}+\norm{\frac{\tau(x,tv)}{t}-e_1}+r_t(x)\right).
\]
The RHS tends to zero in $L^2(\mu)$, proving \eqref{eq:trajectoryComparisonDx}.
\end{proof}
\medskip

Let $X_v=\xi_1U+\xi_2Z$ be the generator of $T$, acting componentwise on vector valued functions. Taylor's theorem applied to \eqref{eq:sigmaTauInt} gives, for every integer $N\geq 1$ and every $\sigma\in\R$,
\[
\norm{\sigma e_1+\tau(\cdot,-\sigma v)-\sum_{j=0}^{N-1}\frac{(-1)^j\sigma^{j+1}}{(j+1)!}X_v^jw}_{\infty}\leq\frac{\abs{\sigma}^{N+1}}{(N+1)!}\norm{X_v^Nw}_{\infty}.
\]
Thus the finite Taylor expansion at $\sigma=0$ begins with $\sigma w-\frac{1}{2}\sigma^2X_vw$.

The relative expansion proved above applies when the perturbation parameter varies with time: for every real function $s(t)\neq0$ with $t^2\abs{s(t)}\to0$,
\[
\norm{\frac{D_{s(t)}(\cdot,t)}{s(t)t^2}-w\circ T_t}_{L^2(\mu)}\longrightarrow0\qquad t\to+\infty.
\]
For fixed $s\neq0$, by contrast, the clock bound gives
\[
\norm{\frac{D_s(\cdot,t)}{st^2}}_{\infty}\leq\frac{2L}{\abs{s}t}\longrightarrow0\qquad t\to+\infty.
\]

\subsection{Exact algebraicity, conjugacy, and optimal exponents}
The quadratic coefficient vanishes precisely for linear time changes:
\begin{equation}\label{eq:wEquivMRhoTau}
    w=0\quad\Longleftrightarrow\quad M=A\Longleftrightarrow\quad\rho(x,p)=Ap\quad\Longleftrightarrow\quad\tau(x,tv)=te_1,
\end{equation}
where the cocycle identities hold for all $x\in X$, $p\in\R^2$, and $t\in\R$. In fact, $w=0$ in \eqref{eq:defwSection12} gives $Me_1=v$ and $U(Me_2)=Z(Me_1)=0$. Haar ergodicity and continuity make $M$ constant, hence equal to $A$. The radial formula and $A^{-1}v=e_1$ give the next two implications; differentiating the last identity at $t=0$ gives $w=0$.

Under the strict compact root hypothesis, Theorem~\ref{thm:oneRootSL3rigidity}, continuity, and full support of $\mu$ give $T_tx=\phi_{te_1}x$ for every $x\in X$ and $t\in\R$. Comparing generators gives
\[
\xi_1U+\xi_2Z=U.
\]
Since $U$ and $Z$ are linearly independent, $w=0$, and the full time change is linear.

By Theorem~\ref{thm:smoothConjugacyCri}, smooth leaf preserving conjugacy of the full action is equivalent to
\begin{equation}\label{eq:smoothConjugacyCri12.4-1}
    \sup_{\sigma\in\R}\norm{\sigma e_1+\tau(\cdot,-\sigma v)}_{L^2(\mu)}<\infty.
\end{equation}
In this case, let $b$ be a smooth transfer as in that theorem and set $F=A^{-1}b$. Substituting $p=\tau(x,tv)$ in the transfer identity and using $v=Ae_1$ gives
\begin{equation}\label{eq:smoothConjugacyCri12.4-2}
    \tau(x,tv)-te_1=F(x)-F(T_tx),\qquad x\in X,\quad t\in\R.
\end{equation}

\begin{proof}[Proof of Theorem~\ref{thm:OptCptRtExp}]
\emph{Lower bounds.} Since $M$ is nonconstant, \eqref{eq:wEquivMRhoTau} gives $w\neq0$; full support of $\mu$ gives $\norm{w}_{L^2(\mu)}>0$. Recall that 
\[
\sigma e_1+\tau(y,-\sigma v)=\int_0^{\sigma}w(T_{-u}y)\,du,
\]
hence we have
\[
\frac{\sigma e_1+\tau(\cdot,-\sigma v)}{\sigma}\to w\qquad \text{uniformly as $\sigma\to0$.}
\]
Choose a fixed small $\sigma_0\neq0$ such that $\norm{\sigma_0e_1+\tau(\cdot,-\sigma_0v)}_{L^2(\mu)}>0$. Theorem~\ref{thm:CptRtAsy}~\ref{item:CptRtAsyitem3} and invariance of $\mu$ give
\[
\norm{D_{\sigma_0/t^2}(\cdot,t)}_{L^2(\mu)}\longrightarrow\norm{\sigma_0e_1+\tau(\cdot,-\sigma_0v)}_{L^2(\mu)}>0.
\]
For all sufficiently large $t$, we have $0<\abs{\sigma_0}/t^2\leq s_0$. Hence
\[
\mathcal{N}_{\delta}(t)\geq \frac{\norm{D_{\sigma_0/t^2}(\cdot,t)}_{L^2(\mu)}}{(\abs{\sigma_0}/t^2)^{\delta}}\geq c_{\delta}t^{2\delta}
\]
for some $c_{\delta}>0$. Any bound $\mathcal{N}_{\delta}(t)\leq C(1+t)^{\gamma}$ therefore requires $\gamma\geq2\delta$, so $\gamma_*(\delta)\geq2\delta$.

\medskip\noindent
\emph{General upper bounds.} For $0<\delta\leq 1$, interpolate the estimates $\norm{D_s(\cdot,t)}_{\infty}\leq C\abs{s}(1+t)^2$ and $\norm{D_s(\cdot,t)}_{\infty}\leq C(1+t)$:
\[
\norm{D_s(\cdot,t)}_{\infty}\leq C_{\delta}\abs{s}^{\delta}(1+t)^{1+\delta}.
\]
This proves $\gamma_*(\delta)\leq 1+\delta$ and gives attainment at $(\delta,\gamma)=(1,2)$. Now let $\delta>1$ and suppose, for a contradiction, that there are finite $C,\gamma\geq0$ such that
\[
\norm{D_s(\cdot,t)}_{L^2(\mu)}\leq C\abs{s}^{\delta}(1+t)^{\gamma},\qquad 0<\abs{s}\leq s_0,\quad t\geq0.
\]
For each fixed $t$, smoothness and compactness give
\[
\frac{D_s(\cdot,t)}{s}\longrightarrow K\tau(\cdot,tv)\qquad \text{uniformly as $s\to0$},
\]
whereas the assumed estimate gives
\[
\norm{\frac{D_s(\cdot,t)}{s}}_{L^2(\mu)}\leq C\abs{s}^{\delta-1}(1+t)^{\gamma}\longrightarrow0.
\]
Hence $\norm{K\tau(\cdot,tv)}_{L^2(\mu)}=0$ for every $t\geq0$, contradicting \eqref{eq:CptRtAsyitem2eq2}, since $\norm{w}_{L^2(\mu)}>0$. Thus $\gamma_*(\delta)=+\infty$ for $\delta>1$.

\medskip\noindent
\emph{The conjugate case.} Equation~\eqref{eq:smoothConjugacyCri12.4-2} gives $\norm{D_s(\cdot,t)}_{\infty}\leq 4\norm{F}_{\infty}$. Combining this with the quadratic estimate gives
\[
\norm{D_s(\cdot,t)}_{\infty}\leq C\min(1,\abs{s}(1+t)^2)\leq C\abs{s}^{\delta}(1+t)^{2\delta},\qquad 0<\delta\leq 1.
\]
The $L^2$ lower bound already proved establishes optimality in both norms, and the displayed estimate gives attainment.
\end{proof}

\appendix
\section{Smooth displacement along unipotent leaves}\label{sec:smoothDisplacement}
\begin{Lemma}\label{lem:smmothDisplacement}
    Let $G_0$ be a connected semisimple real linear group, $\Gamma_0<G_0$ a discrete subgroup, and $N<G_0$ a closed connected unipotent subgroup with Lie algebra $\mathfrak{n}$. Set $X_0=G_0/\Gamma_0$. Every smooth map $\Psi:X_0\to X_0$ satisfying $\Psi(x)\in Nx$ for every $x\in X_0$ has the form
    \[
    \Psi(x)=\exp(F(x))x
    \]
    for a unique smooth map $F:X_0\to\mathfrak{n}$.
\end{Lemma}
\begin{proof}
The case $N=\{e\}$ is immediate. Set $E=\mathfrak{g}_0/\mathfrak{n}$, where $\mathfrak{g}_0$ is the Lie algebra of $G_0$. The adjoint action induces 
\[
\pi:N\longrightarrow\operatorname{GL}(E),\qquad \pi(n)(Y+\mathfrak{n})=\Ad_nY+\mathfrak{n},
\]
with differential
\[
L(B)(Y+\mathfrak{n})=[B,Y]+\mathfrak{n}.
\]
The map $L$ is injective. Indeed, if $L(B)=0$ with $B\neq0$, then $[B,\mathfrak{g}_0]\subset\mathfrak{n}$. By the real Jacobson--Morozov theorem, there are $C,H\in\mathfrak{g}_0$ with $[B,C]=H$ and $[H,B]=2B$. Thus $H\in\mathfrak{n}$, so $\ad_H$ is nilpotent, whereas $\ad_HB=2B$ with $B\neq0$, a contradiction.

Every $L(B)$ is nilpotent. For $d=\dim E$, the finite logarithm
\[
\log M=\sum_{j=1}^{d-1}\frac{(-1)^{j+1}}{j}(M-\id)^j
\]
is polynomial on the closed set $\{M:(M-\id)^d=0\}$. Since $\exp:\mathfrak{n}\to N$ is a diffeomorphism and $\pi(\exp B)=\exp L(B)$, we have
\[
\pi(N)=\{M\in\operatorname{End}(E):(M-\id)^d=0,\ \log M\in L(\mathfrak{n})\}.
\]
Consequently, $\pi(N)$ is closed in $\operatorname{End}(E)$, and $M\mapsto L^{-1}(\log M)$ is smooth on $\pi(N)$. Here $L^{-1}$ is linear on $L(\mathfrak{n})$ and may be extended linearly to $\operatorname{End}(E)$.

On a sufficiently small open set $U\subset X_0$, choose smooth lifts $a,b:U\to G_0$ of $x$ and $\Psi(x)$. The closed sets
\[
E_{\gamma}=\{x\in U:b(x)\gamma a(x)^{-1}\in N\},\qquad \gamma\in \Gamma_0,
\]
cover $U$. Since $\Gamma_0$ is countable, Baire's theorem shows that the union of their interiors is dense in $U$. On each such interior, $v(x)=b(x)\gamma a(x)^{-1}$ is smooth and $\Psi(x)=v(x)x$. Thus $\Psi$ locally admits a smooth $N$-valued displacement on a dense open set $O\subset X_0$.

The fundamental vector fields give smooth isomorphisms:
\[
\vartheta_x:\mathfrak{g}_0\longrightarrow T_xX_0,\qquad \vartheta_x(Y)=\left.\frac{d}{dt}\right|_{t=0}\exp(tY)x.
\]
Define
\[
\mathcal{J}(x)=\vartheta_{\Psi(x)}^{-1}\circ D\Psi_x\circ\vartheta_x.
\]
Where $\Psi(x)=v(x)x$ with $v$ smooth and $N$-valued, differentiation gives
\[
\mathcal{J}(x)Y-\Ad_{v(x)}Y\in\mathfrak{n},\qquad Y\in\mathfrak{g}_0.
\]
Hence $\mathcal{J}(x)\mathfrak{n}\subset\mathfrak{n}$ on $O$, and therefore everywhere by continuity. Its induced quotient map $\overline{\mathcal{J}}(x)\in\operatorname{End}(E)$ is smooth and equals $\pi(v(x))$ on $O$. Closedness of $\pi(N)$ gives $\overline{\mathcal{J}}(x)\in\pi(N)$ everywhere. Define
\[
F(x)=L^{-1}\big(\log\overline{\mathcal{J}}(x)\big).
\]
Then $F$ is smooth and $\Psi(x)=\exp(F(x))x$ on $O$, hence on $X_0$ by continuity. Any other smooth displacement $F'$ satisfies $\overline{\mathcal{J}}(x)=\pi(\exp(F'(x)))$, so the same formula gives $F'=F$.
\end{proof}

\bibliographystyle{plain} 
\bibliography{refs} 

\begin{bibdiv}
\begin{biblist}

\bib{afr22}{article}{
      author={Artigiani, M.},
      author={Flaminio, L.},
      author={Ravotti, D.},
       title={On rigidity properties of time-changes of unipotent flows},
        date={2024},
     journal={Ann. Sc. Norm. Super. Pisa Cl. Sci. (5), published online},
}

\bib{FlaminioForni03InvariantDistributions}{article}{
      author={Flaminio, Livio},
      author={Forni, Giovanni},
       title={Invariant distributions and time averages for horocycle flows},
        date={2003},
        ISSN={0012-7094,1547-7398},
     journal={Duke Math. J.},
      volume={119},
      number={3},
       pages={465\ndash 526},
  url={https://doi-org.libproxy1.nus.edu.sg/10.1215/S0012-7094-03-11932-8},
      review={\MR{2003124}},
}

\bib{forni2012time}{article}{
      author={Forni, Giovanni},
      author={Ulcigrai, Corinna},
       title={Time-changes of horocycle flows},
        date={2012},
        ISSN={1930-5311,1930-532X},
     journal={J. Mod. Dyn.},
      volume={6},
      number={2},
       pages={251\ndash 273},
         url={https://doi.org/10.3934/jmd.2012.6.251},
      review={\MR{2968956}},
}

\bib{Jacobson1951Sl2}{article}{
      author={Jacobson, Nathan},
       title={Completely reducible {L}ie algebras of linear transformations},
        date={1951},
        ISSN={0002-9939,1088-6826},
     journal={Proc. Amer. Math. Soc.},
      volume={2},
       pages={105\ndash 113},
         url={https://doi-org.libproxy1.nus.edu.sg/10.2307/2032629},
      review={\MR{49882}},
}

\bib{kakutani1943131}{article}{
      author={Kakutani, Shizuo},
       title={Induced measure preserving transformations},
        date={1943},
        ISSN={0369-9846},
     journal={Proc. Imp. Acad. Tokyo},
      volume={19},
       pages={635\ndash 641},
         url={http://projecteuclid.org/euclid.pja/1195573248},
      review={\MR{14222}},
}

\bib{katok1977monotone}{article}{
      author={Katok, A.~B.},
       title={Monotone equivalence in ergodic theory},
        date={1977},
        ISSN={0373-2436},
     journal={Izv. Akad. Nauk SSSR Ser. Mat.},
      volume={41},
      number={1},
       pages={104\ndash 157, 231},
      review={\MR{0442195}},
}

\bib{kanigowski2021kakutani}{article}{
      author={Kanigowski, Adam},
      author={Vinhage, Kurt},
      author={Wei, Daren},
       title={Kakutani equivalence of unipotent flows},
        date={2021},
        ISSN={0012-7094},
     journal={Duke Math. J.},
      volume={170},
      number={7},
       pages={1517\ndash 1583},
         url={https://doi.org/10.1215/00127094-2020-0074},
      review={\MR{4255058}},
}

\bib{LindenstraussWei2025}{article}{
      author={Lindenstrauss, Elon},
      author={Wei, Daren},
       title={Time change rigidity for unipotent flows},
        date={2025},
     journal={arXiv preprint arXiv: 2502.08081},
}

\bib{ornstein1982equivalence}{book}{
      author={Ornstein, Donald~S.},
      author={Rudolph, Daniel~J.},
      author={Weiss, Benjamin},
       title={Equivalence of measure preserving transformations},
      series={Mem. Amer. Math. Soc.},
        date={1982},
      volume={37},
      number={262},
         url={https://doi.org/10.1090/memo/0262},
      review={\MR{653094}},
}

\bib{PughShub71}{article}{
      author={Pugh, Charles},
      author={Shub, Michael},
       title={Ergodic elements of ergodic actions},
        date={1971},
        ISSN={0010-437X,1570-5846},
     journal={Compositio Math.},
      volume={23},
       pages={115\ndash 122},
      review={\MR{283174}},
}

\bib{ratner1978horocycle}{article}{
      author={Ratner, Marina},
       title={Horocycle flows are loosely {B}ernoulli},
        date={1978},
        ISSN={0021-2172},
     journal={Israel J. Math.},
      volume={31},
      number={2},
       pages={122\ndash 132},
         url={https://doi.org/10.1007/BF02760543},
      review={\MR{516248}},
}

\bib{ratner1979cartesian}{article}{
      author={Ratner, Marina},
       title={The {C}artesian square of the horocycle flow is not loosely
  {B}ernoulli},
        date={1979},
        ISSN={0021-2172},
     journal={Israel J. Math.},
      volume={34},
      number={1-2},
       pages={72\ndash 96 (1980)},
         url={https://doi.org/10.1007/BF02761825},
      review={\MR{571396}},
}

\bib{ratner1981some}{article}{
      author={Ratner, Marina},
       title={Some invariants of {K}akutani equivalence},
        date={1981},
        ISSN={0021-2172},
     journal={Israel J. Math.},
      volume={38},
      number={3},
       pages={231\ndash 240},
         url={https://doi.org/10.1007/BF02760808},
      review={\MR{605381}},
}

\bib{ratner1986rigidity}{article}{
      author={Ratner, Marina},
       title={Rigidity of time changes for horocycle flows},
        date={1986},
        ISSN={0001-5962},
     journal={Acta Math.},
      volume={156},
      number={1-2},
       pages={1\ndash 32},
         url={https://doi.org/10.1007/BF02399199},
      review={\MR{822329}},
}

\bib{ratner1990measure}{article}{
      author={Ratner, Marina},
       title={On measure rigidity of unipotent subgroups of semisimple groups},
        date={1990},
        ISSN={0001-5962},
     journal={Acta Math.},
      volume={165},
      number={3-4},
       pages={229\ndash 309},
         url={https://doi.org/10.1007/BF02391906},
      review={\MR{1075042}},
}

\bib{ratner1991measure}{article}{
      author={Ratner, Marina},
       title={On {R}aghunathan's measure conjecture},
        date={1991},
        ISSN={0003-486X},
     journal={Ann. of Math. (2)},
      volume={134},
      number={3},
       pages={545\ndash 607},
         url={https://doi.org/10.2307/2944357},
      review={\MR{1135878}},
}

\bib{ratner1991topological}{article}{
      author={Ratner, Marina},
       title={Raghunathan's topological conjecture and distributions of
  unipotent flows},
        date={1991},
        ISSN={0012-7094},
     journal={Duke Math. J.},
      volume={63},
      number={1},
       pages={235\ndash 280},
         url={https://doi.org/10.1215/S0012-7094-91-06311-8},
      review={\MR{1106945}},
}

\bib{RavottiParabolic19}{article}{
      author={Ravotti, Davide},
       title={Parabolic perturbations of unipotent flows on compact quotients
  of {${\rm SL}(3,\mathbb{R})$}},
        date={2019},
        ISSN={0010-3616,1432-0916},
     journal={Comm. Math. Phys.},
      volume={371},
      number={1},
       pages={331\ndash 351},
         url={https://doi-org.libproxy1.nus.edu.sg/10.1007/s00220-019-03348-0},
      review={\MR{4015347}},
}

\bib{RogersMeanValue1955}{article}{
      author={Rogers, C.~A.},
       title={Mean values over the space of lattices},
        date={1955},
        ISSN={0001-5962,1871-2509},
     journal={Acta Math.},
      volume={94},
       pages={249\ndash 287},
         url={https://doi-org.libproxy1.nus.edu.sg/10.1007/BF02392493},
      review={\MR{75243}},
}

\bib{RudolphSmoothOrbitEquivalence79}{article}{
      author={Rudolph, Daniel~J.},
       title={Smooth orbit equivalence of ergodic {${\bf R}\sp{d}$}\ actions,
  {$d\geq 2$}},
        date={1979},
        ISSN={0002-9947,1088-6850},
     journal={Trans. Amer. Math. Soc.},
      volume={253},
       pages={291\ndash 302},
         url={https://doi-org.libproxy1.nus.edu.sg/10.2307/1998198},
      review={\MR{536948}},
}

\bib{SiegelMeanValue1945}{article}{
      author={Siegel, Carl~Ludwig},
       title={A mean value theorem in geometry of numbers},
        date={1945},
        ISSN={0003-486X},
     journal={Ann. of Math. (2)},
      volume={46},
       pages={340\ndash 347},
         url={https://doi-org.libproxy1.nus.edu.sg/10.2307/1969027},
      review={\MR{12093}},
}

\bib{tang2020new}{article}{
      author={Tang, Siyuan},
       title={New time-changes of unipotent flows on quotients of {L}orentz
  groups},
        date={2022},
        ISSN={1930-5311},
     journal={J. Mod. Dyn.},
      volume={18},
       pages={13\ndash 67},
         url={https://doi.org/10.3934/jmd.2022002},
      review={\MR{4385813}},
}

\bib{Tang23}{article}{
      author={Tang, Siyuan},
       title={Rigidity of joinings for time changes of unipotent flows on
  quotients of {L}orentz groups},
        date={2023},
        ISSN={0143-3857,1469-4417},
     journal={Ergodic Theory Dynam. Systems},
      volume={43},
      number={11},
       pages={3796\ndash 3860},
         url={https://doi-org.libproxy1.nus.edu.sg/10.1017/etds.2022.83},
      review={\MR{4651586}},
}

\bib{Wang2015Cohomological}{article}{
      author={Wang, Zhenqi~Jenny},
       title={Cohomological equation and cocycle rigidity of parabolic actions
  in some higher-rank {L}ie groups},
        date={2015},
        ISSN={1016-443X,1420-8970},
     journal={Geom. Funct. Anal.},
      volume={25},
      number={6},
       pages={1956\ndash 2020},
         url={https://doi-org.libproxy1.nus.edu.sg/10.1007/s00039-015-0351-6},
      review={\MR{3432162}},
}

\end{biblist}
\end{bibdiv}

\end{document}